\documentclass[11pt]{article}
\pdfoutput=1

\usepackage[bib]{preamble/packages}
\usepackage{preamble/commands}
\usepackage{preamble/formatting}

\usepackage{spectralsequences}
\usepackage{adjustbox}

\AtEveryBibitem{
    \clearfield{urlyear}
    \clearfield{urlmonth}
}

\newcommand*{\C}{\mathcal{C}}
\newcommand*{\D}{\mathcal{D}}

\newcommand*{\A}{\mathbf{A}}
\newcommand*{\E}{\mathbf{E}}

\newcommand*{\Z}{\mathbf{Z}}
\newcommand*{\Q}{\mathbf{Q}}
\newcommand*{\F}{\mathbf{F}}
\newcommand*{\G}{\mathbf{G}}

\newcommand*{\K}{\mathrm{K}}

\newcommand*{\X}{\mathfrak{X}}
\newcommand*{\Y}{\mathfrak{Y}}

\newcommand*{\Ga}{\Gamma}
\renewcommand*{\epsilon}{\eps}
\newcommand*{\Sph}{\S}

\DeclareMathOperator{\Gr}{Gr}

\newcommand*{\Tr}{\mathrm{Tr}}

\newcommand*{\M}{\mathfrak{M}}
\newcommand*{\Mbar}{\xoverline{\M}}

\newcommand*{\Mellbar}{\xoverline{\M}_\Ell}

\renewcommand*{\top}{\mathrm{top}}

\newcommand*{\kappabar}{\xoverline{\kappa}}

\newcommand*{\Span}{\mathrm{Span}}
\newcommand*{\all}{\mathrm{all}}

\newcommand*{\al}{\alpha}
\newcommand*{\be}{\beta}

\newcommand*{\OO}{\mathcal{O}}

\title{The descent spectral sequence\\
for topological modular forms}
\author{Christian Carrick\footnote{\href{mailto:carrick@math.uni-bonn.de}{\texttt{carrick@math.uni-bonn.de}}},\, Jack Morgan Davies\footnote{\href{mailto:davies@uni-wuppertal.de}{\texttt{davies@uni-wuppertal.de}}},\, and Sven van Nigtevecht\footnote{\href{mailto:nigtevec@math.uni-bonn.de}{\texttt{nigtevec@math.uni-bonn.de}}}}
\date{10th June 2026}

\begin{document}
\maketitle

\begin{abstract}
    We prove the Gap Theorem for the spectrum of topological modular forms $\mathrm{Tmf}$. This removes a longstanding circularity in the literature, thereby confirming the computation of $\pi_\ast \mathrm{tmf}$ from over two decades ago by Hopkins and Mahowald. Our approach is crucially a modern one, developing and refining many techniques in synthetic spectra.
\end{abstract}

\setcounter{tocdepth}{2}
\tableofcontents

\newpage

\section{Introduction}

Elliptic cohomology and topological modular forms ($\tmf$) play an essential role in modern stable homotopy theory. Aside from its connections to physics and number theory, $\tmf$ is a vital approximation to the stable homotopy groups of spheres $\pi_\ast \Sph$; see \cite{goerss_etal_resolution_K2_local} and \cite{wang_xu_61stem}.
The work of Isaksen--Wang--Xu \cite{isaksen_wang_xu_dimension_90}, the state of the art in computations of stable homotopy groups of spheres, uses $\tmf$ as a necessary tool.
A thorough understanding of the homotopy groups of $\tmf$ is therefore indispensable to our knowledge of the stable homotopy groups of spheres.

The computation of $\pi_\ast \tmf$ was first announced by Hopkins--Mahowald in \cite[Section~15]{tmf_book}, and details of this computation have appeared in many sources; see \cite{bauer_tmf}, \cite{bruner_rognes_ASS_tmf}, \cite{isaksen_etal_motivic_ANSS_tmf}, \cite{konter_Tmf}, \cite{rezk_tmf_lecture_notes}.
Shockingly, a complete proof has never appeared in the literature: all sources take either $\MU_*(\tmf)$ or $\uH_*(\tmf;\F_2)$ as input.
This circularity in the literature was pointed out by Meier \cite{meier_review_homology_tmf}; see \cref{ssec:history} below for a more detailed discussion.
Mathew \cite{mathew_homology_tmf} has shown that computing $\MU_*(\tmf)$ or $\uH_*(\tmf;\F_2)$ requires the Gap Theorem.
Giving an independent proof of the Gap Theorem would therefore fix the circularity.
This is what we do in this paper.

\begin{restatable}{theoremLetter}{thmgap}
    \label{thm:gap}
    The homotopy groups $\oppi_n \Tmf$ vanish for $-21<n<0$.
\end{restatable}

The spectrum $\tmf$ is by definition the connective cover of $\Tmf$, and the Gap Theorem allows one to deduce $\MU_*(\tmf)$ from $\MU_*(\Tmf)$.
The latter of these homologies follows directly from the definition of $\Tmf$ as the global sections of a spectral stack.

As a consequence of our techniques, we also obtain an independent computation of the homotopy groups of both $\Tmf$ and $\tmf$; see \cref{thm:homotopyoftmf}.

The difficulty in proving the Gap Theorem lies in running a spectral sequence converging to $\pi_*\Tmf$ that does not depend on the Adams (ASS) or Adams--Novikov (ANSS) spectral sequences for $\tmf$.
In work of Konter \cite{konter_Tmf} for example, the descent spectral sequence (DSS) for $\Tmf$ is computed, but only by assuming the ANSS for $\tmf$.
Running the DSS without this input is difficult for at least two reasons: the DSS is not \emph{a priori} an algebra over the ANSS for $\Sph$, and the DSS is much less sparse than the ANSS for $\tmf$.
The former makes it difficult to deduce differentials using low-degree information from $\pi_*\Sph$, and the latter means there are many possible differentials that need to be ruled out.
On the other hand, directly computing the ANSS for $\Tmf$ would be extremely difficult; we detail the problems in \cref{sssec:DSSvsANSS}.

However, in \cite{CDvN_part1} the authors constructed the synthetic $\E_\infty$-ring $\Smf$ of \emph{synthetic modular forms}, which implements the DSS for $\Tmf$ as an algebra over the ANSS for $\Sph$ in the strongest possible sense. Exploiting this structure is the key idea of this paper.

\begin{center}
    \emph{Using the $\E_\infty$-algebra $\Smf$ in $\MU$-synthetic spectra in tandem with modern techniques, we carefully compute the DSS, leading to the Gap Theorem.}
\end{center}

The map of synthetic $\E_\infty$-rings $\Sph\to\Smf$ allows us to prove detection statements for $\Tmf$, using low-degree information about $\pi_*\Sph$ and the naturality of Toda brackets along $\E_\infty$-ring maps. These detection statements force certain differentials in low degrees that we propagate to all further differentials using many techniques in synthetic spectra, such as generalized versions of the Leibniz rule, synthetic transfer maps, synthetic Toda brackets, and a version of Moss' theorem. 

The variants of the Leibniz rule allow one to make sophisticated differential stretching arguments, and we use this technique in particular to deduce a critical $d_7$ on $\Delta^4$, solving a problem posed by Isaksen--Kong--Li--Ruan--Zhu \cite[Problem~1.2]{isaksen_etal_motivic_ANSS_tmf}.
One notable use of synthetic Toda brackets is our proof of the $d_{23}$-differential killing $\kappabar^6$.
Previous sources have sketched a proof of this differential using a $6$-fold Toda bracket; using synthetic brackets, we give a much simpler argument using a $4$-fold bracket.

Below, we detail the history of this circularity and our approach to $\Tmf$ using the descent spectral sequence.
We then describe our techniques and highlight some key steps in the computation.


\subsection{History}\label{ssec:history}
The discovery of \emph{topological modular forms} $\tmf$ was announced in Hopkins' 1994 ICM address \cite{hopkins_icm95_tmf}. Many of the insights to its construction were years ahead of their time, and it took a while for sufficient technology to appear. The push towards these new techniques, such as model categories, obstruction theories for operadic structures, and eventually higher categories and derived algebraic geometry, launched several threads of homotopy theory still occupying mathematicians to this day. We would like to give a short summary of the history of $\tmf$, some of which appears in print and some that does not.

\subsubsection{Original construction}
Hopkins' original construction of $\tmf$ (completed at the prime $2$), was done in several chromatic layers. First, he considered the $G_{48}$-homotopy fixed-points spectrum of height $2$ Morava E-theory $\uE_2^{hG_{48}}$, where $G_{48}$ is a rank $48$ subgroup of the height $2$ extended Morava stabiliser group $\G_2$ acting on the height $2$ Morava E-theory spectrum $\uE_2$. It was only after computing with the associated homotopy fixed-point spectral sequence, that it became clear that there was strong connection between $\pi_\ast \uE_2^{hG_{48}}$ and modular forms. In hindsight, these fixed-points are precisely $L_{\K(2)}\tmf$, the localisation of $\tmf$ at the height $2$ Morava K-theory spectrum $\K(2)$.

Classically, one would use the Landweber exact functor theorem to show that the height~$n$ extended Morava stabiliser group $\G_n$ acts on the cohomology theory associated with $\uE_n$ for any $n$. This action on $\uE_n$ is not sufficient to construct the homotopy fixed-point spectrum above though, as it is confined to the homotopy category of spectra. Inspired by Robinson's $\E_1$-obstruction theory, and its applications to the Morava K-theories $\K(n)$ \cite{robinson_obstructiontheory} and variants of Morava E-theories by Baker \cite{baker_ainfty_etheories}, Hopkins--Miller refined this $\G_n$-action on $\uE_n$ to one of $\A_\infty = \E_1$-rings, with a proof written down by Rezk \cite{rezk_hopkinsmillertheorems}. Roughly speaking, one proves that the spaces of $\E_1$-endomorphisms of $\uE_n$ is discrete. The same techniques also show that the spaces of maps from $\uE_n^{\otimes d}$ to $\uE_n$ is discrete, which suggests that $\uE_n$ is a commutative monoid in $\E_1$-rings, so an $\E_\infty$-ring. 

This suspicion was confirmed by Goerss--Hopkins \cite{goerss_hopkins_obstructiontheory} using their $\E_\infty$-obstruction theory, and lead to the first Goerss--Hopkins--Miller theorem, recognising $\uE_n$ as an $\E_\infty$-ring with an $\E_\infty$-action by $\G_n$; a mild correction to the original proof appears in \cite[Section~7]{pstragowski_vankoughnett_obstruction_theory}. An alternative approach using derived algebraic geometry was given by Lurie in \cite{Lurie_ecsurvey,Lurie_EC2} as well as his generalisation appearing in \cite[Section~8.1]{behrens_lawson_TAF} and \cite{davies_luries_theorem}. Crucially, this $\E_\infty$-action of $\G_n$ on $\uE_n$ yields an $\E_\infty$-ring $\uE_2^{hG_{48}}$. The $\uE_2$-page of the associated homotopy fixed-point spectral sequence is hard to compute, as $G_{48}$ is a complicated group, but more so because the action of $\G_n$ on $\pi_\ast \uE_n$ is notoriously mysterious to compute with. Nevertheless, Hopkins--Mahowald computed this $G_{48}$-homotopy fixed-point spectral sequence. This particular computation has never appeared in print, although it can be recovered from this article; also see \cite{duan_etal_K2_tmf} for more on this connection.

The next step in constructing $\tmf$ was to understand its $\K(1)$-local piece: \emph{initial $\E_\infty$-$\K(1)$-local elliptic cohomology theory} $T$. This $\E_\infty$-ring is described in \cite[Section~16]{tmf_book}, originally written in 1998, and a full proof computing the homotopy groups of $T$ was given by Laures in \cite{laures_k1localtmf_2004}. Of course, this $T$ is simply $L_{\K(1)}\tmf$ in disguise. The universality of this $T$ produces a map $T \to L_{\K(1)} \uE_2^{hG_{48}}$, whose pullback along the localisation map $\uE_2^{hG_{48}} \to L_{\K(1)} \uE_2^{hG_{48}}$ gives a definition of $L_2 \tmf = \Tmf^\wedge_2$, the localisation of $\tmf$ at $\uE_2$ and the $2$-completion of \emph{projective topological modular forms} $\Tmf$. A similar construction of an integral $\Tmf$ using these chromatic techniques can be found in Behrens' chapter \cite[Section~12]{tmf_book}. Again, Lurie provides an alternative construction of $\TMF$ in \cite[Section~7]{Lurie_ecsurvey,Lurie_EC2} using derived algebraic geometry, with extensions to $\Tmf$ also given by the second named author with Linskens in \cite{Davies_Linskens_Tatecurve}. The unpublished Hopkins--Mahowald computation, together with a further computation of $\pi_\ast L_{\K(1)} \uE_2^{hG_{48}}$ and Laures' work on~$T$, could now be combined to give the homotopy groups of $L_2\tmf$ together with a version of the Gap Theorem (\Cref{thm:gap}).

\subsubsection{A series of computations}
In the absence of all of the details of the Hopkins--Mahowald computation of $\pi_\ast \uE_2^{hG_{48}}$, multiple authors have provided various new insights into the computations of $\pi_\ast \tmf$.

\begin{itemize}
    \item In \cite[Section~15]{tmf_book}, originally written in 1998, Hopkins--Mahowald sketch how $\pi_*\tmf$ can be computed by either an Adams spectral sequence (ASS) or an Adams--Novikov spectral sequence (ANSS). The $\uE_2$-pages of both spectral sequences are assumed here and most details are omitted.
    \item In \cite{rezk_tmf_lecture_notes}, originally written in 2001, Rezk assumes the $\MU$-homology of $\tmf$ and outlines how one can compute the homotopy groups of $\tmf$.
    \item In \cite{bauer_tmf}, Bauer assumes the $\MU$-homology of $\tmf$, and gives many subsequent details in the ANSS for $\tmf$.
    \item In \cite{konter_Tmf}, Konter uses Bauer's work to compute the DSS for $\Tmf$.
    \item Much later, in \cite{bruner_rognes_ASS_tmf}, Bruner--Rognes produce the most thorough computation of $\oppi_\ast\tmf$ using the ASS for $\tmf$. They assume the $\F_p$-homology of $\tmf$ as their starting point.
    \item Recently, in \cite{isaksen_etal_motivic_ANSS_tmf}, Isaksen--Kong--Li--Ruan--Zhu compute the ANSS for $\tmf$ in the context of $\mathbf{C}$-motivic homotopy theory, using the $\uE_2$-pages of the ASS and ANSS for $\tmf$ as input.
\end{itemize}

A logical flowchart connecting these computations appears below.

In all of these examples, there is a reliance on either the $\MU$- or $\F_p$-homology of $\tmf$, or on the $\uE_2$-page of the ASS or ANSS for $\tmf$, and no reason is given why one might know these \emph{a priori}.
Indeed, the definition of $\tmf$ as the connective cover of $\Tmf$ does not make it clear how to compute its homology from first principles.

Mathew \cite{mathew_homology_tmf} suggests a path to close this hole in the literature. He shows that if $\Tmf$ satisfies the Gap Theorem (\cref{thm:gap}), then one can deduce the $\MU$-homology of $\tmf$ from $\MU_\ast\Tmf$, which itself follows from the algebro-geometric definition of $\Tmf$. The $\F_p$-homology of $\tmf$ follows from its $\MU$-homology by a careful analysis of the map $\MU \to \F_p$. (Bear in mind, however, that the $\F_p$-homology of $\Tmf$ vanishes for all $p$.)

In \cite{mathew_homology_tmf}, Mathew implies that a proof of the Gap Theorem can be found in \cite{konter_Tmf}, but obtaining the Gap Theorem in this way would be circular, as it relies on Bauer's work which assumes a computation of the $\MU$-homology of $\tmf$. The first time this circularity is explicitly brought up in the literature is in Meier's review of \cite{mathew_homology_tmf}; see \cite{meier_review_homology_tmf}.

This left the literature in a precarious situation. Many groundbreaking papers in topology, such as \cite{wang_xu_61stem,isaksen_wang_xu_dimension_90}, rely on the homotopy groups of $\tmf$ as well as its ASS and ANSS. To fix this circularity, we prove the Gap Theorem using the \emph{descent spectral sequence} for $\Tmf$; see \cref{thm:dss}.
In addition, this yields the ANSS for $\tmf$; see \cref{thm:ansstmf}.

\subsection{The descent spectral sequence}

The $\E_\infty$-ring $\Tmf$ of \emph{\textbr{projective} topological modular forms} is defined as the global sections of the Goerss--Hopkins--Miller sheaf $\calO^\top$ on the moduli stack of generalised elliptic curves $\Mellbar$; see \cite[Section~12]{tmf_book} for a construction and \cite{davies_otop} for an axiomatic characterisation. This sheaf has the property that $\oppi_{2d} \calO^\top$ is isomorphic to the $d$-fold tensor power of the canonical line bundle $\omega$. Immediately from this property, one obtains a \emph{descent spectral sequence} (DSS) converging to the homotopy groups of $\Tmf$:
\begin{equation}\label{eq:dss}
\uE_2^{n,s} \cong \uH^s(\Mellbar, \,\omega^{\otimes (n+s)/2}) \implies \oppi_n \Tmf.
\end{equation}
The $\uE_2$-page is computed purely algebraically; see \cite{deligne_modularforms} for the computation where $s=0$ and $n\geq 0$, and \cite[Figures~10 and~25]{konter_Tmf} for the general pictures.

The main computation of this article is to determine all of the differentials in this spectral sequence.

\begin{restatable}{theoremLetter}{thmdss}
    \label{thm:dss}
    The DSS for $\Tmf$ of \eqref{eq:dss} takes the form depicted in \cref{etwoprime2,efiveprime2_part1,efiveprime2_part2,efiveprime2_part3,efiveprime2_part4} at the prime~$2$, depicted in \cref{ssprime3} at the prime~$3$, and collapses otherwise as detailed in \cref{rationalcomputation}.
\end{restatable}

The Gap Theorem (\cref{thm:gap}) immediately follows. In fact, we do not know how to prove the Gap Theorem without computing essentially the whole DSS.

Computing this spectral sequence leads us immediately to the homotopy groups of $\Tmf$, and hence also to the homotopy groups of its connective cover $\tmf$, which we call \emph{connective topological modular forms}.

\begin{restatable}{corollaryLetter}{htpytmf}
    \label{thm:homotopyoftmf}
    The homotopy groups of $\Tmf$, and hence also those of $\tmf = \tau_{\geq 0} \Tmf$, are computed; see \cref{rationalcomputation} away from $6$, \cref{ssprime3} at the prime~$3$, and \cref{efiveprime2_part1,efiveprime2_part2,efiveprime2_part3,efiveprime2_part4} at the prime~$2$.
\end{restatable}

Not only do we obtain the homotopy groups of $\tmf$, but also its ANSS.

\begin{restatable}{corollaryLetter}{ansstmf}
    \label{thm:ansstmf}
    There is an inclusion of the ANSS for $\tmf$ into the DSS for $\Tmf$ as a retract of spectral sequences. In particular, the ANSS for $\tmf$ is the region under the blue line of \cref{efiveprime2_part1,efiveprime2_part2,efiveprime2_part3,efiveprime2_part4} at the prime $2$ from the $\uE_5$-page, the region under the blue line of \cref{ssprime3} and $3$, and the connective part of \cref{rationalcomputation} away from $6$.
\end{restatable}

The ANSS and homotopy groups of periodic topological modular forms, denoted $\TMF$, also follow from the DSS for $\Tmf$.

\begin{restatable}{corollaryLetter}{sssTMF}
    \label{thm:sssTMF}
    The ANSS for $\TMF$ is obtained from the DSS for $\Tmf$ by inverting $\Delta^{24}$.
    Specifically, at the prime $2$ it is obtained by inverting $\Delta^{8}$ in \cref{etwoprime2,efiveprime2_part1,efiveprime2_part2,efiveprime2_part3,efiveprime2_part4}, at the prime $3$ by inverting $\Delta^{3}$ in \cref{ssprime3}, and away from $6$ by inverting $\Delta$ in \cref{rationalcomputation}.
\end{restatable}

\Cref{thm:gap,thm:dss} and their corollaries above are proven in \cref{sec:proofs_main_theorems}.

The following flowchart indicates the relationship between these results and the literature on topological modular forms.
\[\begin{tikzcd}
	{\text{DSS for Tmf (Th.~\ref{thm:dss})}} & {\text{Gap Theorem (Th.~\ref{thm:gap})}} \\
    { } &   {\text{MU-homology of tmf}}   &   {\F_p\text{-homology of tmf}}\\
	{\text{DSS for TMF}} & {{\uE_2\text{-page of ANSS for tmf}}} & {{\uE_2\text{-page of ASS for tmf}}} \\
	{\text{ANSS for TMF}} & {\text{ANSS for tmf}} & {\text{ASS for tmf}}
	\arrow[from=1-1, to=1-2]
	\arrow["\Delta^{-24}"', from=1-1, to=3-1]
	\arrow["{\text{\cite[Cor.~5.2]{mathew_homology_tmf}}}", from=1-2, to=2-2, swap]
	\arrow["{\text{\cite[Th.~5.13]{mathew_homology_tmf}}}", from=1-2, to=2-3, bend left=15]
	\arrow["{\text{\cite[Th.~C]{CDvN_part1}}}"', no head, from=3-1, to=4-1, "\cong"]
	\arrow["{\text{\cite[\textsection 5\&7]{bauer_tmf}}}", from=2-2, to=3-2, swap, shift right=1.5]
	\arrow["{\text{\cite[\textsection 6\&8]{bauer_tmf}}}", from=3-2, to=4-2, swap, shift right=1.5]
    \arrow["{\text{\cite{rezk_tmf_lecture_notes}}}", shift left=1.5, from=2-2, to=3-2]
    \arrow["{\text{\cite{rezk_tmf_lecture_notes}}}", shift left=1.5, from=3-2, to=4-2]
	\arrow["{\text{\cite[Part I]{bruner_rognes_ASS_tmf}}}", from=2-3, to=3-3]
    \arrow["{\text{\cite[Part II]{bruner_rognes_ASS_tmf}}}", from=3-3, to=4-3]
	\arrow["{{\Delta^{-24}}}", from=4-2, to=4-1]
	\arrow["{\text{\cite{isaksen_etal_motivic_ANSS_tmf}}}"{name=M}, from=3-3, to=4-2]
    \arrow[no head, from=3-2, to=M,shorten=0.3em]
    \arrow[from=1-1,to=4-2, bend right,"{\text{Cor.~\ref{thm:ansstmf}}}"]
    \arrow["{\text{\cite[§21]{rezk_tmf_lecture_notes}}}", from=2-2, to=2-3]
\end{tikzcd}\]

As the goal of this article is to find a geodesic route from the definition of $\Tmf$ to its homotopy groups, we do not compute all of the structure of $\tmf$.
For instance, we do not compute all hidden extensions in the ANSS for $\tmf$, nor do we give a comprehensive record of Toda brackets in $\tmf$.
We only determine the information needed to deduce \cref{thm:gap,thm:dss}; in particular, the tables of \cref{ssec:tables} only contain the information we need to establish these results.
Generalisations of the techniques to deduce hidden extensions are discussed in \cref{rmk:moreexts}.

A more thorough record of the multiplicative structures on $\tmf$ can be found in \cite{bruner_rognes_ASS_tmf} and \cite{isaksen_etal_motivic_ANSS_tmf}.
We find the techniques of \cite{isaksen_etal_motivic_ANSS_tmf}, comparing the ANSS and the ASS of $\tmf$, to be the more methodical and efficient path, once the Gap Theorem has been proven. In addition, we should mention that the multiplicative structure of $\pi_\ast\tmf$ at odd primes is discussed in \cite[Section~13]{bruner_rognes_ASS_tmf}, and that Marek \cite{marek_nu_tmf} has computed the bigraded homotopy groups of the $\F_2$-synthetic analogue of $\tmf$ based on \cite{bruner_rognes_ASS_tmf}.

\begin{remark}
    It may be possible to deduce the Gap Theorem from the homotopy groups of $L_{\K(2)}\tmf$ using an open-closed decomposition of $\Mellbar$ via the $j$-invariant.
    The homotopy groups of $L_{\K(2)}\tmf$ are essentially computed by \cite{duan_etal_K2_tmf} using techniques from equivariant homotopy such as restrictions, transfers, norms, and vanishing lines coming from Real bordism theory. However, they do not prove the Gap Theorem, and we do not pursue this approach in this paper.
    Our approach proceeds more from first principles using the DSS, and our methods apply much more broadly in contexts outside the homotopy fixed-point spectral sequence.
\end{remark}

\subsubsection{The DSS versus the ANSS}\label{sssec:DSSvsANSS}

We do not use the ANSS for $\Tmf$ in any of our computations, and this is a very deliberate choice.
While its $\uE_2$-page is readily computed using the definition of $\Tmf$ (see \cite[Proposition~5.1]{mathew_homology_tmf}), the resulting spectral sequence is incredibly hard to compute.
Part of the $\uE_2$-page is the same as that of the DSS, which we call the \emph{connective region}; this is the area under the line $5s \leq n+12$.
The ANSS differs outside of this region, as elements in negative stems are moved down by one filtration.
We refer to this region as the \emph{nonconnective region}, although it also lives in stems $n\geq 0$, but only in high filtration.
This filtration shift in the nonconnective region causes many of the techniques of this article to utterly fail for the ANSS of $\Tmf$.

\begin{itemize}
    \item The differentials in the DSS are propagated from explicit atomic differentials using the meta-arguments of \cref{ssec:meta}.
    These meta-arguments fail for the ANSS for $\Tmf$, as all elements outside of the connective region are $\Delta$-power torsion.
    \item The class $c_4$ is seen to be a $d_{11}$-cycle by arguing that the target of the potential $d_{11}$ is one that supports a $d_{11}$; see \cref{cor:nolinecrossingE11}.
    Not only do we not know how to justify this second nonzero $d_{11}$ due to the first problem above, but now $c_4$ could \emph{a priori} support a $d_{10}$ in the ANSS.
    The same issue holds for other classes such as $h_1\Delta$.
    \item One way of stating the essential difference between $\Smf$ and $\opnu \Tmf$ is that the former is an \emph{even} $\MU$-synthetic spectrum, and the latter is not.
    Concretely, this means that the DSS for $\Tmf$ satisfies a checkerboard phenomenon, and the ANSS for $\Tmf$ does not, and this causes the difficulties mentioned above; see \cref{rmk:checkmate} for more details.
\end{itemize}

We do not know how to work around these problems without appealing to the computation of the DSS.

As mentioned above, the advantage of the ANSS over the DSS for $\Tmf$ is that the former naturally receives a map from the ANSS for $\S$.
Using synthetic spectra, we are able to define a similar map for the DSS, so that the ANSS has no advantage over the DSS.

\subsection{Synthetic methods}

The first technical step necessary for our computation is to put the DSS on good footing. In particular, we would like it to receive a map from the ANSS for $\S$.
A map out of the ANSS for $\S$ can be constructed by moving to a setting where it has a universal property.
The $\infty$-category $\Syn_\MU$ of $\MU$-synthetic spectra defined by Pstr\k{a}gowski \cite{pstragowski_synthetic} provides such a setting.
It functions as a type of $\infty$-category of modified Adams--Novikov spectral sequences; see \cite[Section~9]{burklund_hahn_senger_manifolds} or \cite[Section~1]{christian_jack_synthetic_j}.
There is a functor $\nu \colon \Sp \to \Syn_\MU$ called the \emph{synthetic analogue functor}, which sends $X$ to the synthetic spectrum encoding the ANSS for $X$.
Importantly, this functor is not essentially surjective; synthetic spectra outside its essential image act as modified Adams--Novikov spectral sequences.
The $\infty$-category $\Syn_\MU$ is also naturally symmetric monoidal, with $\nu\S$ being the monoidal unit.
This means $\nu \S$ is initial in $\CAlg(\Syn_\MU)$, which gives the ANSS for~$\S$ the universal property we need.

To obtain the desired map, one has to implement the DSS for $\Tmf$ as an object in this $\infty$-category.
In previous work \cite{CDvN_part1}, we construct an $\E_\infty$-algebra in $\MU$-synthetic spectra that does exactly this.
It is constructed by imitating the definition of $\Tmf$ as global sections of a sheaf of spectra; accordingly, we call it \emph{synthetic modular forms}.

\begin{construction}[\cite{CDvN_part1}]
    \label{constr:recall_Smf_definition}
    Write $\Mellbar$ for the moduli stack of generalised elliptic curves. The Goerss--Hopkins--Miller sheaf of $\E_\infty$-rings $\calO^\top$ is a sheaf
    \[
        \calO^\top \colon (\mathrm{Aff}_{/\xoverline{\scriptstyle\M}_\Ell}^\et)^\op \to \CAlg(\Sp)
    \]
    on the small \'{e}tale site of affine schemes with an \'etale map to $\Mellbar$.
    The $\E_\infty$-ring $\Tmf$ is defined as the global sections of $\calO^\top$ over $\Mellbar$.
    By \cite[Proposition~2.8]{CDvN_part1}, the composition of $\calO^\top$ with $\nu \colon \Sp \to \Syn_\MU$ is an \'etale sheaf valued in $\Syn_\MU$.
    The synthetic spectrum of \defi{\textbr{projective} synthetic modular forms}, denoted by $\Smf$, is defined as the global sections of this sheaf.
    Since $\nu$ is lax symmetric monoidal, $\Smf$ is naturally an $\E_\infty$-algebra in $\Syn_\MU$.
\end{construction}

Implementing the DSS as a homotopy ring in synthetic spectra is enough to give a map of  multiplicative spectral sequences
\[\mathrm{ANSS}(\S)\to\mathrm{DSS}(\Tmf).\]
However, the $\E_\infty$-structure on $\Smf$ ensures that this map also respects all coherent multiplicative structure; see \cite[Theorem~B]{CDvN_part1}. We use this structure for example to show that the elements $\eta$, $\nu$, $\kappa$ and $\kappabar$ are detected along the $\E_\infty$-map $\nu\S\to\Smf$; see \cref{detectionsection}. This is crucial also to our use of tools such as the synthetic Leibniz rule and synthetic Toda brackets.

\begin{remark}
    We would like to emphasize that implementing the DSS as a synthetic $\E_\infty$-ring in this way is not simply a matter of taste. Both the $\E_\infty$-structure and the use of the $\infty$-category of synthetic spectra are crucial to almost all aspects of our computation. We do not know how one could carry out this computation without them. 
\end{remark}

\begin{remark}
    There is also a synthetic version of the periodic $\TMF$, denoted by $\SMF$, but this turns out to be equivalent to $\opnu\TMF$ as an $\E_\infty$-algebra; see \cite[Theorem~C]{CDvN_part1}.
    For the connective version $\tmf$, it is not clear how to give a purely synthetic definition that would help us carry out this computation.
    For instance, when phrased in terms of $\MU$-synthetic spectra, the $\E_\infty$-algebra $\mmf$ of \cite{mmf} is by definition given by $\opnu\tmf$.
    As such, the synthetic spectrum $\opnu \tmf$ does not give a new way to compute the ($\uE_2$-page of the) ANSS for $\tmf$.
    See also the discussion in \cite[Remark~4.9]{CDvN_part1}.
\end{remark}

\begin{remark}\label{rmk:checkmate}
    The synthetic spectrum $\Smf$ is not equivalent to $\opnu\Tmf$, stemming from the fact that the DSS and ANSS for $\Tmf$ are not isomorphic.
    The homological (a.k.a.\ canonical) t-structure on $\Syn_\MU$ allows one to see this: $\Smf$ is $(-1)$-connective but not connective, and there is an $\E_\infty$-map $\opnu\Tmf \to \Smf$ that exhibits $\opnu\Tmf$ as the connective cover of $\Smf$; see \cite[Corollary~4.15]{CDvN_part1}.
    Another essential difference between $\Smf$ and $\opnu\Tmf$ is that the former is an \emph{even} $\MU$-synthetic spectrum (in the sense of \cite[Section~5.2]{pstragowski_synthetic}), while the latter is not.
    In particular, after $p$-completion, $\Smf$ is a $\mathbf{C}$-motivic spectrum.
\end{remark}

\subsubsection{Truncated synthetic spectra}

The key to the computation of the DSS is not to apply our techniques to $\Smf$ directly, but rather to synthetic truncations of it.
Let $X$ be a synthetic spectrum.
For every $k\geq 1$, there is a synthetic spectrum $X/\tau^k$, and these fit into a tower
\[
    X \to \dotsb \to X/\tau^2 \to X/\tau.
\]
If $X$ is an $\E_\infty$-algebra, then this tower is naturally one of $\E_\infty$-algebras.
The bottom object $X/\tau$ encodes exactly the $\uE_2$-page of the underlying spectral sequence.
For larger $k$, the object $X/\tau^k$ can be thought of as encoding a designer spectral sequence that artificially terminates on the $\uE_{k+1}$-page, while still being a highly structured object.
In particular, Toda brackets in it make sense; these can be thought of as Toda brackets that are only ``temporarily'' defined, in the sense that they may involve elements that do not survive past a certain page. Similarly, we can speak of hidden extensions that are only ``temporarily'' defined.
In addition to these features, the homotopy group $\oppi_{n,s}X/\tau^k$ captures information about the $n$-stem, but does not see phenomena that occur in filtrations $s+k$ and above.
Applied to $X = \Smf$, these truncated objects therefore allow one to work with the DSS as if it did have a vanishing line.

Carrying this out requires a description of the homotopy groups of $\oppi_{*,*}X/\tau^k$ for $k\geq 1$.
The \emph{Omnibus Theorem} of \cref{ssec:omnibus} carefully expresses how the bigraded homotopy groups of a synthetic spectrum capture its underlying spectral sequence.
The version we describe is a modification and generalisation of the original theorem of the same name of Burklund--Hahn--Senger \cite[Theorem~9.19]{burklund_hahn_senger_manifolds}.
It is modified in that it describes synthetic spectra of the form $X/\tau^k$, and more general in that it applies to an arbitrary synthetic spectrum $X$, not just a synthetic analogue.
The notion of the underlying spectral sequence is made precise by the \emph{signature spectral sequence}; we recall this in \cref{ssec:filtered_spectra}, but this has appeared before in \cite[Definition~1.5]{christian_jack_synthetic_j} and \cite[Definition~1.10]{CDvN_part1}.

\subsubsection{Total differentials and the synthetic Leibniz rule}
A major concrete advantage of working in synthetic spectra is that the differentials in the spectral sequence that a synthetic spectrum $X$ implements may be understood via the boundary map $\delta_1^\infty$ in the cofiber sequence
\[
    \begin{tikzcd}
        \opSigma^{0,-1} X \ar[r,"\tau"] & X \ar[r] & X/\tau \ar[r,"\delta_1^\infty"] & \opSigma^{1,\, -2} X.
    \end{tikzcd}
\]
This boundary map is called the \emph{total differential}, and it has been used by many authors in both synthetic and motivic spectra; see \cite{chua_E3_ASS} and \cite{isaksen_etal_motivic_ANSS_tmf}, for example.

The total differential captures all differentials in the following way. If $x \in \oppi_{*,*}X/\tau = \uE_2^{*,*}$, then $d_2(x)$ can be computed as the mod $\tau$ reduction of $\delta_1^\infty(x)$.
If $\delta_1^\infty(x)$ is $\tau$-divisible, then this vanishes, in which case $d_3(x)$ can be computed as the mod $\tau$ reduction of any choice of $\tau$-division of $\delta_1^\infty(x)$, and so on.
There are also truncated versions $\delta_n^N$ coming from the cofiber sequences
\[
    \begin{tikzcd}
        \opSigma^{0,-n}X/\tau^{N-n} \ar[r,"\tau^n"] & X/\tau^N \ar[r] & X/\tau^n \ar[r,"\delta_n^N"] & \opSigma^{1,\, -n-1} X/\tau^{N-n},
    \end{tikzcd}
\]
which capture the differentials $d_{n+1},\dotsc,d_N$ in a similar manner.

As an immediate consequence, if $X$ is a homotopy ring in synthetic spectra, then $\delta_1^\infty$ is $\oppi_{*,*}X$-linear, and $\delta_n^N$ is $\oppi_{*,*}(X/\tau^N)$-linear. This leads to a direct correspondence between stretched differentials and hidden extensions, which has been studied also in \cite{chua_E3_ASS} and \cite{isaksen_etal_motivic_ANSS_tmf}. We give an example of this that appears in our computation.

\begin{example}[Linearity of the total differential]
In \cref{prop:d5_Delta}, we use information from the sphere to deduce a crucial differential $d_5(\Delta)=h_2g$. This results in a total differential $\delta_4^8(\Delta)=\nu\kappabar$, where $\Delta$ is the unique lift of $\Delta\in\oppi_{24,0}\Smf/\tau$ to $\oppi_{24,0}\Smf/\tau^4$; see \cref{lem:total_differential_14_Delta}. We import the hidden $2$-extension from $2\nu$ to $\eta^3$ in the ANSS for $\S$ to deduce the relation $4\nu=\tau^2\eta^3$ in $\oppi_{3,1}\Smf$. Using that $\delta_4^8$ is $\oppi_{*,*}\Smf$-linear, we have
\[\delta_4^8(4\Delta)=4\nu\kappabar=\tau^2\eta^3\kappabar,\]
which results in the ``stretched'' differential $d_7(4\Delta)=h_1^3g$ of \cref{prop:dsevens}.
\end{example}

The most useful aspect of the truncated total differential $\delta_n^N$ is a theorem of Burklund that describes a \emph{synthetic Leibniz rule}, which is a major improvement of the ordinary Leibniz rule in a multiplicative spectral sequence.
We give a proof of his result in \cref{ssec:synthetic_leibniz}.

In a multiplicative spectral sequence, the \emph{Leibniz rule} says that for $x,y\in \uE_{r}^{*,*}$, we have
\[
    d_r(xy) = d_r(x)\cdot y + (-1)^{\abs{x}}x\cdot d_r(y),
\]
where $\abs{x}$ refers to the stem of $x$. A limiting factor of this rule however is that it can only deduce nontrivial differentials of the \emph{same length} as the ones that we started with.
The synthetic Leibniz rule on the other hand can deduce a differential on $xy$ of length $\leq 2r-2$, even if $x$ and $y$ both support $d_r$-differentials.
We prove this below as \cref{syntheticleibnizrule}.
Note that it only requires a homotopy-commutative structure.

\begin{theorem}[Burklund]
    Let $R$ be a homotopy ring in $\Syn_E$.
    For any $r\geq 1$, the map
    \[\delta_r^{2r} \colon \oppi_{*,*}R/\tau^r \to \oppi_{*-1,\, *+r+1} R/\tau^r\]
    is a derivation. In particular, for any two classes $x,y \in \oppi_{\ast,\ast}R/\tau^r$, we have the equality
    \[\delta_r^{2r}(xy) = \delta_r^{2r}(x)\cdot y + (-1)^{\abs{x}}x\cdot  \delta_{r}^{2r}(y).\]
\end{theorem}

Using the synthetic Leibniz rule, we solve a question posed by Isaksen--Kong--Li--Ruan--Zhu \cite[Problem~1.2]{isaksen_etal_motivic_ANSS_tmf}. There is a differential $d_7(\Delta^4)=h_1^3g\Delta^3$ in the DSS that does not follow from the ordinary Leibniz rule. This first appeared in the context of the ANSS for $\tmf$ without proof in \cite[Section 15]{tmf_book}. It appeared again in Bauer's account of ANSS for $\tmf$ \cite{bauer_tmf}; however, it only appears in his charts and is not mentioned in the text.
The difficulty of this differential was pointed out by \cite{isaksen_etal_motivic_ANSS_tmf}, who showed that it follows from information in the ASS for $\tmf$, which requires the Gap Theorem. We show that it follows quite easily from the total differential $\delta_4^8(\Delta)$ along with the synthetic Leibniz rule.
For the full proof, see \cref{prop:dsevens}.

\begin{theorem}
    In the DSS for $\Tmf$, there is a differential $d_7(\Delta^4)=h_1^3g\Delta^3$.
\end{theorem}

Indeed, applying the synthetic Leibniz rule to the total differential $\delta_4^8(\Delta)=\nu\kappabar$ yields
\[\delta_4^8(\Delta^4)=4\Delta^3\delta_4^8(\Delta) = 4\nu\kappabar\Delta^3 = \tau^2\eta^3\kappabar\Delta^3.\]
We are not aware of another way to deduce this differential that does not depend on the Gap theorem.

\subsubsection{Synthetic Toda brackets and Moss' Theorem}
We require the extensive use of Toda brackets to deduce hidden extensions and nonzero differentials, as well as to rule out possible differentials. Toda brackets in synthetic spectra often have advantages over their non-synthetic counterparts. For example, it is often easier in this setting to apply versions of Moss' theorem, which gives conditions for when a Toda bracket is detected by a Massey product formed on some page of a spectral sequence. Moreover, synthetic Toda brackets often have smaller indeterminacy than their non-synthetic counterparts.

We give a treatment of Toda brackets formed in the Picard-graded homotopy groups of the unit in a monoidal stable $\infty$-category and various bracket shuffling formulas; see \cref{sec:toda_brackets_setup}.
This is far from the most general context in which one can speak of Toda brackets, but is sufficient for our applications, in particular in the bigraded homotopy of synthetic $\E_\infty$-rings like $\Smf/\tau^k$.
Our approach generalises what is already in the literature by allowing for brackets of arbitrary length and the flexibility of working in a stable monoidal $\infty$-category. This approach applies quite broadly, for example to motivic and equivariant spectra.

We also describe a general approach to determining where synthetic Toda brackets (of arbitrary length) are detected in their associated spectral sequences. Here, we closely follow ideas of Burklund in \cite{burklund_cookware_draft}. In particular, we prove a general form of Moss' theorem that applies to $3$-fold synthetic Toda brackets; see \cref{thm:synthetic_moss}. We apply this extensively to $3$-fold brackets, and we apply our approach to some crucial $4$-fold brackets as well.

One such example is the classical bracket $\angbr{\kappa, 2, \eta, \nu}$ from the sphere.
In $\Smf$, one has $\tau^2 \kappabar \in \angbr{\kappa, 2, \eta, \nu}$; see \cref{sssec:kappabar}. As the class $\kappabar$ comes from the synthetic sphere, the Nishida nilpotence theorem tells us that some power of $\kappabar$ is $\tau$-power torsion in the synthetic sphere, and hence also in $\Smf$. Using the $4$-fold Toda bracket containing $\kappabar$ and the shuffling formula for $4$-fold Toda brackets in $\Smf/\tau^{24}$, we find
\[\tau^{22} \kappabar^6 \in \tau^{16}\kappabar^{4} \angbr{\kappa, 2, \eta, \nu} \tau^4\kappabar = \angbr{\tau^{16}\kappabar^4, \kappa, 2, \eta} \tau^4\nu \kappabar = 0,\]
using that $\tau^4 \nu\kappabar = 0$, a consequence of the key $d_5$-differential of \cref{prop:d5_Delta}.
The truncated Omnibus Theorem (\cref{thm:omnibus}) now yields the key $d_{23}$-differential.
For the full proof, see \cref{prop:diff23}.

\begin{theorem}
    In the DSS for $\Tmf$, there is a differential $d_{23}(h_1\Delta^5)=g^6$.
\end{theorem}

In \cite{bauer_tmf}, Bauer produces this $d_{23}$ using various shuffling formulas for $6$-fold Toda brackets. The use of $6$-fold brackets requires subtle indeterminacy arguments and some version of Moss' theorem, which Bauer does not discuss. Using our treatment of Toda brackets along with the approach to Moss' theorem, it seems plausible one could verify the $6$-fold shuffling arguments that Bauer gives for this $d_{23}$. This would involve delicate indeterminacy checks however, and our approach is much more direct.

\subsubsection{A synthetic transfer argument}
At several points in our computation, we need additional input to verify that a certain class does not support a differential. In many cases, this follows for degree reasons or because the class comes from a permanent cycle in the ANSS for $\S$, and in other cases we can show the class is a cycle because it detects a Toda bracket in $\Smf/\tau^k$ for a large enough $k$. In one crucial case, these methods are not sufficient, and we need to apply a transfer map coming from spectral algebraic geometry.

The map of stacks
\[\Mbar_1(3)\times \Spec \Z_{(2)}\to \Mellbar\times \Spec \Z_{(2)}\]
determines a \emph{synthetic transfer map}
\[\Smf_1(3)\to\Smf\]
which induces mod $\tau$ the algebraic transfer map on sheaf cohomology; see \cref{thm:synthetictransfers}. The DSS for $\Tmf_1(3)$ has no differentials, hence anything in the image of this transfer map is a permanent cycle in the DSS for $\Tmf$.
The following is part of \cref{survivalofmodularformsattwo}.

\begin{theorem}
    The class $2c_6\in\oppi_{12,0}\Smf/\tau$ is a permanent cycle in the DSS.
\end{theorem}

\begin{remark}
    If one assumes the Gap Theorem, the corresponding claim in the ANSS for $\tmf$ follows easily for degree reasons. In the DSS however, there is a potential $d_7$ on $2c_6$ that we see no other way of ruling out.
\end{remark}

\subsection{Outline}
The core of this paper is \cref{sec:prime2}, where the $2$-local descent spectral sequence for $\Tmf$ is computed in its entirety.
We recommend the reader who is interested in learning the techniques of \cref{sec:synthetic_tools,sec:toda_brackets_setup} to start there, and refer back to the precise statements in previous sections as needed.
Tables and figures summarising these computations may be found in \cref{theappendix}.

In \cref{sec:synthetic_tools}, we discuss the total differentials (\cref{ssec:total_differentials}), and we prove the the Omnibus Theorem (\cref{thm:omnibus,prop:precise_lifting_higher_powers_tau}) and Burklund's synthetic Leibniz rule (\cref{syntheticleibnizrule}).
We do this by proving the analogous statements for filtered spectra, and then translating these results to synthetic spectra.
\Cref{ssec:filtered_spectra,ssec:tau_BSS} provide the necessary background on filtered spectra, the $\tau$-Bockstein spectral sequence, and their relationship to synthetic spectra for this.
In \cref{sec:toda_brackets_setup}, we give a general setup of Toda brackets, prove general shuffling formulas, and prove a synthetic version of Moss' theorem.
For both \cref{sec:synthetic_tools,sec:toda_brackets_setup}, when referring to $E$-synthetic spectra, we work with an arbitrary homotopy-associative ring spectrum $E$ of Adams type.
After these sections, we will always work in the special case $E = \MU$.

\Cref{sec:cubic_hopf_algebroid,detectionsection} form the setup for the computation.
In \cref{sec:cubic_hopf_algebroid}, we show how part of the $\uE_2$-page of the descent spectral sequence is computed by the cubic Hopf algebroid.
In \cref{detectionsection}, we show that certain low-dimensional classes from the sphere are detected in $\Smf$, and use this to import a few relations and Toda brackets.

After our $2$-primary computations in \cref{sec:prime2}, we follow up with a computation away from $2$ in \cref{sec:computations_away_from_2}.
Finally, in \cref{sec:proofs_main_theorems} we use these computations to prove the main theorems and corollaries mentioned above.


\subsection{Notation and terminology}
Roman letters are used for names of elements on the $\uE_2$-page, while elements in the homotopy groups of $\Smf/\tau^k$ for $k\geq 2$ or $\Tmf$ itself are denoted by Greek letters.

The symbol $\pm$ means an ambiguous sign, in the sense that $\pm x$ and $-(\pm x)$ have the same meaning.
In particular, we do not write $\mp x$.

The term \emph{synthetic spectrum} refers to an object of $E$-based synthetic spectra in the sense of \cite{pstragowski_synthetic}, where $E$ is a (fixed) homotopy-associative ring spectrum of Adams type.
In \cref{sec:synthetic_tools,sec:toda_brackets_setup}, the spectrum $E$ can be an arbitrary such spectrum, while in later sections we always work in the specific case where $E = \MU$.
When $E$ is clear from the context, we write $\Syn$ instead of $\Syn_E$.

All spectral sequences are indexed with Adams grading.
The bigraded synthetic sphere $\S^{n,s}$ is defined as
\[
    \S^{n,s} \defeq \opSigma^{-s}\nu(\S^{n+s}).
\]
This differs from the convention used in \cite{pstragowski_synthetic}.
Practically, it means that $(n,s)$ refers to an $(x,y)$-coordinate in an Adams chart.
In particular, the combination of these conventions means that $\oppi_{n,s}(\Smf/\tau) = \uE_2^{n,s}$.
We may write $\S$ for $\S^{0,0}$; we then let context determine whether we mean the sphere spectrum or the synthetic sphere.

We use the term \emph{hidden extension} in a simple way, referring to a relation $\alpha \beta = \tau^r \gamma$, where $\alpha$, $\beta$ and $\gamma$ have nonzero projections mod $\tau$ and where $r>0$. In particular, these are relations that hold in the homotopy groups of a synthetic spectrum (or a truncation thereof), without reference to the $\uE_2$-page.

The synthetic $\E_\infty$-ring $\Smf$ is the one defined in \cite[Definition~4.8]{CDvN_part1}; see also the recollection above in \cref{constr:recall_Smf_definition}.
We have analogously defined synthetic $\E_\infty$-rings $\SMF$ and $\Smf_\Gamma$, where $\Gamma$ denotes a type of level structure; see \cite[Definition~4.8, Variant~4.11]{CDvN_part1}, respectively.

\subsection{Acknowledgements}
We regard the computation of $\pi_*\tmf$ as a theorem of Hopkins--Mahowald, as announced in \cite[Section 15]{tmf_book}, and we are grateful to them for such an inspiring computation. Our paper is meant to give a complete proof of their theorem, and we are grateful to Akhil Mathew for observing that this could be done by proving our \cref{thm:gap}. It would have been a different task entirely to carry out this computation without the prior work of many authors on this topic, especially the work of Rezk \cite{rezk_tmf_lecture_notes}, Bauer \cite{bauer_tmf}, Konter \cite{konter_Tmf}, Henriques \cite[Section~13]{tmf_book}, Bruner--Rognes \cite{bruner_rognes_ASS_tmf}, and Isaksen--Kong--Li--Ruan--Zhu \cite{isaksen_etal_motivic_ANSS_tmf}. 

We are grateful to Lennart Meier for pointing out the circularity in the literature almost ten years ago, and for continuing to insist that it be corrected.

We thank Robert Burklund for helpful conversations. His work in synthetic spectra inspired some of this work.

The spectral sequence charts were produced using Hood Chatham's \texttt{spectralsequences} package, and we thank him for answering questions about his package.

Thank you as well to Tilman Bauer, Mark Behrens, Paul Goerss, Gerd Laures, Guchuan Li, Lennart Meier, Haynes Miller, Itamar Mor, Marius Nielsen, Charles Rezk, and John Rognes for comments on a draft, helpful conversations, and insights into the history of topological modular forms.

The first author was supported by NSF grant \texttt{DMS-2401918}.
The first and third author were supported by the NWO grant \texttt{VI.Vidi.193.111}.
The second author is an associate member of the Hausdorff Center for Mathematics at the University of Bonn (\texttt{DFG GZ 2047/1}, project ID \texttt{390685813}).


\section{Synthetic differentials and the Omnibus Theorem}
\label{sec:synthetic_tools}

The goal of this section is to fill our synthetic toolbox with everything we need in our computations.
The main players are the \emph{total differentials} of \cref{ssec:total_differentials} and the \emph{Omnibus Theorem} of \cref{ssec:omnibus}.
All results in this section can be, and are, proved in the setting of \emph{filtered spectra} rather than synthetic spectra; we discuss this connection in \cref{ssec:filtered_spectra}.
The $\infty$-category of filtered spectra is regarded as a homotopical version of the category of spectral sequences; see, e.g., \cite{antieau_decalage}, \cite{hedenlund_phd}, \cite[Section~1.2.2]{HA}, or \cite{christian_jack_synthetic_j,CDvN_part1} for sources using the same notation as here.
The formalism of $\tau$ naturally arises in the setting of filtered spectra too, and one can rephrase the yoga of spectral sequences in terms of~$\tau$.
In this way, results like the Omnibus Theorem become theorems in any good enough deformation, although for the sake of simplicity we have decided not to phrase the results in that generality.
The tool used to prove these results is the \emph{$\tau$-Bockstein spectral sequence}, which we set up in the filtered setting in \cref{ssec:tau_BSS}.

On the whole, we are economic in what we prove.
We do not intend to here give a full expository account of the $\tau$-Bockstein spectral sequence.
We also do not prove the full Omnibus Theorem here, but only the version for the truncated synthetic spectra $X/\tau^k$ for $k\geq 1$.
Nevertheless, such results are straightforward to prove given the methods and perspective we describe here.
These details, as well as a more in-depth explanation of the ideas involved here, are worked out in expository work \cite{vN_fil_syn_lecture_notes} of the third named author.


\subsection{Filtered spectra and deformations}
\label{ssec:filtered_spectra}

In this section, we briefly set up the theory of filtered spectra.
This is done for two purposes: to fix indexing conventions, and to allow ourselves to later reduce statements about synthetic spectra to filtered spectra.
We will be brief; for more information and references, we refer to \cite[Section~1]{christian_jack_synthetic_j} (which in particular uses the same indexing conventions), or \cite[Section~2]{barkan_monoidal_algebraicity}, \cite[Section~II.1]{hedenlund_phd}, and \cite{vN_fil_syn_lecture_notes}.

\begin{definition}
    A \defi{filtered spectrum} is a functor $\Z^\op \to \Sp$.
    We write $\FilSp \defeq \Fun(\Z^\op,\Sp)$ for the $\infty$-category of filtered spectra.
    We regard this as a symmetric monoidal $\infty$-category under Day convolution, where $\Z^\op$ carries addition as symmetric monoidal structure.
\end{definition}

If $X$ is a filtered spectrum, then we write $X(-\infty)$ for its colimit.
The \defi{associated graded} of~$X$ is the graded spectrum given by $\Gr^s X \defeq \cofib(X(s+1)\to X(s))$.
A filtered spectrum~$X$ naturally leads to a spectral sequence, indexed here as
\[
    \uE_2^{n,s} \defeq \pi_n(\Gr^{s+n} X) \implies \pi_n(X(-\infty))
\]
with differential $d_r$ of bidegree $(-1,r)$ for $r\geq 2$.
This converges conditionally to $\pi_*(X(-\infty))$ if and only if the limit spectrum $X(\infty)$ vanishes.

To make filtered spectra interface better with other categories (most notably, synthetic spectra), it is convenient to reformulate this spectral sequence in terms of an endomorphism $\tau$ of the filtered spectrum.
It is not truly an endomorphism, but only up to a twist, leading to the definition of the bigraded filtered spheres.

\begin{definition}
    \leavevmode
    \begin{numberenum}
        \item For $n,s\in\Z$, we write $\S^{n,s}$ for the filtered spectrum given by
        \[
            (\S^{n,s})(k) \defeq \begin{cases}
                \ \S^n &\text{if }k\leq s+n,\\
                \ 0 &\text{if }k>s+n
            \end{cases}
        \]
        and whose transition maps are given by the identities between the $n$-spheres, and zero else.
        \item We write $\oppi_{n,s}(\blank)$ for $[\S^{n,s},\blank]$, and $\Sigma^{n,s}$ for $\S^{n,s}\otimes \blank$.
        \item We write $\tau \colon \S^{0,-1}\to \S$ for the map
        \[
            \begin{tikzcd}
                \dotsb \ar[r] & 0 \ar[r] \ar[d] & 0 \ar[r] \ar[d] & \S \ar[r,equals] \ar[d,equals] & \dotsb\\
                \dotsb \ar[r] & 0 \ar[r] & \S \ar[r,equals] & \S \ar[r,equals] & \dotsb,
            \end{tikzcd}
        \]
        and we write $C\tau$ for the cofibre of this map.
    \end{numberenum}
\end{definition}

In this notation, suspension is given by $\Sigma^{1,-1}$, while shifting to the left is given by $\Sigma^{0,1}$.

If $X$ is a filtered spectrum, then we will also write $\tau \colon \opSigma^{0,-1}X\to X$ for the map $\tau \otimes X$.
By \cite[Proposition~3.2.5]{lurie_rotation_invariance}, there exists a canonical $\E_\infty$-algebra structure on $C\tau$, and by \cite[Proposition~3.2.7]{lurie_rotation_invariance}, there exists an equivalence $\grSp \simeq \Mod_{C\tau}(\FilSp)$ fitting into a commutative diagram
\[
    \begin{tikzcd}
        \FilSp \ar[r,"\Gr"] \ar[dr,"C\tau\otimes\blank"'] & \grSp \ar[d,"\simeq"] \\
        & \Mod_{C\tau}(\FilSp).
    \end{tikzcd}
\]
We say that a filtered spectrum is \defi{$\tau$-invertible} if the endomorphism $\tau$ on it is an equivalence.
It is straightforward to see that the colimit functor $X \mapsto X(-\infty)$ restricts to an equivalence from $\tau$-invertible spectra to spectra, with inverse equivalence given by the constant filtered spectrum functor.
As such, we will also write $X\mapsto X[\tau^{-1}]$ for the colimit functor.

In terms of these identifications, the spectral sequence associated to $X$ takes the form
\[
    \uE_2^{n,s} \defeq \pi_{n,s}(C\tau \otimes X) \implies \oppi_n X[\tau^{-1}],
\]
and the convergence condition reads that it converges conditionally if and only if $X$ is $\tau$-complete.

Phrased in this way, we can use filtered spectra to study other categories.
The structure we need on an $\infty$-category to do this is that of a \emph{deformation}.
We refer to \cite[Section~2]{barkan_monoidal_algebraicity} and \cite[Appendix~C]{burklund_hahn_senger_Rmotivic} for a discussion of these ideas.
For our purposes, we could work in the context of a presentably symmetric monoidal $\infty$-category $\C$ receiving a symmetric monoidal colimit-preserving functor $\FilSp\to \C$.
Such a functor has a (lax symmetric monoidal) right adjoint $\C \to \FilSp$, which we then think of as an ``underlying spectral sequence'' functor.
Rather than working in this generality however, we restrict ourselves to the deformation given by Pstr\k{a}gowski's $\infty$-category of synthetic spectra \cite{pstragowski_synthetic}.
This is only for simplicity: much of what we write generalises to any deformation arising in the above fashion.

Henceforth we fix a homotopy-associative ring spectrum $E$ of Adams type, and we write $\Syn$ for $\Syn_E$.

\begin{construction}
    \label{constr:left_adjoint_out_of_FilSp}
    Let $\C$ be a presentably symmetric monoidal stable $\infty$-category.
    The $\infty$-category $\FilSp$ is the stabilisation of the $\infty$-category $\PSh(\Z)$ of (space-valued) presheaves on $\Z$.
    As a result, it has a universal property: a symmetric monoidal functor $\Z \to \C$ corresponds to a symmetric monoidal colimit-preserving functor $\FilSp \to \C$.
\end{construction}

We think of a functor $\Z\to\C$ as a description of what the map $\tau$ should be in $\C$.
In $\Syn$, we have a pre-established notion of what $\tau$ should be, leading us to the following definition.

\begin{lemma}
    \label{lem:tau_tower}
    There is a natural symmetric monoidal structure on the functor $\Z \to \Syn$ given by the diagram in $E$-synthetic spectra
    \[
        \begin{tikzcd}
            \dotsb \ar[r,"\tau"] & \S^{0,-1} \ar[r,"\tau"] & \S \ar[r,"\tau"] & \S^{0,1} \ar[r,"\tau"] & \dotsb.
        \end{tikzcd}
    \]
\end{lemma}
\begin{proof}
    See the proof of Corollary~6.1 in \cite{lawson_cellular}.
\end{proof}

\begin{notation}
    \label{not:rho_sigma_adjunction}
    We write $\rho \colon \FilSp \to \Syn$ for the symmetric monoidal left adjoint induced by \cref{lem:tau_tower} via \cref{constr:left_adjoint_out_of_FilSp}.
    We write $\sigma \colon \Syn \to \FilSp$ for its right adjoint, and refer to it as the \defi{signature functor}.
\end{notation}

Concretely, if $X$ is a synthetic spectrum, then $\opsigma X$ is given by the filtered spectrum
\begin{equation}\label{eq:concrete_expression_sigma}
    \begin{tikzcd}
        \dotsb \ar[r,"\tau"] & \map(\S,\,  \Sigma^{0,-1} X) \ar[r,"\tau"] & \map(\S,\, X) \ar[r,"\tau"] & \map(\S,\, \Sigma^{0,1} X) \ar[r,"\tau"] & \dotsb.
    \end{tikzcd}
\end{equation}

The following definition already occurred in {\cite[Definition 1.5]{christian_jack_synthetic_j} and \cite[Definition 1.10]{CDvN_part1}}.

\begin{definition}
    \label{def:signature_sseq}
    Let $X$ be a synthetic spectrum.
    We refer to the spectral sequence underlying $\sigma X$ as the \defi{signature spectral sequence} of $X$.
\end{definition}

The adjunction between $\rho$ and $\sigma$ lets us understand the bigraded homotopy groups of a synthetic spectrum in terms of a spectral sequence.
The basis for this is the following result, which says that $\sigma$ preserves the relevant structure for a spectral sequence interpretation.
In what follows, if $X$ is a synthetic spectrum, we write $\tau_X$ for the endomorphim $\tau$ on $X$, and likewise if $X$ is a filtered spectrum.

\begin{proposition}
    \label{prop:sigma_and_pistar}
    Let $X$ be a synthetic spectrum.
    \begin{numberenum}
        \item The functor $\sigma$ sends $\tau_X$ to $\tau_{\sigma X}$.
        \item \label{item:reindex_syn_fil_htpy_groups} We have natural isomorphisms of $\Z[\tau]$-modules
        \[
            \pi_{n,*}(X) \cong \pi_{n,*}(\sigma X).
        \]
        \item The functor $\sigma$ preserves filtered colimits.
    \end{numberenum}
\end{proposition}
\begin{proof}
    The functor $\rho$ is characterised by preserving small colimits, by sending the filtered spectrum $\S^{0,s}$ to the synthetic sphere $\S^{0,s}$, and by sending the filtered map $\tau$ to the synthetic map $\tau$.
    It follows that $\rho(\S^{n,s}) \simeq \S^{n,s}$, which by adjunction implies that for $X \in\Syn$,
    \[
        \pi_{n,s}(X) \cong \pi_{n,s}(\sigma X).
    \]
    To see that this assembles to an isomorphism of $\Z[\tau]$-modules, one uses the expression \eqref{eq:concrete_expression_sigma} for $\sigma X$ to see that $\sigma$ sends the synthetic map $\tau$ on $X$ to the filtered map $\tau$ on $\sigma X$.
    
    As the filtered and synthetic spheres are compact, it follows that $\rho$ preserves compact objects.
    This implies that $\sigma$ preserves filtered colimits.
\end{proof}

\begin{remark}
    If $\Syn$ is \emph{cellular} in the sense that it is generated under colimits by the bigraded spheres, then $\sigma$ is also a conservative functor.
\end{remark}

In particular, the functor $\sigma$ commutes with inverting $\tau$.
We find that the signature spectral sequence of $X \in \Syn$ is of the form
\[
    \uE_2^{n,s} = \pi_{n,s}(C\tau \otimes X) \implies \oppi_n X[\tau^{-1}]
\]
with differentials $d_r$ of bidegree $(-1,r)$ for $r\geq 2$, and it converges conditionally if and only if $X$ is $\tau$-complete.
In many examples, one can identify this spectral sequence with a familiar one.

\begin{example}
    \leavevmode
    \begin{numberenum}
        \item Let $X$ be an $E$-nilpotent complete spectrum.
        Then the signature of $\nu X$ can be identified with (the d\'ecalage of) the $E$-Adams spectral sequence of $X$; see \cite[Section~1.4]{CDvN_part1} for a further discussion.
        
        \item The signature of $\Smf$ can be identified with (the d\'ecalage of) the descent spectral sequence of $\Tmf$; see \cite[Corollary~2.12]{CDvN_part1}.
        More generally, this applies to the global sections of a sheaf of synthetic $\E_\infty$-rings $\calO^\syn_\X$ arising from an even-periodic refinement $(\X,\calO_\X^\top)$, as discussed in \cite[Section~2.2]{CDvN_part1}.
        
        \item In another direction, the first-named author constructs an $\MU$-synthetic spectrum whose signature is the slice spectral sequence of a Borel $G$-spectrum with complex-oriented underlying spectrum; see \cite{carrick_slice}. \qedhere
    \end{numberenum}
\end{example}

The bigraded homotopy groups of a filtered spectrum capture the spectral sequence of the filtered spectrum.
As a result, we learn that the bigraded homotopy groups of a filtered spectrum capture its signature spectral sequence.
Our next objective is to make this relationship precise.


\subsection{Intermezzo: the \texorpdfstring{$\tau$}{tau}-Bockstein spectral sequence}
\label{ssec:tau_BSS}

We take a short detour to set up the (formal) $\tau$-Bockstein spectral sequence of a filtered spectrum.
The definitions can easily be interpreted for synthetic spectra as well, and applying $\sigma$ retrieves the case of filtered spectra.
This spectral sequence is the key to connecting the homotopy groups of a synthetic spectrum to its signature spectral sequence.

\begin{definition}
    \label{def:tau_BSS_filtrations}
    Let $X$ be a filtered spectrum.
    The \defi{$\tau$-Bockstein filtration} of $X$ is the bifiltered spectrum $\Z^\op \to \FilSp$ given by
    \[
        \begin{tikzcd}
            \dotsb \ar[r,"\tau"] & \opSigma^{0,-2} X \ar[r,"\tau"] & \opSigma^{0,-1}X \ar[r,"\tau"] & X \ar[r,equals] & \dotsb.
        \end{tikzcd}
    \]
    For $k\geq 1$, the \defi{$k$-truncated $\tau$-Bockstein filtration} of $X$ is the bifiltered spectrum $\Z^\op \to \FilSp$ given by
    
    \adjustbox{scale=0.98,center}{
        \begin{tikzcd}
            \dotsb \ar[r] & 0 \ar[r] & \opSigma^{0,\,-k+1} X/\tau \ar[r,"\tau"] & \dotsb \ar[r,"\tau"] & \opSigma^{0,-1}X/\tau^{k-1} \ar[r,"\tau"] & X/\tau^k \ar[r,equals] & \dotsb.
        \end{tikzcd}
    }
\end{definition}

We index the $\tau$-Bockstein spectral sequence as
\begin{equation}
    \label{eq:E1_tau_BSS}
    \uE_1^{n,w,s} \cong \begin{cases}
        \ \oppi_{n,\, w+s} X/\tau &\text{if }s\geq 0,\\
        \ 0 &\text{else.}
    \end{cases}
\end{equation}
The $\tau$-Bockstein spectral sequence
\[
    \uE_1^{n,w,s} \implies \oppi_{n,w} X
\]
has differentials $d_r^\tau$ of tridegree $(-1,1,r)$ for $r\geq1$, and converges conditionally to the $\tau$-adic filtration on $\oppi_{*,*}X$ if and only if $X$ is $\tau$-complete.

An element in filtration $s$ in the $\tau$-BSS can be thought of as a formal $\tau^s$-multiple of the corresponding element in filtration $0$.
For degree reasons, the differential $d_r^\tau$ only hits elements in filtration $\geq r$.
Similarly, $d_r^\tau$-differential can be thought of as killing formal $\tau^r$-multiples, thereby recording $\tau^r$-torsion in $\oppi_{*,*}X$.

\begin{remark}
    It is slightly awkward to index this spectral sequence as starting on the $\uE_1$-page, while our earlier convention for filtered spectra was to start these on the $\uE_2$-page.
    We do this because of the interpretation of the $\tau$-Bockstein differentials $d_r^\tau$ as killing formal $\tau$-multiples.
    One of the main purposes of this spectral sequence is to analyse $\tau$-power torsion in bigraded homotopy groups, so we find this indexing to be more convenient.
\end{remark}

The $k$-truncated $\tau$-Bockstein spectral sequence has
\begin{equation}
    \label{eq:E1_truncated_tau_BSS}
    { }^k\uE_1^{n,w,s} \cong \begin{cases}
        \ \oppi_{n,\, w+s} X/\tau &\text{if }0\leq s \leq k-1,\\
        \ 0 &\text{else,}
    \end{cases}
\end{equation}
and the spectral sequence
\[
    { }^k \uE_1^{n,w,s} \implies \oppi_{n,w} X/\tau^k
\]
converges strongly to the $\tau$-adic filtration on $\oppi_{*,*}X/\tau^k$.
As we will see shortly, the differentials turn out to coincide with the differentials in the non-truncated $\tau$-BSS.

In this article, we only require the truncated Bockstein spectral sequence.
Nevertheless, the non-truncated version is a helpful device to use in proofs.

The differentials in the $\tau$-Bockstein spectral sequence of $X$ are exactly the differentials in the underlying spectral sequence of $X$.
This requires some care to make precise, due to the fact that the differentials in the $\tau$-BSS raise filtration, so that the classes below filtration~$r$ cannot be hit by differentials of length larger than $r$.
For this reason, we momentarily view a $d_r$-differential as a map from $r$-cycles to the $\uE_r$-page.

\begin{lemma}
    \label{lem:diff_in_BSS}
    Let $X$ be a filtered spectrum.
    \begin{numberenum}
        \item The differentials in the $\tau$-BSS are linear with respect to formal $\tau$-multiplication: using~\eqref{eq:E1_tau_BSS}, the differential
        \begin{align*}
            \SwapAboveDisplaySkip
            d_r^\tau \colon \uZ_r^{n,w,s} &\to \uE_r^{n-1,\, w+1,\, s+r}\\
        \shortintertext{can be identified with}
            d_r^\tau \colon \uZ_r^{n,\,w-1,\,s+1} &\to \uE_r^{n-1,\, w,\, s+r+1}.
        \end{align*}
        \item Under the identification of \eqref{eq:E1_tau_BSS}, the differential
        \begin{align*}
            d_r^\tau \colon \uZ_r^{n,w,s} &\to \uE_r^{n-1,\, w+1,\, s+r}\\
        \shortintertext{is identified with}
            d_{r+1} \colon \uZ_{r+1}^{n,\,w+s} &\to \uE_{r+1}^{n-1,\, w+s+r+1}.
        \end{align*}
        \item For every $k\geq 1$, there is a morphism of spectral sequences
        \[
            \uE_r^{n,w,s} \to { }^k\uE_r^{n,w,s}
        \]
        that is an isomorphism in the range $0\leq s\leq k-1$.
        In particular, this captures all the differentials in the $k$-truncated $\tau$-Bockstein spectral sequence.
    \end{numberenum}
\end{lemma}

One should read items~(1) and~(2) inductively.
For instance, \eqref{eq:E1_tau_BSS} provides an identification of the $\uE_1$-$\tau$-BSS page with the $\uE_2$-page, so that a comparison between the $d_1^\tau$ and $d_2$-differentials makes sense.
Then (2) for $r=1$ says that \eqref{eq:E1_tau_BSS} induces an isomorphism $\uZ_1^{n,w,s}\cong \uZ_2^{n,\,w+s}$ for all $s$, and $\uB_1^{n,w,s}\cong \uB_2^{n,\,w+s}$ if $s\geq 1$.
This in particular results in an isomorphism $\uE_2^{n,w,s}\cong \uE_3^{n,w+s}$ for $s\geq 1$.
As a result, the comparison between the $d_2^\tau$ and $d_3$-differentials makes sense, as $d_2^\tau$ takes values in groups $\uE_2^{n,w,s}$ with $s\geq 2$.
This pattern continues for the higher differentials.

\begin{proof}
    The first claim is an elementary diagram chase.
    For the second, note that evaluation of the $\tau$-Bockstein filtration at level $s$ results in the filtered spectrum
    \[
        \begin{tikzcd}
            \dotsb \rar & X(s+2) \rar & X({s+1}) \rar & X(s) \rar[equals] & \dotsb.
        \end{tikzcd}
    \]
    This is the filtered spectrum used to define the differentials going out of $\uE_2^{*,s}$, so the result follows.

    For the last claim, we use the morphism of bifiltered spectra
    \[
        \begin{tikzcd}
            \dotsb \ar[r,"\tau"] & \opSigma^{0,-k} X \ar[r,"\tau"] \ar[d] & \opSigma^{0,-k+1}X \ar[r,"\tau"] \ar[d] & \opSigma^{0,-k+2}X \ar[r,"\tau"] \ar[d] & \dotsb \ar[r,"\tau"] & X \ar[d] \\
            \dotsb \ar[r] & 0 \ar[r] & \opSigma^{0,-k+1} X/\tau \ar[r,"\tau"] & \opSigma^{0,-k+2} X/\tau^2 \ar[r,"\tau"] & \dotsb \ar[r,"\tau"] & X/\tau^k
        \end{tikzcd}
    \]
    which clearly induces an equivalence on associated graded in the range $0\leq s\leq k-1$.
    This proves the claim.
\end{proof}


\subsection{Total differentials}
\label{ssec:total_differentials}

The \emph{total differential} is a map of synthetic spectra that in a precise sense captures information about differentials of all lengths.
One can think of this as a generalisation of the following situation in chain complexes.

\begin{example}
\label{ex:chain_cx_total_diff}
Let $C_*$ be a cdga, and let $\tau \in C_*$ be a cycle such that $C_*$ is $\tau$-torsion free and $\tau$-adically complete.
The $\tau$-adic filtration on $C_*$ gives rise to a spectral sequence of the form
\[
    \uE_1^{*,*} = \uH_*(C_*/\tau)[\tau]\implies \uH_*C_*.
\]
Let $x \in \uH_*(C_*/\tau)$.
By $\tau$-adic completeness, we may write the boundary of $x$ as a power series
\begin{equation}\label{eq:powerseries}
    d(x)= y_n\tau^n+y_{n+1}\tau^{n+1}+\dotsb\in C_*,
\end{equation}
where each $y_i$ is not divisible by $\tau$, and where $n>0$ since $x$ is a cycle in $C_*/\tau$.
It follows from this formula that there is a differential $d_n(x)=y_n\tau^n$,
so we see that the first term of this power series determines differentials.

The higher terms in \eqref{eq:powerseries} are also very important for deducing differentials. Suppose for instance one has that $2y_n=xy_{n+1}=0$ in $\uH_*(C_*/\tau)$ but $2y_{n+1}$ and $2xy_{n+2}$ are nonzero.
It follows then from the Leibniz rule for the $d_n$-differential that $2x$ and  $x^2$ are $d_{\le n}$ and $d_{\le n+1}$ cycles respectively. Applying the Leibniz rule for $d$ however, one has the differentials
\begin{align*}
    d_{n+1}(2x)&=2y_{n+1}\tau^{n+1}\\
    d_{n+2}(x^2)&=(2x\cdot y_{n+2})\tau^{n+2}.
\end{align*}
In particular, the use of \eqref{eq:powerseries} allows us to deduce differentials of length \emph{greater} than $n$.
\end{example}

One of the key advantages of working in synthetic spectra (or even just filtered spectra) is that it allows us to view the spectral sequence as arising from a $\tau$-Bockstein spectral sequence similar to the above.
Moreover, this allows us to mimic the above arguments to deduce longer differentials, by providing a suitable replacement for the boundary map as in \cref{ex:chain_cx_total_diff}.
We in fact do not have to modify this idea at all: we can use the boundary map in the cofibre sequence defining $C\tau$.

\begin{remark}
While we phrase everything for synthetic spectra, the results are actually about filtered spectra.
One obtains the synthetic statement from the filtered ones by applying the signature functor $\sigma$ from \cref{not:rho_sigma_adjunction}.
We have stated the results in terms of synthetic spectra because this is all we need in this paper.
\end{remark}

\begin{definition}
    Let $X$ be a synthetic spectrum.
    For $n\geq 1$, we write $\delta_n^\infty$ for the boundary map in the cofibre sequence
    \[
        \begin{tikzcd}
            \opSigma^{0,-n} X \ar[r,"\tau^n"] & X \ar[r] & X/\tau^n \ar[r,"\delta_n^\infty"] & \opSigma^{1,\, -n-1} X.
        \end{tikzcd}
    \]
    For $N \geq n$, we write $\delta_n^N$ for the boundary map in the cofibre sequence
    \[
        \begin{tikzcd}
            \opSigma^{0,-n}X/\tau^{N-n} \ar[r,"\tau^n"] & X/\tau^N \ar[r] & X/\tau^n \ar[r,"\delta_n^N"] & \opSigma^{1,\, -n-1} X/\tau^{N-n}.
        \end{tikzcd}
    \]
    We call $\delta_1^\infty$ the \defi{total differential}, and $\delta_1^N$ the \defi{$N$-truncated total differential}.
\end{definition}

Using the long exact sequence, we see that for $x \in \oppi_{*,*}X/\tau^n$, the element $\delta_n^N(x)$ is the obstruction to lifting $x$ to $\oppi_{*,*}X/\tau^N$.

Informally, the map $\delta_n^N$ captures information about the $d_{n+1},\dotsc,d_N$-differentials in the signature spectral sequence of $X$.
While decreasing $N$ results in a loss of information, one should think of increasing $n$ as an increase of information: roughly speaking, $\delta_n^N$ is only defined on elements on which the differentials $d_2,\dotsc,d_{n}$ vanish.
This second intuition will be made precise by our later description of $\oppi_{*,*}X/\tau^n$ in \cref{ssec:omnibus}.

The following codifies the relation between the various total differentials.

\begin{proposition}\label{totaldifferentialjuggling}
    Let $n\geq 1$ and $n \leq N \leq \infty$, and let $X$ be a synthetic spectrum.
    \begin{numberenum}
        \item We have a commutative diagram
        \[
            \begin{tikzcd}
                X/\tau^n \ar[r,"\delta_n^\infty"] \ar[dr, "\delta_n^N"'] & \opSigma^{1,\,-n} X \ar[d]\\
                & \opSigma^{1,\,-n} X/\tau^{N-n}.
            \end{tikzcd}
        \]
        \item For $n\geq k\geq 1$, we have $\tau^{n-k} \cdot \delta_n^\infty = \delta_k^\infty$.
        \item \label{item:linearity_total_differential}
        Suppose $X$ is a \textbr{left} homotopy-module over a homotopy-associative ring $R$ in $\Syn$.
        Then the map $\oppi_{*,*}\delta_n^N$ is $\oppi_{*,*}R$-linear.
        If $N<\infty$, then the map $\oppi_{*,*}\delta_n^N$ is also $\oppi_{*,*}R/\tau^N$-linear.
    \end{numberenum}
\end{proposition}
\begin{proof}
    For readability, we omit bigraded suspensions in this proof.
    The first claim follows from the diagram
    \[
        \begin{tikzcd}
            X \ar[r,"\tau^n"] \ar[d] & X \ar[r] \ar[d] & X/\tau^n \ar[r,"\delta_n^\infty"] \ar[d, equals] & X \ar[d]\\
            X/\tau^{N-n} \ar[r,"\tau^n"] & X/\tau^N \ar[r] & X/\tau^n \ar[r,"\delta_n^N"] & X/\tau^{N-n}.
        \end{tikzcd}
    \]
    The second follows from the diagram
    \[
        \begin{tikzcd}
            X/\tau^n \ar[d] \ar[r, "\delta_n^\infty"] & X \ar[d,"\tau^{n-k}"] \ar[r,"\tau^n"] & X \ar[d,equals]\\
            X/\tau^k \ar[r,"\delta_k^\infty"] & X \ar[r,"\tau^k"] & X.
        \end{tikzcd}
    \]
    The last claim follows as $\delta_n^N$ is the boundary map of a cofibre sequence of $C\tau^N$-modules.
\end{proof}

Next, our goal is to relate the total differentials to the differentials in the signature spectral sequence.
As the input to $\delta_n^N$ is $X/\tau^n$, this requires an understanding of the homotopy groups $\oppi_{*,*}X/\tau^n$.
For this reason, we postpone this discussion to the next section; see \cref{prop:total_differential_versus_differentials}.


\subsection{The truncated Omnibus Theorem}
\label{ssec:omnibus}

In this section we give a precise description of how the bigraded homotopy groups of a synthetic spectrum capture its signature spectral sequence (see \cref{def:signature_sseq}).
The intuition is that if a class $x \in \uE_2^{n,s}$ is hit by a $d_r$-differential, then $x$ gives rise to an element in $\oppi_{*,*}X$ that is $\tau^{r-1}$-torsion.
While for $\oppi_{*,*}X$ this roughly captures all behaviour, for $X/\tau^k$, the situation is different: classes supporting differentials also contribute to the bigraded homotopy groups, but only in a limited fashion.
Our later computations only need a description of these truncated homotopy groups, so we only prove the truncated version.
As before, all results are about filtered spectra, but we have stated them for synthetic spectra.

Both our formulation of and proof strategy for the truncated Omnibus Theorem are modelled on the one given by Burklund--Hahn--Senger in \cite[Appendix~A]{burklund_hahn_senger_manifolds}.
The main difference in the proof is that the role played by the $\nu E$-ASS in their proof is played by the (truncated) $\tau$-BSS in ours.

\begin{theorem}[Truncated Omnibus, part 1]\label{thm:omnibus}
    Let $X$ be a synthetic spectrum, and let $n,s\in\Z$ be integers.
    Let $x \in \uE_2^{n,s} = \oppi_{n,s}X/\tau$ be a class.
    For every $k\geq 1$, the following are equivalent.
    \begin{enumerate}[label={\upshape(1\alph*)}]
        \item The differentials $d_2(x),\dotsc,d_{k}(x)$ vanish.
        \item The element $x \in \oppi_{n,s}X/\tau$ lifts to an element of $\oppi_{n,s}X/\tau^k$.
    \end{enumerate}
    For any such lift $\alpha$ to $\oppi_{n,s} X/\tau^k$, the following are true.
    \begin{enumerate}[label={\upshape(2\alph*)}]
        \item If $x$ survives to page $r$ for $r \leq k$, then $\tau^{r-2}\cdot \alpha$ is nonzero.
        \item The image of $\alpha$ under $\delta_{k-1}^{k} \colon \oppi_{n,s} X/\tau^k \to \oppi_{n-1,\, s+k} X/\tau$ is a representative for $d_{k}(x)$.
    \end{enumerate}
    Moreover, if $x$ lifts, then there exists a lift $\alpha$ that also satisfies the following.
    \begin{enumerate}
        \item[{\upshape(3)}] If $x$ is the target of a $d_r$-differential for $r<k$, then $\tau^{r-1} \cdot \alpha = 0$.
    \end{enumerate}
\end{theorem}
\begin{proof}
    Using the functor $\sigma$, we can translate the theorem into one for filtered spectra.
    We prove this version, so assume from now on that $X$ is a filtered spectrum.
    
    Consider $x$ as an element of tridegree $(n,s,0)$ (i.e., in filtration $0$) in the $k$-truncated $\tau$-BSS sequence for~$X$.
    Observe that on the $\uE_2$-pages of the truncated $\tau$-BSSs, the map $X/\tau^k \to X/\tau$ is given by projecting onto the filtration $0$ part.
    As a result, the equivalence of (1a) and (1b) follows from the identification of the differentials in the $k$-truncated $\tau$-BSS from \cref{lem:diff_in_BSS}.
    This also tells us (2a), as a $d_t^\tau$-differential only hits elements in $\tau$-Bockstein filtration $t$, so in particular it cannot kill elements of filtration below $t$.

    Next, we prove (2b).
    Let $\alpha \in \oppi_{n,s}X/\tau^k$ be a lift of $x$.
    Note that by \cref{lem:diff_in_BSS}, it is enough to show that $\delta_{k-1}^k(\alpha)$ is a representative for $d_{k-1}^\tau(x)$ in the non-truncated $\tau$-BSS, where we regard $x$ in tridegree $(n,s,0)$.
    Recall how the $d_{k-1}$-differential on $x$ in the non-truncated $\tau$-BSS is computed: we apply the boundary map $\delta_1^\infty(x)$, choose a $\tau^{k-2}$-division of this element (which exists precisely by the assumption that $d_1^\tau(x) = \dotsb = d_{k-2}^\tau(x)$), and reduce this mod $\tau$.
    \[
        \begin{tikzcd}
            \opSigma^{0,\,-k+1} X \ar[r,"\tau^{k-2}"] \ar[d] & \Sigma^{0,\, -1}X \ar[r,"\tau"] & X \ar[d]\\
            \opSigma^{0,\,-k+1} X/\tau & &  X/\tau \ar[ul, dashed, "\delta_1^\infty"]
        \end{tikzcd}
    \]
    We have $\tau^{k-2} \cdot \delta_{k-1}^\infty (\alpha) = \delta_1^\infty(x)$.
    In other words, $\delta_{k-1}^\infty(\alpha)$ is a valid choice of $\tau^{k-2}$-division of $\delta_1^\infty(a)$, so that its projection to $\oppi_{n,\, s+k}X/\tau$ is a representative for $d_{k-1}^\tau(x)$.
    But the mod~$\tau$ reduction of $\delta_{k-1}^\infty$ is $\delta_{k-1}^{k}$, proving (2b).

    Lastly, we prove~(3).
    Suppose that $x$ is the target of a $d_r$-differential for $r<k$.
    Our goal is to find a lift $\alpha \in \oppi_{n,s}X/\tau^k$ such that the relation $\tau^{r-1} \cdot \alpha = 0$ holds in $\oppi_{n,\,s-r+1}X/\tau^k$.
    Let us consider $x$ as an element of the $k$-truncated $\tau$-Bockstein spectral sequence in tridegree $(n,\ s-r+1,\ r-1)$.
    Using \cref{lem:diff_in_BSS}, the $d_r$-differential hitting $x$ translates to a $d_{r-1}^\tau$-differential hitting the $x$ in this tridegree.
    
    When we unroll what this means, we learn the following.
    There exists is a class $y \in \oppi_{n+1,\,s-r} X/\tau$, and a lift $\beta$ of $\delta_1^{k}(y)$ under $\tau^{r-2} \colon \opSigma^{0,\,-r+1}X/\tau^{k-r+1}\to \opSigma^{0,-1} X/\tau^{k-1}$ such that $\beta$ reduces to $x$ under $X/\tau^{k-r+1} \to X/\tau$.
    By exactness, the element $\delta_1^{k}(y)$ maps to zero under $\tau\colon \opSigma^{0,\,-1} X/\tau^{k-1}\to X/\tau^k$, so any choice of $\beta$ maps to zero under
    \[
        \tau^{r-1} \colon \opSigma^{0,\,r-1} X/\tau^{k-r+1} \to X/\tau^k.
    \]
    As a result, it suffices to show that there is a choice of $\beta \in \oppi_{n,s}X/\tau^{k-r+1}$ that lifts to an $\alpha \in \oppi_{n,s}X/\tau^k$.
    Indeed, we have a commutative diagram
    \[
        \begin{tikzcd}
            \opSigma^{0,\,-r+1}X/\tau^k \rar["\tau^{r-1}"] \ar[d] & X/\tau^k,\\
            \opSigma^{0,\,-r+1} X/\tau^{k-r+1} \ar[ur,"\tau^{r-1}"'] &
        \end{tikzcd}
    \]
    so if such an $\alpha$ exists, then $\tau^{r-1}\cdot \alpha = 0$.
    
    To produce such an $\alpha$, we first show that $y \in \oppi_{n+1,\,s-r} X/\tau$ lifts to $X/\tau^{r-1}$.
    This follows from (1) because $y$ is a $d_{\leq r-1}$-cycle.
    Write $\widetilde{y}$ for a lift.
    It then follows that $\delta_{r-1}^k(\widetilde{y})$ is a valid choice for $\beta$ as above.
    To show that this $\beta$ lifts to $X/\tau^k$, we need to show that $\delta_{k-r+1}^{k}(\beta)=0$.
    Note that $\delta_{k-r+1}^k\circ \delta_{r-1}^k$ can be written as (omitting shifts for readability)
    \[
        \begin{tikzcd}
            X/\tau^{k-1} \rar["\delta_{k-1}^\infty"] & X \rar & X/\tau^{k-r+1} \rar["\delta_{k-r+1}^\infty"] &[1em] X \rar & X/\tau^{r-1}.
        \end{tikzcd}
    \]
    The middle two maps are part of a cofibre sequence, so in particular their composition is zero.
    This means that indeed $\delta_{k-r+1}^k(\beta) = 0$, showing that $\beta$ lifts to the desired $\alpha$, thus proving~(3).
\end{proof}

\begin{warning}
    \label{warn:diff_not_imply_tau_torsion}
    Suppose that $\alpha \in \oppi_{*,*}X/\tau^k$ is a lift of $x \in \oppi_{*,*}X/\tau$, and that $x$ is the target of a $d_r$-differential, where $r < k$.
    This does \emph{not} imply that $\tau^{r-1}\cdot \alpha = 0$.
    Indeed, \cref{thm:omnibus}\,(3) only tells us that there exists \emph{some} lift of $x$ that is $\tau^{r-1}$-torsion.
\end{warning}

We can now deliver on our promise to relate the total differentials to the differentials in the signature spectral sequence.
Part of this is described by \cref{thm:omnibus}\,(2b), which we now expand on.

\begin{construction}
    \label{constr:map_from_mod_tau_r_to_Er}
    Let $X$ be a synthetic spectrum, let $n,s\in\Z$, and let $r \geq 1$.
    The reduction map $\oppi_{n,s} X/\tau^r \to \oppi_{n,s}X/\tau$ lands, by definition, in $\uE_2^{n,s}$ of the signature spectral sequence of $X$.
    By \cref{thm:omnibus}\,(1), this map lands in the subgroup $\uZ_{r+1}^{n,s}$.
    Postcomposing this with the natural quotient, we obtain a map
    \[
        \oppi_{n,s} X/\tau^r \to \uE_{r+1}^{n,s}.
    \]
    These maps for different values of $r$ are compatible in the obvious way.
\end{construction}

Because our computations only require a description of the truncated total differentials, we state the comparison for these only.
Unfortunately, this leads to a slightly more complicated indexing to keep track of.

\begin{proposition}
    \label{prop:total_differential_versus_differentials}
    Let $X$ be a synthetic spectrum, and let $x \in \oppi_{n,s}X/\tau$.
    Consider $x$ as an element on the $\uE_2$-page of the signature spectral sequence of $X$.
    Let $r \geq 2$.
    \begin{numberenum}
        \item If $\delta_1^{r+1}(x)$ reduces to zero mod $\tau^{r-1}$, then $x$ is a $d_{\leq r}$-cycle.
        \item If $\alpha \in \oppi_{n-1,\,s+r+1}X/\tau$ is a preimage of $\delta_1^{r+1}(x)$ under the map
        \begin{equation}
            \label{eq:mult_by_tau_rmin_1}
            \tau^{r-1} \colon \opSigma^{0,\, -r+1}X/\tau \to X/\tau^r,
        \end{equation}
        then the image of $\alpha$ under
        \[
            \oppi_{n-1,\, s+r+1} X/\tau \to \uE_{r+1}^{n-1,\, s+r+1}
        \]
        coincides with $d_{r+1}(x)$.
        \item \label{item:total_diff_versus_differential} Let $R \geq r$. Then we have a commutative diagram
        \[
            \begin{tikzcd}
                \oppi_{n,s}X/\tau^{r-1} \ar[r,"\delta_{r-1}^R"] \ar[d] & \oppi_{n-1,\, s+r} X/\tau^{R-r+1} \ar[d] \\
                \uE_r^{n,s} \ar[r,"d_r"] & \uZ_{R-r+1}^{n-1,\, s+r}/\uB_r^{n-1,\, s+r},
            \end{tikzcd}
        \]
        where the right vertical map is the map from \cref{constr:map_from_mod_tau_r_to_Er} composed with the natural projection $\uZ_{R-r+1} \to \uZ_{R-r+1}/\uB_r$.
    \end{numberenum}
\end{proposition}
\begin{proof}
    We again reduce ourselves to filtered spectra by applying $\sigma$.
    If $\delta_1^{r+1}(x)$ reduces to zero mod $\tau^{r-1}$, then it follows by unwinding the $(r+1)$-truncated $\tau$-Bockstein spectral sequence that $x \in \uE_1^{n,s,0}$ is a $d^\tau_{\leq r}$-cycle.
    Via the identifications of \cref{lem:diff_in_BSS}, items~(1) and~(2) follow.
    To see item~(3), note that $\delta_{r-1}^r$ is the mod $\tau$ reduction of $\delta_{r-1}^R$.
    If now $\alpha \in \oppi_{n,s}X/\tau^{r-1}$, then $\delta_{r-1}^{r}(\alpha)$ composed with \eqref{eq:mult_by_tau_rmin_1} is $\delta_1^r(\alpha)$.
    This proves item~(3).
\end{proof}

\begin{remark}
    \label{rmk:distinction_tau_divble}
    The condition that $\delta_1^{r+1}(x)$ reduces to zero mod $\tau^{r-1}$ is equivalent to $\delta_1^{r+1}(x)$ being in the image of \eqref{eq:mult_by_tau_rmin_1}.
    This in particular happens if $\delta_1^{r+1}(x)$ is $\tau^{r-1}$-divisible in $\oppi_{*,*}X/\tau^r$.
    In general however, the converse need not be true.
\end{remark}

Whereas the first part of the Omnibus Theorem concerns finding lifts of elements from the spectral sequence, the second part is concerned with the question of how many lifts we need in order to generate $\oppi_{n,s}X/\tau^k$.
This is where the truncated version deviates the most from the non-truncated version: elements in the signature spectral sequence that support differentials also determine nonzero elements in $\oppi_{*,*}X/\tau^k$.
This is captured by the following phenomenon.

\begin{remark}
    Let $y \in \uE_2^{n,w}$ such that $d_r(y) = x$, let $k \geq r$, and let $\alpha$ be a $\tau^{r-1}$-torsion lift of $x$ to $X/\tau^k$.
    The class $y$ does not contribute to $X/\tau^k$ in degree $(n,w)$, but it does contribute to it in degree $(n,\,w-k+r-1)$.
    Indeed, the differential forces the Toda bracket
    \[
        \angbr{\alpha, \tau^{r-1}, \tau^{k-r+1}} \subseteq \oppi_{n,\,w-k+r-1}X/\tau^k.
    \]
    to be nonempty. However, elements in this Toda bracket will not lift to the non-truncated case unless $\tau^k$ acts by zero on $X$.
\end{remark}

In general, the following criterion gives a very precise description of how large the group $\oppi_{n,s}X/\tau^k$ is.

\begin{theorem}[Truncated Omnibus, part 2]
    \label{prop:precise_lifting_higher_powers_tau}
    Let $X$ be a synthetic spectrum, let $n,s\in\Z$, and let $k\geq 1$.
    Write $\uE_r^{n,s}$ for the $r$-th page of the signature spectral sequence of $X$.
    Write $\uB_r^{n,s}\subseteq \uE_r^{n,s}$ for the $r$-boundaries, and write $\uZ_r^{n,s} \subseteq \uE_r^{n,s}$ for the $r$-cycles.
    By convention, we write $\uZ_1 = \uE_2$ and $\uB_1 = 0$.
    
    \begin{numberenum}
        \item Let $k\geq 1$ be fixed.
        Suppose that for every $1 \leq i \leq k$, we have a collection of elements
        \[
            \set{\beta^i_j}_j \subseteq \oppi_{n,\, s+k-i}X/\tau^i
        \]
        whose images under $X/\tau^i \to X/\tau$ generate the abelian group
        \begin{equation}\label{eq:cycles_mod_boundaries}
            \uZ_i^{n,\,s+k-i}/\uB_{k+1-i}^{n,\, s+k-i}.
        \end{equation}
        \textbr{Note that by \cref{thm:omnibus}\,\textup{(1)}, such a collection exists for every $i$.}
        Let $\alpha^i_j \in \oppi_{n,s}X/\tau^k$ denote the image of $\beta^i_j$ under the map
        \[
            \tau^{k-i}\colon \opSigma^{0,\,-k+i}X/\tau^i \to X/\tau^k.
        \]
        Then $\set{\alpha^i_j}_{i,j}$ is a set of generators for the abelian group $\oppi_{n,s}X/\tau^k$.
        
        \item Let $1 \leq m \leq k$ be fixed.
        Suppose that for every $1 \leq i \leq k-m$, we have a collection of elements
        \[
            \set{\beta^i_j}_j \subseteq \oppi_{n,\, s+k-i}X/\tau^i
        \]
        whose images under $X/\tau^i \to X/\tau$ generate the abelian group \eqref{eq:cycles_mod_boundaries}.
        Let $\alpha^i_j \in \oppi_{n,s}X/\tau^k$ denote the image of $\beta^i_j$ under the map $\tau^{k-i}$.
        Then $\set{\alpha^i_j}_{i,j}$ is a set of generators for the abelian group
        \[
            \ker ( \oppi_{n,s}X/\tau^k \to \oppi_{n,s}X/\tau^m).
        \]
    \end{numberenum}
\end{theorem}

\begin{proof}
    Using the isomorphism $\oppi_{n,s} X \cong \oppi_{n,s}\sigma X$ of \cref{prop:sigma_and_pistar}, we are reduced to proving the same statement for filtered spectra.
    The $k$-truncated Bockstein spectral sequence converges strongly to $\oppi_{*,*}X/\tau^k$, because the filtration giving rise to it is zero for $s \gg 0$.
    The first result therefore follows from \cref{lem:diff_in_BSS}.
    The second follows analogously by considering the natural map from the $k$-truncated $\tau$-BSS to the $m$-truncated $\tau$-BSS for~$X$, defined in the same way as the map in the proof of \cref{lem:diff_in_BSS}.
\end{proof}

\begin{warning}
    In general, the element $\alpha^i_j$ as in \cref{prop:precise_lifting_higher_powers_tau} need not be a $\tau^{k-i}$-multiple in the $\Z[\tau]$-module $\oppi_{*,*}X/\tau^k$; cf.\ \cref{rmk:distinction_tau_divble}.
\end{warning}

Sometimes, the following simplified criterion is sufficient.

\begin{corollary}
    \label{cor:easy_lifting_higher_power_tau}
    Let $X$ be a synthetic spectrum, let $n,s\in\Z$, and let $k \geq m \geq 1$.
    \begin{numberenum}
        \item If $\oppi_{n,\, s+d} X/\tau$ vanishes for $0 \leq d \leq k-1$, then $\oppi_{n,s}X/\tau^k$ vanishes.
        \item If $\oppi_{n,\, s+d} X/\tau$ vanishes for $m \leq d \leq k-1$, then the reduction map $\oppi_{n,s}X/\tau^k \to \oppi_{n,s}X/\tau^m$ is injective.
    \end{numberenum}
\end{corollary}
\begin{proof}
    In the notation of \cref{prop:precise_lifting_higher_powers_tau}, we have $\uE_2^{n,s} = \oppi_{n,s}X/\tau$, and $\uZ_r^{n,s}$ is a subgroup of this.
    It follows that the relevant groups in \eqref{eq:cycles_mod_boundaries} vanish, so the claim follows.
\end{proof}

We will also need a non-truncated version of the Omnibus Theorem at one point in our computation; see \cref{lem:Delta24_lifts}.
For the sake of brevity, we only prove what we need.

\begin{proposition}
    \label{prop:non_truncated_omnibus}
    Let $X$ be a synthetic spectrum.
    Suppose that $X$ is $\tau$-complete and that the signature spectral sequence converges strongly.
    Let $x \in \uE_2^{n,s}$.
    Then the following are equivalent.
    \begin{letterenum}
        \item The class $x$ is a permanent cycle, i.e., all differentials on $x$ vanish.
        \item The class $x$ lifts to $\oppi_{n,s}X$.
    \end{letterenum}
\end{proposition}
\begin{proof}
    We may again apply $\sigma$ using that $\sigma$ preserves limits, so for the rest of the proof we work with filtered spectra.
    The strong convergence of the underlying spectral sequence is equivalent to the strong convergence of the $\tau$-Bockstein spectral sequence.
    Consider $x$ as an element of $\uE_1^{n,s,0}$.
    Note that the map from the $\tau$-BSS to the $1$-truncated $\tau$-BSS for~$X$ is given on $\uE_1$-pages by the projection onto filtration $0$.
    As a result, $x$ lifts to $\oppi_{n,s}X$ if and only if $x$ as an element in $\uE_1^{n,s,0}$ is a permanent cycle.
    Using the identifications of \cref{lem:diff_in_BSS}, we find that $x$ defines a permanent cycle in the $\tau$-BSS if and only if it does in the underlying spectral sequence.
    This proves the claim.
\end{proof}

\begin{remark}
    As noted in the proofs, the theorems of this section are actually results about filtered spectra.
    In particular, one can apply these results much more generally than just synthetic spectra: the exact same arguments go through in the setting of an adjunction $\FilSp \rightleftarrows \C$ as discussed in \cref{ssec:filtered_spectra}.
\end{remark}


\subsection{The synthetic Leibniz rule}
\label{ssec:synthetic_leibniz}

In this section we prove a theorem of Burklund describing a synthetic version of the Leibniz rule.
This is an incredibly powerful tool.
For example, in the context of the $\Tmf$-computations at the prime $2$ in \cref{sec:prime2}, we use this synthetic Leibniz rule to compute $d_7(\Delta^4)$ in \cref{prop:dsevens}.
This provides the first unconditional proof of this differential, and answers a question of Isaksen--Kong--Li--Ruan--Zhu \cite[Problem 1.2]{isaksen_etal_motivic_ANSS_tmf}.

Unlike Bruner's power operations in the Adams spectral sequence \cite[Chapter~VI]{H_infty_book} (interpreted synthetically by Burklund--Xu \cite[Section~7.1.5]{burklund_xu_hi3}), which requires an $\bH_\infty$-structure, this result holds for any homotopy-commutative synthetic ring (or even a homotopy-commutative ring in filtered spectra).

\begin{theorem}[Synthetic Leibniz rule, Burklund]\label{syntheticleibnizrule}
    Let $R$ be a homotopy ring  in $\Syn$. For any $n\geq 1$, the map
    \[\delta_n^{2n} \colon \oppi_{*,*}R/\tau^n \to \oppi_{*-1,\, *+n+1} R/\tau^n\]
    is a derivation. In particular, for any two classes $x,y \in \oppi_{\ast,\ast}R/\tau^n$, we have the equality
    \[\delta_n^{2n}(xy) = \delta_n^{2n}(x)\cdot y + (-1)^{\abs{x}}x\cdot  \delta_{n}^{2n}(y).\]
\end{theorem}

\begin{proof}
    By applying $\sigma$, we reduce to proving the same in the context of filtered spectra.
    We first prove the claim in the absolute case $R=\S$.
    Since $C\tau^n$ is a ring spectrum, the left unit $\eta_L\colon C\tau^n\to C\tau^n\otimes C\tau^n$ is split by the multiplication map $\mu\colon C\tau^n\otimes C\tau^n\to C\tau^n$, resulting in a splitting of $C\tau^n$-modules
    \[C\tau^n\otimes C\tau^n\simeq C\tau^n\oplus\opSigma^{1,\,-n-1}C\tau^n,\]
    which determines a splitting
    \begin{equation}\label{eq:derivationsplitting}
    \oppi_{*,*}(C\tau^n\otimes C\tau^n)\cong\oppi_{*,*}(C\tau^n)\oplus\oppi_{*,*}(\Sigma^{1,\,-n-1}C\tau^n)
    \end{equation}
    of $\oppi_{*,*}C\tau^n$-modules.
    The theorem follows from a pair of claims: the projection onto the right hand factor in \eqref{eq:derivationsplitting} coincides with $\id_{C\tau^n}\otimes\delta_{n}^{\infty}$, and \eqref{eq:derivationsplitting} is an isomorphism of rings, where we consider the right-hand side as a trivial square-zero extension.
    Indeed, given these, it follows that the composite
    \[
        \begin{tikzcd}
            \oppi_{*,*}C\tau^n \ar[r,"\eta_R"] & \pi_{*,*}(C\tau^n\otimes C\tau^n) \ar[r,"\id_{C\tau^n}\otimes\delta_{n}^{\infty}"] &[2.1em] \oppi_{*,*}\Sigma^{1,\,-n-1}(C\tau^n)
        \end{tikzcd}
    \]
    is a derivation, and the composite is identified with $\delta_n^{2n}$ via the following commutative diagram.
    \[
    \begin{tikzcd}
        C\tau^n\otimes C\tau^n\arrow[r,"\id_{C\tau^n}\otimes\delta_n^\infty"] &[2.1em] \opSigma^{1,\,-n-1}C\tau^n\\
        C\tau^n\arrow[u,"\eta_R"]\arrow[d,equal]\arrow[r,"\delta_n^\infty"] & \opSigma^{1,\,-n-1}\Sph\arrow[u]\arrow[d]\\
        C\tau^n\arrow[r,"\delta_{n}^{2n}"] & \opSigma^{1,\,-n-1}C\tau^n
    \end{tikzcd}
    \]
   To establish this pair of claims, we need to unravel the splitting of \eqref{eq:derivationsplitting}. First we tensor the cofiber sequence
    \[
     \begin{tikzcd}
        \opSigma^{0,-n}\S\arrow[r,"\tau^n"] & \S\arrow[r] & C\tau^n\arrow[r,"\delta_n^\infty"] & \opSigma^{1,\,-n-1}\S
    \end{tikzcd}
    \]
    on the left with $C\tau^n$ and, examining the corresponding long exact sequence, one sees that there is a unique class $\widehat{\sigma}\in\pi_{1,\,-n-1}(C\tau^n\otimes C\tau^n)$ such that $(\id_{C\tau^n}\otimes\delta_n^\infty)(\widehat{\sigma})=1\in\oppi_{0,0}C\tau^n$.
    We let 
    \[\sigma\colon \opSigma^{1,\,-n-1}C\tau^n\to C\tau^n\otimes C\tau^n\]
    be the map of $C\tau^n$-modules adjoint to $\widehat{\sigma}$. It now follows by definition of $\sigma$ that the above splitting of $C\tau^n$-modules takes the form
    \[
    \begin{tikzcd}
        C\tau^n\otimes C\tau^n \ar[r,shift left, "\begin{pmatrix}
            \mu\\ \id_{C\tau^n}\otimes\delta_n^{\infty}\end{pmatrix}"] &[3em] C\tau^n \oplus \Sigma^{1,\,-n-1}C\tau^n, \ar[l,shift left, "\begin{pmatrix}
                \eta_L \hspace{0.1cm} \sigma
            \end{pmatrix}"]
    \end{tikzcd}
    \]
    which establishes the first of the above claims.

    For the second claim, we examine a product of elements
    
    \adjustbox{scale=0.89,center}{$(\eta_L(x)+\sigma(y))\cdot(\eta_L(x')+\sigma(y'))=\eta_L(x)\eta_L(x')+\sigma(y)\eta_L(y')+(-1)^{|x|}\eta_L(x)\sigma(y')+\sigma(y)\sigma(y')$
    }
    
    in $\pi_{*,*}(C\tau^n\otimes C\tau^n)$. Assuming the last term vanishes, applying the upper map in the above splitting yields the class
    \[
        (xx',yx'+(-1)^{\abs{x}}xy')\in\pi_{*,*}(C\tau^n)\oplus\pi_{*,*}(\opSigma^{1,\,-n-1}C\tau^n),
    \]
    completing the proof. To see that $\sigma(y)\sigma(y')=0$, note that by $C\tau^n$-linearity, this product factors through a map of $C\tau^n$-modules
    \[
        \begin{tikzcd}
            \opSigma^{1,\,-n-1} C\tau^n \otimes_{C\tau^n} \opSigma^{1,\,-n-1}C\tau^n \rar["\sigma\otimes\sigma"] & (C\tau^n\otimes C\tau^n)^{\otimes_{C\tau^n}2} \rar["\mu"] & C\tau^n\otimes C\tau^n
        \end{tikzcd}
    \]
    which is adjoint to a map in $\pi_{2,\,-2n-2}(C\tau^n\otimes C\tau^n)$. This group is seen to be zero by a long exact sequence argument.

    The proof for general $R$ goes through without change by applying $R\otimes-$, with the exception of the last step. The map $R\otimes(\mu\circ(\sigma\otimes\sigma))$ is not automatically zero for degree reasons as in the case $R=\S$, but since the map is obtained by applying $R\otimes-$ to the universal case, the resulting map remains zero.
\end{proof}

\begin{remark}
    It is stated in Claim~3.3.3 of Burklund's Cookware \cite{burklund_cookware_draft} that, in addition, the map $C\tau^{2n}\to C\tau^n$ is a square zero extension of $\E_\infty$-rings. We will not need such a generalisation in this article.
\end{remark}


\section{Toda brackets}
\label{sec:toda_brackets_setup}

In this section we discuss Toda brackets.
These were invented by Toda in \cite{toda} in the category of topological spaces, and were used extensively in \cite{kochman_stable} in the category of spectra.
It has long been known to experts that one can define Toda brackets in any sufficiently coherent homotopical context, for example, in an $\infty$-category.
For simplicity and with our eyes towards applications to synthetic, equivariant, and motivic spectra, as well as the associated categories of modules, we restrict our attention to the Picard-graded homotopy groups of the unit in a monoidal stable $\infty$-category.

First, we set up this abstract theory, including proving shuffling formulas for Toda brackets of arbitrary length, and afterwards discuss a synthetic version of Moss' theorem for evaluating Toda brackets.
Let us reiterate that for the second part, we closely follow ideas of Burklund from \cite{burklund_cookware_draft}.

\subsection{Toda brackets and iterated cones}

Let $(\C,\otimes,\1)$ be a stably monoidal $\infty$-category.

\begin{definition}\label{def:shifts}
    We let a \defi{shift} of an object or morphism in $\C$ refer to applying $X\otimes\blank$ or $\blank\otimes X$ for some $X\in\Pic(\C)$.
\end{definition}

We will abuse notation slightly and use the same name for a map $f:X\to Y$ for $X,Y\in \Pic(\C)$ and its shift $f\otimes Z\colon X\otimes Z\to Y\otimes Z$ for $Z\in \Pic(\C)$. One important example is the dual map $Y^\vee\to X^\vee$, as this can be identified with the shift of $Y^{-1}\otimes f\otimes X^{-1}$ of $f$ up to a sign, in the following sense. 

\begin{definition}\label{def:signs}
    We say that (sets of) morphisms agree up to a \defi{sign} when they become homotopic after multiplying by a unit of the ring $[\1,\1]$ of endomorphisms of the unit of $\C$.
\end{definition}

All of the formulas given below only hold up to a sign, and we do not make these signs explicit. 

\begin{definition}\label{def:formsofcones}
    Let $X_0,\dotsc,X_{n}\in\Pic(\C)$ and
    \[
        \begin{tikzcd}
            X_n \rar["a_n"] & X_{n-1} \rar & \dotsb \rar & X_1 \rar["a_1"] & X_0
        \end{tikzcd}
    \]
    a composable sequence of arrows in $\C$. We define the data of a \defi{form of $C(a_1,\dotsc,a_n)$} in $\C$ inductively as follows:
    \begin{itemize}
        \item A form of $C(a_1)$ is a cofiber of $a_1$ together with the canonical projection $C(a_1)\to \Sigma X_1$. There is a contractible space of forms of $C(a_1)$.
        \item Suppose that forms of $C(a_1,\dotsc,a_{n-1})$ have been defined so that each such form comes with a canonical ``projection to the top cell''
        \[C(a_1,\dotsc,a_{n-1})\to \Sigma^{n-1}X_{n-1}\]
        A form of $C(a_1,\dotsc,a_n)$ is then a cofiber of a morphism
        \[\Sigma^{n-1}X_n\to C(a_1,\dotsc,a_{n-1})\]
        such that the composite
        \[\Sigma^{n-1}X_n\to C(a_1,\dotsc,a_{n-1})\to\Sigma^{n-1}X_{n-1}\]
        is homotopic to $\Sigma^{n-1}a_n$, for some form of $C(a_1,\dotsc,a_{n-1})$. Rotating the cofiber sequence, one obtains a canonical map
        \[C(a_1,\dotsc,a_n)\to\Sigma^nX_n\]
    \end{itemize}
\end{definition}

\begin{remark}\label{rmk:formsofconesdiagram}
    A form of $C(a_1,\dotsc,a_n)$ is a cell complex built out of $X_0,\dotsc,X_n$ which may be depicted as follows
    \[
    \begin{tikzcd}
        \boxed{\Sigma^nX_n}\arrow[d,dash,"a_n"]\\
        \boxed{\Sigma^{n-1}X_{n-1}}\arrow[d,dash]\\
        \vdots\arrow[d,dash]\\
        \boxed{\Sigma X_1}\arrow[d,dash,"a_1"]\\
        \boxed{X_0}
    \end{tikzcd}
    \]
    By definition such a form fits into a canonical cofiber sequence
    \[C(a_1,\dotsc,a_{n-1})\to C(a_1,\dotsc,a_n)\to\Sigma^nX_n\]
    for some form of $C(a_1,\dotsc,a_{n-1})$. Except in the case $n=1$, these do not always exist and are not in general unique when they do.
\end{remark}

\begin{example}\label{lem:dualofcones}
    The dual of a form of $C(a_1,\dotsc,a_n)$ is a shift of a form of $C(a_n,\dotsc,a_1)$, since taking duals preserves cofiber sequences. 
\end{example}

\begin{definition}\label{def:overlinenotation}
    Let $K$ be a form of $C(a_1,\dotsc,a_n)$. For maps $a\colon X\to X_n$ and $b\colon X_0\to X$, we use the notation 
    \begin{align*}
        \SwapAboveDisplaySkip
        \Sigma^nX &\stackrel{\overline{a}}{\longrightarrow} K\\
        K &\stackrel{\underline{b}}{\longrightarrow} X
    \end{align*}
    to denote maps that make the following diagrams commute
    \[
    \begin{tikzcd}
        X_0\arrow[r,"b"]\arrow[d,"\iota"']&X\\
        K\arrow[ur,"\underline{b}"']
    \end{tikzcd}
    \qquad\qquad
    \begin{tikzcd}
        \Sigma^{n}X\arrow[r,"a"]\arrow[d,"\overline{a}"']&\Sigma^{n}X_{n},\\
        K\arrow[ur,"p"']
    \end{tikzcd}
    \]
    where $\iota$ and $p$ are the canonical inclusions and projections respectively.
\end{definition}

\begin{definition}\label{def:todabracket}
    Let $K$ be a form of $C(a_2,\dotsc,a_{n-1})$, and suppose one has maps
    \begin{align*}
        \opSigma^{n-2}X_n &\stackrel{\overline{a_n}}{\longrightarrow} K\\
        K &\stackrel{\underline{a_1}}{\longrightarrow} X_0
    \end{align*}
    as in \cref{def:overlinenotation}.
    The \defi{Toda bracket}
    \[\angbr{a_1,\dotsc,a_n}\subseteq [\Sigma^{n-2}X_n,X_0]\]
    is the set of all composites $\underline{a_1}\circ\overline{a_n}$ formed in this way, by running over all forms $K$ of $C(a_2,\dotsc,a_{n-1})$ and all possible choices of $\underline{a_1}$ and $\overline{a_n}$.
\end{definition}

\begin{remark}\label{rmk:todabracketdiagram}
In a diagram, this is the composite
\[
    \begin{tikzcd}
        \boxed{\Sigma^{n-2}X_n}\arrow[r,"a_n"]&\boxed{\Sigma^{n-2}X_{n-1}}\arrow[d,dash,"a_{n-1}"]\\
        &\vdots\arrow[d,dash,"a_2"]\\
        &\boxed{X_1}\arrow[r,"a_1"]&\boxed{X_0}
    \end{tikzcd}
    \]    
\end{remark}

\begin{lemma}\label{lem:functorspreservebrackets}
    Let $F\colon \C\to\mathcal{D}$ be a unital exact functor between symmetric monoidal stable $\infty$-categories. For any morphisms $a_i\in\C$ as above, there is a containment
    \[F(\angbr{a_1,\dotsc,a_n})\subseteq\angbr{F(a_1),\dotsc,F(a_n)}.\]
\end{lemma}
\begin{proof}
    The assumptions guarantee that $F$ preserves forms of $C(a_1,\dotsc,a_n)$.
\end{proof}

\begin{lemma}
    Up to signs, one has
    \[\angbr{a_1,\dotsc,a_n}=\angbr{a_n,\dotsc,a_1}.\]
\end{lemma}
\begin{proof}
    This follows immediately from the definition using \cref{lem:dualofcones}.
\end{proof}

The set $\angbr{a_1,\dotsc,a_n}$ may be empty. Indeed, a form of $C(a_2,\dotsc,a_{n-1})$ need not exist when $n>3$, nor the maps $\underline{a_1}$ and $\overline{a_n}$. We therefore give some general existence statements for forms of iterated cones and non-emptiness of Toda brackets. In the following, a \emph{sub-bracket} of $\angbr{a_1,\dotsc,a_n}$ refers to any bracket of the form $\angbr{a_i,a_{i+1},\dotsc,a_{i+k}}$.

\begin{proposition}\label{prop:conesandbracketsexistence}
    \leavevmode
    \begin{numberenum}
        \item A form of $C(a_1,a_2,a_3)$ exists if and only if $a_1a_2=a_2a_3=0$ and $0\in\angbr{a_1,a_2,a_3}$. The Toda bracket $\angbr{a_1,a_2,a_3}$ is nonempty if and only if $a_1a_2=a_2a_3=0$.
        \item Let $n>3$. A form of $C(a_1,\dotsc,a_n)$ exists if and only if all sub-brackets of $\angbr{a_1,\dotsc,a_n}$ contain zero.
        \item Let $n>3$. If $\angbr{a_1,\dotsc,a_k}=\set{0}$ for all $k<n$,
        and $0\in\angbr{a_2,\dotsc,a_n}$, then $\angbr{a_1,\dotsc,a_n}$ is nonempty. Conversely, if $\angbr{a_1,\dotsc,a_n}$ is nonempty, then all sub-brackets of $\angbr{a_1,\dotsc,a_n}$ of length $<n$ contain zero. 
    \end{numberenum}
\end{proposition}
\begin{proof}
    For item~(1), if a form of $C(a_1,a_2,a_3)$ exists, then a form of $C(a_1,a_2)$ exists, and taking the fiber of $C(a_1)\to C(a_1,a_2)$ gives a map $\Sigma X_2\to C(a_1)$ whose projection onto the top cell gives $a_2$. A simple long exact sequence argument using the cofiber sequence
    \[
        \begin{tikzcd}
            X_1\rar["a_1"] & X_0 \rar & C(a_1)
        \end{tikzcd}
    \]
    shows that this implies $a_1a_2=0$, and similarly for $a_2a_3$. To see that $0\in\angbr{a_1,a_2,a_3}$, consider the cofiber sequence
    \[\Sigma^2X_3\to C(a_1,a_2)\to C(a_1,a_2,a_3)\]
    Taking the dual of the first map gives, up to a shift, a map $C(a_2,a_1)\to X_3$ whose restriction to the bottom cell is $a_3$. Since the restriction of this map to $C(a_2)$ extends over $C(a_2,a_1)$, its precomposition with the attaching map $\Sigma X_0\to C(a_2)$ of the top cell in $C(a_2,a_1)$ defines an element of $\angbr{a_1,a_2,a_3}$ that is nullhomotopic.

    Conversely, if $a_1a_2=a_2a_3=0$, we have diagrams \begin{equation}\label{eq:3folddiagram}
    \begin{tikzcd} X_1\arrow[r,"a_1"]\arrow[d,"\iota"']&X_0\\
    C(a_2)\arrow[ur,dashed,"\underline{a_1}"']
    \end{tikzcd}
    \qquad \qquad
    \begin{tikzcd}
        \Sigma X_3\arrow[r,"a_3"]\arrow[d,dashed,"\overline{a_3}"']&\Sigma X_{2}.\\
        C(a_2)\arrow[ur,"p"']
    \end{tikzcd}
    \end{equation}
    Moreover, for some choices of $\underline{a_1}$ and $\overline{a_3}$, there is a nullhomotopy of the composite $\underline{a_1}\circ\overline{a_3}$, which provides an extension of $\underline{a_1}$ over the cofiber of $\overline{a_3}$, which is a form of $C(a_2,a_3)$. Taking the fiber of this extension $C(a_2,a_3)\to X_0$ gives (up to a shift) a form of $C(a_1,a_2,a_3)$.
    
    The Toda bracket $\angbr{a_1,a_2,a_3}$ is nonempty if and only if we have diagrams as in \cref{eq:3folddiagram}, which exist if and only if $a_1a_2=a_2a_3=0$. 
    
    Item~(2) is proved via a straightforward induction argument with base case the $n=3$ case just proved.
    
    For item~(3), suppose $\angbr{a_1,\dotsc,a_{n-1}}=\set{0}$ and $0\in\angbr{a_2,\dotsc,a_n}$.
    Using item~(2), we have a cofiber sequence $K\to K'\to\Sigma^{n-1}X_n$, where $K$ is a form of $C(a_2,\dotsc,a_{n-1})$ and $K'$ is a form of $C(a_2,\dotsc,a_{n})$.
    Rotating gives a map $\overline{a_n}\colon \Sigma^{n-2}X_n\to K$.
    Consider the question of extending the map $a_1\colon X_1\to X_0$ over $K$ inductively by providing nullhomotopies of each of the attaching maps of $K$. The postcomposition of the attaching map of the $k$-th cell of $K$ with an extension of $a_1$ over a $(k-1)$-skeleton defines an element of $\angbr{a_1,\dotsc,a_{k+2}}$.
    If these brackets are equal to $\set{0}$ for $k+2<n$, then we may extend over $K$ and precomposing with $\overline{a_n}$ defines an element of $\angbr{a_1,\dotsc,a_n}$.
    
    Conversely, if $\angbr{a_1,\dotsc,a_n}$ is nonempty, then a form of $C(a_2,\dotsc,a_{n-1})$ exists. From item~(2), it follows all sub-brackets of $\angbr{a_2,\dotsc,a_{n-1}}$ contain zero. Moreover, there is an extension of $a_1$ over a form of $C(a_2,\dotsc,a_{n-1})$, which by analyzing attaching maps in $C(a_2,\dotsc,a_{n-1})$ as above shows that each of the brackets $\angbr{a_1,\dotsc,a_k}$ contains zero for $k<n$. Dualizing the argument yields the same for the brackets $\angbr{a_k,\dotsc,a_n}$ for $k>1$.
\end{proof}

\begin{proposition}
    The following shuffling formulas hold up to sign.
    \label{prop:shufflingformulas}
    \leavevmode
    \begin{numberenum}
        \item If $\angbr{a_i,
        \dotsc,a_n}=\angbr{a_1,\dotsc,a_i}=\set{0}$ for all $3
        \le i<n-1$, then
        \[a_1\angbr{a_2,\dotsc,a_n}=\angbr{a_1,\dotsc,a_{n-1}}a_n.\]
        \item Let $a\colon X\to X_n$ be a map. Then $\angbr{a_1,\dotsc,a_n}a\subseteq\angbr{a_1,\dotsc,a_na}.$
        \item Let $a\colon X\to X_n$ be a map. Then
        \[\angbr{a_1,\dotsc,a_na}\subseteq \angbr{a_1,\dotsc,a_{n-1}a,a_n}.\]
        \item Let $2<k<n$ and $a:X_{k-1}\to X$ be a map.
        If all sub-brackets of $\angbr{a_2,\dotsc,a_{k-1}a,a_k,\dotsc,a_n}$ contain zero and all sub-brackets of $\angbr{a_1,\dotsc,a_{k-1},aa_k,\dotsc,a_{n-1}}$ are equal to $\set{0}$, then 
        \[\angbr{a_1,\dotsc,a_{k-1}a,a_k,\dotsc,a_n}\cap \angbr{a_1,\dotsc,a_{k-1},aa_k,\dotsc,a_n}\neq\varnothing.\]
    \end{numberenum}
\end{proposition}

\begin{proof}
For item (1), we will show that $a_1\angbr{a_2,\dotsc,a_n}\supseteq\angbr{a_1,\dotsc,a_{n-1}}a_n$ up to signs; dualizing the argument gives the other inclusion.
Suppose one is given an element
\[
    \begin{tikzcd}
        \Sigma^{n-3}X_{n-1}\rar["\overline{a_{n-1}}"] & C(a_2,\dotsc,a_{n-2}) \rar["\underline{a_1}"] &X_0
    \end{tikzcd}
\]
in $\angbr{a_1,\dotsc,a_{n-1}}$. One forms a commutative diagram
\[
\begin{tikzcd}
    C(a_3,\dotsc,a_{n-1})\arrow[d,"\underline{a_2}"]\arrow[r]&\Sigma^{n-3}X_{n-1}\arrow[d,"\overline{a_{n-1}}"]\arrow[r]&\Sigma C(a_3,\dotsc,a_{n-2})\arrow[d,equal]\\
    X_1\arrow[r]&C(a_2,\dotsc,a_{n-2})\arrow[r]&\Sigma C(a_3,\dotsc,a_{n-2})
\end{tikzcd}
\]
by taking fibers horizontally in the right square, so that the rows are cofiber sequences.
By \cref{prop:conesandbracketsexistence}, if $\angbr{a_i,\dotsc,a_n}=\set{0}$ for all $3\le i<n-1$, then there exists a factorization
\[
\begin{tikzcd}
    &\Sigma^{n-3}X_n\arrow[d,"a_n"]\arrow[dl,dashed,"\overline{a_n}"']\\
    C(a_3,\dotsc,a_{n-1})\arrow[r]&\Sigma^{n-3}X_{n-1}
\end{tikzcd}
\]
for the form of $C(a_3,\dotsc,a_{n-1})$ constructed above. Pasting the two diagrams above, one forms a commutative diagram
\[
\begin{tikzcd}
\Sigma^{n-3}X_n\arrow[dr,"\overline{a_n}"']\arrow[r,"a_n"]&\Sigma^{n-3}X_{n-1}\arrow[r,"\overline{a_{n-1}}"]&C(a_2,\dotsc,a_{n-2})\arrow[r,"\underline{a_1}"]&X_0\\
&C(a_3,\dotsc,a_{n-1})\arrow[r,"\underline{a_2}"']&X_1\arrow[u]\arrow[ur,"a_1"']
\end{tikzcd}
\]
proving the inclusion.    

Item (2) follows immediately from the definition. 

For item (3), suppose we are given an element
\[
    \begin{tikzcd}
        \opSigma^{n-2}X \rar["\overline{aa_n}"] & C(a_2,\dotsc,a_{n-1}) \rar["\underline{a_1}"] & X_0
    \end{tikzcd}
\]
in $\angbr{a_1,\dotsc,aa_n}$. By starting with the upper left-hand square and taking cofibers, one produces the commutative diagram
    
\adjustbox{scale=0.89,center}{
\begin{tikzcd}
    X\otimes\Sigma^{n-3}X_{n-1}\arrow[d,"a"]\arrow[r]&X_n\otimes C(a_2,\dotsc,a_{n-2})\arrow[d,equal]\arrow[r]&X_n\otimes C(a_2,\dotsc,aa_{n-1})\arrow[d]\arrow[r]&X\otimes\Sigma^{n-2}X_{n-1}\arrow[d,"a"]\\
    X_n\otimes\Sigma^{n-3}X_{n-1}\arrow[d]\arrow[r,"\overline{a_{n-1}}"]&X_n\otimes C(a_2,\dotsc,a_{n-2})\arrow[r]\arrow[d]&X_n\otimes C(a_2,\dotsc,a_{n-1})\arrow[r]\arrow[d]&X_n\otimes \Sigma^{n-2}X_{n-1}\arrow[d]\\
    C(a)\otimes\Sigma^{n-3}X_1\arrow[r]&0\arrow[r]&C\arrow[r,"\sim"]&C(a)\otimes\Sigma^{n-2}X_{n-1}
\end{tikzcd}
}

where the rows and columns are cofiber sequences. This shows that to lift our map $\overline{aa_n}\colon \Sigma^{n-2}X\to C(a_2,\dotsc,a_{n-1})$ to $C(a_2,\dotsc,aa_{n-1})$, we need to show that postcomposing it with the projection to the top cell followed by inclusion into $C(a)$ gives zero. But the projection to the top cell is by definition $aa_{n-1}$, which is divisible by $a$. This yields a factorization similar to that of item~(1).

For item~(4), if all sub-brackets of $\angbr{a_2,\dotsc,a_{k-1}a,a_k,\dotsc,a_n}$ contain zero, then by \cref{prop:conesandbracketsexistence}\,(2), there exists a map $\overline{a_n}\colon \Sigma^{n-2}X_n\to C(a_2,\dotsc,a_{k-1}a,a_k,\dotsc,a_{n-1})$. In a manner similar to the above, one can construct a map
\[C(a_2,\dotsc,a_{k-1}a,a_k,\dotsc,a_{n-1})\to C(a_2,\dotsc,a_{k-1},aa_k,\dotsc,a_{n-1})\]
that extends the inclusion of the bottom cell and projects to the identity on the top cell, for some form of $C(a_2,\dotsc,a_{k-1},aa_k,\dotsc,a_{n-1})$. Using the arguments of \cref{prop:conesandbracketsexistence}, the map $a_1$ extends over the above form of $C(a_2,\dotsc,a_{k-1},aa_k,\dotsc,a_{n-1})$ if all sub-brackets of $\angbr{a_1,\dotsc,a_{k-1},aa_k,\dotsc,a_{n-1}}$ are equal to $\set{0}$. Composing these maps gives an element in the claimed intersection.
\end{proof}

\subsection{Moss' theorem}
\label{ssec:moss}
Moss' theorem gives conditions for Toda brackets in $\pi_*R$ to be detected by Massey products in a spectral sequence converging to $\pi_*R$. The classical reference is \cite{moss_theorem}; also see \cite{belmontTodaBracketConvergence2024} for a modern reinterpretation.
We will discuss a version of Moss' theorem in the setting of synthetic spectra.
Throughout this section, we fix a homotopy-associative ring spectrum $E$ of Adams type, and write $\Syn_E$ for $\Syn$.

Let $R\in \Syn$ be a synthetic $\E_2$-ring, then we may work with Toda brackets in the monoidal $\infty$-category of (left) modules over $R$.
Given a Toda bracket
\[\angbr{a_1,\dotsc,a_n}\subseteq\oppi_{*,*}R,\]
we will describe a general technique for determining the image of $\angbr{a_1,\dotsc,a_n}$ along the map $\pr_r\colon \pi_*R\to\oppi_*R/\tau^r$, following ideas of Robert Burklund \cite{burklund_cookware_draft}.

There is one immediate observation.

\begin{proposition}\label{prop:modtaupreservesbrackets}
    There is an inclusion
    \[\pr_r(\angbr{a_1,\dotsc,a_n})\subseteq\angbr{a_1,\dotsc,a_n}\subseteq\oppi_{*,*}R/\tau^r.\]
\end{proposition}
\begin{proof}
    The map $R\to R/\tau^r$ is one of $\E_2$-rings, so the claim follows from \cref{lem:functorspreservebrackets}.
\end{proof}

Suppose $C_*$ is a dga in an abelian category $\calA$.
Recall that, when we regard $C_*$ as an object of the derived $\infty$-category $\D(\calA)$, then Toda brackets in the homotopy groups of a dga $C_*$ coincide with Massey products in $C_*$.
This immediately implies the following.

\begin{proposition}\label{prop:todagoestomassey} 
If the $\E_2$-ring $R/\tau$ is the underlying object in $\Stable$ of a dga $C_*$, then the set $\pr_1(\angbr{a_1,\dotsc,a_n})$ is contained in the Massey product $\angbr{a_1,\dotsc,a_n}\subseteq \uH_*C_*$.
\end{proposition}

However, this is not particularly useful if one of the $a_i$'s is divisible by $\tau^r$ for $r\geq 1$.
In that case, the bracket in $R/\tau^r$ will contain zero and have large indeterminacy generically. 

Instead of proving a version of Moss' theorem in this context, we describe a general approach to determining the image $\pr_r(\angbr{a_1,\dotsc,a_n})$ that works well in most cases. We outline how this works in a simple situation that happens to cover all of our applications in this article; see \cref{sec:large_carrick_brackets}.

Suppose that $\pr_r(a_1)$ and $\pr_r(a_n)$ are nonzero, and that $a_2,\dotsc,a_{n-1}$ are all divisible by $\tau^r$; write $b_i\tau^r=a_i$.
An element in $\angbr{a_1,\dotsc,a_n}$ is given by a composite  
\[
    \begin{tikzcd}
        \bullet\arrow[r,"\overline{a_n}"]&\bullet\arrow[d,dash,"b_{n-1}\tau^r"]\\
        &\vdots\arrow[d,dash,"b_2\tau^r"]\\
        &\bullet\arrow[r,"\underline{a_1}"']&\bullet
    \end{tikzcd}
    \]  
where we suppress all degrees for simplicity. All primary attaching maps in this form of $C(a_2,\dotsc,a_{n-1})$ become zero after tensoring down to $C\tau^r$, and we assume for simplicity that the attaching maps themselves are divisible by $\tau^r$. It follows that the complex splits as a sum of shifts of the unit after tensoring down to $C\tau^r$. The maps $\overline{a_n}$ and $\underline{a_1}$ are then represented by matrices 
\[
    \overline{a_n} = \begin{pmatrix}
        x_{n}\\
        \vdots\\
        x_1
    \end{pmatrix} \qquad \text{and} \qquad 
    \underline{a_1} = \begin{pmatrix}
        y_n&\cdots&y_1
    \end{pmatrix}
\]
with entries in $\oppi_{*,*}C\tau^r$, where $x_n=\pr_r(a_n)$ and $y_1=\pr_r(a_1)$.
Assuming for simplicity that $y_i=0$ for all $i>1$ (this happens for degree reasons in our examples), then this composite projects to the element $\pr_r(a_1)x_1\in \oppi_{*,*}C\tau^r$. 

We therefore reduce to understanding the component $x_1$, which may be determined inductively using that $x_n=\pr_r(a_n)$ and working our way down the cells of $C(a_2,\dotsc,a_{n-1})$.
Assuming the brackets $\angbr{b_2,a_3,\dotsc,a_{n-1}}$ and $\angbr{a_2,a_3,\dotsc,a_{n-1}}$ are both equal to $\set{0}$, the shuffling formula
\[\tau^r\angbr{b_2,a_3,\dotsc,a_{n-1}}\subseteq\angbr{a_2,a_3,\dotsc,a_{n-1}}\]
allows us to form a diagram
\[
\begin{tikzcd}
    C(a_3,\dotsc,a_{n-1})\arrow[r,"\underline{a_2}"]\arrow[d,"\underline{b_2}"']&\1\arrow[d,equal]\\
    \1\arrow[r,"\tau^r"]&\1
\end{tikzcd}\]
where we suppress degrees as before. Taking fibers horizontally, we form a commutative square
\[
\begin{tikzcd}
    C(a_2,\dotsc,a_{n-1})\arrow[r]\arrow[d]&C(a_3,\dotsc,a_{n-1})\arrow[d,"\underline{b_2}"]\\
    C\tau^r\arrow[r]&\1
\end{tikzcd}
\]
where the horizontal maps collapse the respective bottom cells, and restricting the left-hand map to $C(a_2)$ gives a map $C(a_2)\to C\tau^r$ that is the identity on the bottom cell and $b_1$ on the top cell. The composite
\[
    \begin{tikzcd}
        \1 \rar["\overline{a_n}"] & C(a_2,\dotsc,a_{n-1})\rar & C\tau^r
    \end{tikzcd}
\]
therefore induces a $(2\times 1)$-matrix after tensoring with $C\tau^r$ whose bottom component is $x_1$, and whose top component is $\pr_r(b_1)x_2$.

In other words, we have $\delta_r^{2r}(x_1)=\pr_r(b_1)x_2$.
Using our understanding of total differentials, this refers to a specific differential, hence $x_1$ may be determined up to indeterminacy in the kernel of this differential, once the class $\pr_r(b_1)x_2$ is known. However, $x_2$ is the bottom component of the composite 
\[
    \begin{tikzcd}
        \1 \rar["\overline{a_n}"] & C(a_2,\dotsc,a_n)\rar &  C(a_3,\dotsc,a_n)
    \end{tikzcd}
\]
in the above diagram, which is of the form $\overline{a_n}$. The process may thus be repeated with this composite, and inductively we are reduced to understanding the composite
\[
    \begin{tikzcd}
        \1 \rar["\overline{a_n}"] & C(a_2,\dotsc,a_n)\rar & \1
    \end{tikzcd}
\]
where the right-hand map is projection to the top cell, and this composite is $a_n$ by assumption.

For three-fold Toda brackets, this method works in a straightforward manner, and we can give the following synthetic version of Moss' theorem.

\begin{theorem}
    \label{thm:synthetic_moss}
    Let $R$ be an $\E_2$-ring in $\Syn$, and let $a_1,a_2,a_3\in\oppi_{*,*} R$ such that $\tau^r a_1a_2 = \tau^s a_2a_3= 0$ in $\oppi_{*,*}R$ for $r \geq s \geq 0$, and let $r,s$ be minimal with respect to this property.
    \begin{numberenum}
        \item Suppose that $r,s> 0$.
        Then there exist $H_0,H_1\in\oppi_{*,*}(R/\tau^s)$ such that 
        \begin{align*}
        d_{r+1}(H_0)&=a_1a_2\\
        d_{s+1}(H_1)&=a_2a_3,
        \end{align*}
        and the Toda bracket $\angbr{a_1,\tau^r a_2, a_3} \subseteq \oppi_{*,*}R$ contains a lift of
        \[
            \tau^{r-s} a_1H_1 \pm H_0a_3 \in \pi_{*,*}(R/\tau^{s}).
        \]
        Note that here we do not change the name of an element along either of the maps in the composite
        \[R\to R/\tau^n\to R/\tau\]
        for $n=r,s$.
        
        In particular, if $r=s$, then the Toda bracket $\angbr{a_1,\tau^r a_2, a_3}$ contains a lift of the element $a_1H_1 \pm H_0a_3$ in $\pi_{*,*}(R/\tau)$.
        If $r>s$, then the Toda bracket $\angbr{a_1,\tau^r a_2, a_3}$ contains a lift of the element $H_0a_3$ in $\pi_{*,*}(R/\tau)$.
        {\upshape(}Note that in both cases, these linear combinations are contained in the Massey product $\angbr{a_1,a_2,a_3}$ formed in the dga $(\uE_{r+1},d_{r+1})$ in $\sigma R$.{\upshape)}
        
        \item Suppose that $r > 0$ while $s = 0$.
        Then there exists $H_0\in\pi_{*,*}(R/\tau^r)$ such that 
        \[
        d_{r+1}(H_0)=a_1a_2
        \]
        and the Toda bracket $\angbr{a_1,\tau^r a_2, a_3} \subseteq \pi_{*,*}R$ contains a lift of
        \[
            H_0a_3\in \pi_{*,*}(R/\tau^{r}).
        \]
        In particular, the Toda bracket contains a lift of an element in the Massey product $\angbr{a_1,a_2,a_3}$ formed in the dga $(\uE_{r+1},d_{r+1})$ in $\sigma R$.
    \end{numberenum}
\end{theorem}

\begin{proof}
    We give the proof of (1), as (2) is similar and easier. Since $\tau^r a_1a_2 = \tau^s a_2a_3= 0$, we may choose maps $\overline{a_3}\colon R\to R/\tau^sa_2$ and $\underline{\tau^{r-s}a_1}\colon R/\tau^sa_2\to R$, and form the following $R$-linear composite (suppressing all shifts for ease of notation).
    \[
    \begin{tikzcd}
        \bullet\arrow[r,"\overline{a_3}"]&\bullet\arrow[d,dash,"\tau^sa_2"]\\
        &\bullet\arrow[r,"\underline{\tau^{r-s}a_1}"']&\bullet
    \end{tikzcd}
    \]
    Tensoring with $C\tau^s$, we have an $R/\tau^s$-linear composite
    \[
    \begin{tikzcd}
        \bullet\arrow[r,"a_3"]\arrow[dr,"H_1"']&\bullet\arrow[dr,"H_0"]\\
        &\bullet\arrow[r,"\tau^{r-s}a_1"']&\bullet
    \end{tikzcd}
    \]
    which gives the sum $\tau^{r-s} a_1H_1+H_0a_3$. To determine $H_1$, as in the above discussion, we may form an $R$-linear map $R/\tau^sa_2\to R/\tau^s$ that is the identity on the bottom cell and $a_2$ on the top cell.
    The postcomposition of $\overline{a_3}$ with this map defines a map
    \[R\xrightarrow{\overline{a_2a_3}}R/\tau^s,\]
    which, after tensoring with $C\tau^s$, has component $H_1$ onto the bottom cell.
    The $R/\tau^s$-linear map $H_1$ is therefore adjoint to a $R$-linear map $R\to R/\tau^s$ (which we also denote by $H_1$) with the property that $\delta_{s}^{2s}(H_1)=a_2a_3$.
    Since $s$ is minimal with respect to the property that $\tau^sa_2a_3=0$, by \cref{prop:total_differential_versus_differentials} this implies the differential $d_{s+1}(H_1)=a_2a_3$. The analogous considerations apply to $H_0$ after taking $R$-linear duals and shifting, which introduces a sign depending on the degrees we have suppressed, so we do not attempt to determine this sign.
\end{proof}

\begin{remark}
    Usually, any version of Moss' theorem contains delicate assumptions about crossing differentials. This does not appear here because we are only working synthetically and we assume that we have elements $a_i\in\oppi_{*,*}R$ such that $\tau^ra_1a_2=\tau^sa_2a_3=0$. If, on the other hand, we started with elements $b_i\in\oppi_*\tau^{-1}R$ such that $b_1b_2=b_2b_3=0$, with $d_{r+1}(H_0)=b_1b_2$ and $d_{s+1}(H_1)=b_2b_3$ in the spectral sequence $\sigma R$, then it is not automatic that there exists lifts $a_i$ of $b_i$ along $\tau^{-1}$ so that both $\tau^ra_1a_2$ and $\tau^sa_2a_3$ are zero.
\end{remark}

\begin{remark}
    The theorem above gives conditions for when certain Toda brackets in $\oppi_{*,*}R$ contain lifts of certain elements in $\pi_{*,*}(R/\tau^s)$. We will also need similar statements about lifts along truncation maps of the form 
    \[
    \pi_{*,*}(R/\tau^{s+k})\to\pi_{*,*}(R/\tau^s)
    \]
    The theorem does not apply verbatim to this case as $(R/\tau^{s+k})/\tau^s\not\simeq R/\tau^s$. However, there is essentially no difference between these for the purposes of the theorem. Indeed, there is a canonical splitting
    \[(R/\tau^{s+k})/\tau^s\simeq R/\tau^s \oplus \opSigma^{1,\,-s-k-1}R/\tau^s.\]
    Moreover, the spectral sequence
    arising from the $\tau$-adic tower that begins with the homotopy groups of $(R/\tau^{s+k})/\tau^s$ and converges to the homotopy groups of $R/\tau^{s+k}$ splits in an analogous way into two copies of the truncated Bockstein spectral sequence beginning with $\pi_{*,*}(R/\tau^s)$ and converging to $\pi_{*,*}(R/\tau^{s+k})$. In particular, the differentials in one of these two spectral sequences determine the other, hence the same is true of the Massey products formed in the spectral sequences.
\end{remark}


\section{Relationship to the cubic Hopf algebroid}
\label{sec:cubic_hopf_algebroid}

The $\uE_2$-page of the DSS for $\Tmf$ is computed via the sheaf cohomology of the line bundles $\omega^{\otimes k}$ over the compactification of the moduli stack of elliptic curves $\Mellbar$. A small but crucial part of this $\uE_2$-page is captured by the cohomology of the \emph{cubic Hopf algebroid}.


\subsection{The connective region}\label{e2pagesection}

If we knew the Gap Theorem, then we could apply \cite[Corollary 5.3]{mathew_homology_tmf} to see that the cohomology of this Hopf algebroid is the $\uE_2$-page of the ANSS for $\tmf$, but it would be circular to do that here. Instead, we note that various localisations of the cubic Hopf algebroid can be identified with localisations of $\oppi_{\ast,\ast}\Smf/\tau$. Gluing these idenfications together gives us the desired comparison; see \cref{etwopagebracketsprimethree}.

\begin{definition}
    Define the (graded) \defi{cubic Hopf algebroid} $(A,\Ga)$ as the data
    \[
        A  =  \Z[a_1, a_2, a_3, a_4, a_6],\qquad \abs{a_i}=2i,
    \]
    \[
            \Ga = A[r,s,t]    , \qquad         \abs{r}=4, \quad \abs{s}=2, \quad \abs{t}=6,\]
    equipped with the left unit $\eta_L\colon A \to \Ga$ given by the standard inclusion, the right unit $\eta_R\colon A \to \Ga$ and comultiplication $\psi\colon \Ga \to \Ga\otimes_A \Ga$ defined as in \cite[Section~3]{bauer_tmf}.
\end{definition}

Bauer gives simplifications of this Hopf algebroid when one inverts $2$ or $3$; see \cite[Section~4]{bauer_tmf}.

Recall the equivalence between Hopf algebroids and algebraic stacks, as given in \cite[Theorem~8]{naumann_stacks} for example.
Under this equivalence, the stack associated to the cubic Hopf algebroid is the \emph{moduli stack of cubic curves} $\M_{\cub}$.

There is a map $\M_\cub \to \M_\fg$ arising from a map of the corresponding Hopf algebroids; see \cite[Section~3]{bauer_tmf}.
Taking the derived pushforward of the structure sheaf along this map allows us to view $(A,\Gamma)$ as defining an $\E_\infty$-algebra in $\D(\mathrm{grComod}_{\MU_*\MU})$.
Through the equivalence
\begin{equation}
    \label{eq:recall_syn_MU_mod_tau_derived}
    L_{\nu \MU}\Mod_{C\tau}(\Syn_\MU) \simeq \D(\mathrm{grComod}_{\MU_*\MU})
\end{equation}
from \cite[Theorem~4.5.4]{pstragowski_synthetic}, we can thus regard $(A,\Gamma)$ as an $\E_\infty$-algebra in $C\tau$-modules in $\Syn_\MU$.
We will denote this by $(A,\Gamma)$ again in this section.

Our goal now is to import Bauer's computations about $(A,\Gamma)$ to $\Smf/\tau$.
As we would also like to import Bauer's Massey products, we are interested in these objects not just as $C\tau$-modules, but rather as $\E_\infty$-algebras over $C\tau$.

Through the aforementioned equivalence \eqref{eq:recall_syn_MU_mod_tau_derived}, our previous result \cite[Corollary~2.12]{CDvN_part1} identifies the $\E_\infty$-$C\tau$-algebra $\Smf/\tau$ with the derived pushforward of the structure sheaf of $\Mellbar$ along $\Mellbar \to \M_\fg$.
As a result, the map of stacks $\Mellbar \to \M_\cub$ results in a map of $\E_\infty$-$C\tau$-algebras
\[
    \Phi \colon (A,\Gamma) \to \Smf/\tau.
\]
The works of Bauer \cite{bauer_tmf} and Konter \cite{konter_Tmf} explicitly computing the cohomology of these stacks may be summarised as follows.

\begin{proposition}[\cite{bauer_tmf,konter_Tmf}]\label{etwopagebracketsprimethree}
    The morphism of $\E_\infty$-$C\tau$-algebras
    \[\Phi\colon (A, \Ga) \to \Smf/\tau\]
    induces a split injection on bigraded homotopy groups and an isomorphism on $\pi_{n,s}$ for $5s \leq n+12$.
\end{proposition}

If one inverts $2$, then the map $\Phi$ is an isomorphism in all bidegrees where the cohomology of $(A,\Gamma)$ is nonzero.
This is not the case at the prime $2$, due to the existence of $h_1$-towers.
We refer to either \cite[Figures~10 and~25]{konter_Tmf} or \cref{ssprime3,etwoprime2} below for a graphic demonstration of this at the primes $3$ and $2$, respectively.

This corollary is particularly useful, as we can now interpret the Massey product computations of Bauer in the cohomology of $(A,\Ga)$ as happening in the bigraded homotopy groups of $\Smf/\tau$. We can also use Hopf algebroid computations of \emph{transfer maps}.


\subsection{Synthetic transfer maps}

The main theorem of \cite{hill_lawson_level_structure} states that there exists a log étale\footnote{The use of the adjective \emph{log étale} is only to accommodate the map of stacks $\Mbar_1(3) \to \Mellbar\times \Spec \Z[\tfrac{1}{3}]$, which is not finite étale, but is finite log étale. This family of examples is the reason for Hill--Lawson's extension of $\calO^\top$ from the small étale site of $\Mellbar$ to its small log étale site found in \cite{hill_lawson_level_structure}.} sheaf of $\E_\infty$-rings $\calO^\top\colon (\Mellbar^{\mathrm{log}\text{-}\et})^\op\to \CAlg$ whose global sections are $\Tmf$. If a map of stacks $f\colon \X \to \Y$ in this site is \emph{finite} log-étale, there exists a transfer map
\[f_! \colon \OO^\top(\X) \to \OO^\top(\Y)\]
in the $\infty$-category of $\OO^\top(\Y)$-modules. The existence of such maps can be found in \cite[Remark~1.16]{davies_hecke_tmf}. We would like a synthetic version of these maps.

\begin{theorem}\label{thm:synthetictransfers}
    Write $\calO^\syn$ for the sheafification of the composite
    \[
        \begin{tikzcd}
            (\Mellbar^{\mathrm{log}\text{-}\et})^\op\rar["\OO^\top"] & \CAlg \rar["\nu"] & \CAlg(\Syn_\MU).
        \end{tikzcd}
    \]
    The functor $\OO^\syn$ post-composed with the forgetful functor to $\Syn_\MU$ admits a unique factorisation 
    \[\bO^\syn\colon \Span(\Mellbar^{\mathrm{log}\text{-}\et},\,  \all,\,  \fin) \to \Syn_\MU\]
    through the span $\infty$-category where the backwards maps are all maps of stacks in the small log \'{e}tale site for $\Mellbar$, and the forward maps are only the finite log \'{e}tale maps of stacks.
    Moreover, given a finite étale morphism $f\colon \X \to \Y$ in this site, then the map $\bO^\syn(f)=f_!\colon \OO^\syn(\X) \to \OO^\syn(\Y)$ induces the algebraic transfer map \textbr{see \textup{\cite[Tag~\href{https://stacks.math.columbia.edu/tag/03SH}{03SH}]{stacks_project}}} on the $\uE_2$-page of synthetic DSS \textbr{see \textup{\cite[Section~3.1]{CDvN_part1}}}.
    In particular, the effect of $f_!$ on mod $\tau$ bigraded homotopy groups is the algebraic transfer map on the $\uE_2$-page of DSS.
\end{theorem}

As with the proof of \cite[Theorem~1.14]{davies_hecke_tmf}, we obtain the above result by directly applying \cite[Corollary~C.13]{bachmann_hoyois}.

\begin{proof}
    Consider the composite
    \[
        \begin{tikzcd}
            (\Mellbar^{\mathrm{log}\text{-}\et})^\op \rar["\OO^\syn"] & \CAlg(\Syn_\MU) \rar & \Syn_\MU
        \end{tikzcd}
    \]
    equipped with the wide subcategory of the domain spanned by finite log étale morphisms.
    Here we consider $\Syn_\MU$ as the $\infty$-category of $\E_\infty$-algebras in $\Syn_{\MU}$ with respect to the cartesian monoidal structure (a.k.a.\ $\E_\infty$-monoids in $\Syn_\MU$); as $\Syn_{\MU}$ is stable, this is indeed equivalent to $\Syn_{\MU}$.
    We claim this data satisfies the conditions of \cite[Corollary~C.13]{bachmann_hoyois}. Indeed, the domain category $\calC$ is extensive, meaning it admits finite coproducts, coproducts are disjoint, and finite coproducts decompositions are preserved under base-change, as this is true for any slice category of stacks over a base stack which is closed under coproducts and summands. The other hypotheses hold due to the fact that finite log étale morphisms are log étale locally finite disjoint unions. By \cite[Corollary~C.13]{bachmann_hoyois}, we obtain a unique lifting of the above composite to the desired functor
    \[\bO^\syn \colon \Span(\Mellbar^{\mathrm{log}\text{-}\et},\,  \all,\,  \fin) \to \Syn_\MU.\]
    Given an object $X$ in the above site, the $\uE_2$-page of the DSS for $\OO^\syn(\X)$ takes the form
    \[\uE_2^{k,\ast,s} \cong \uH^s(\X,\, \omega_\X^{\otimes (k+s)/2}) \otimes \Z[\tau]\]
    by \cite[Construction~3.1]{CDvN_part1}.
    (Although \cite{CDvN_part1} deals primarily with the étale topology, the same arguments hold in our log étale case here; see \cite[Variant~4.11]{CDvN_part1}.)
    Copying the proof of \cite[Proposition~1.18]{davies_hecke_tmf}, which proves that these spectral transfer maps induce the algebraic transfer maps on classical DSSs, we can identify the effect of the synthetic transfer maps on the $\uE_2$-page of the synthetic DSS with the algebraic transfer maps base changed over $\Z[\tau]$.
\end{proof}

Using the finite log étale map $p\colon \Mbar_1(3) \times \Spec \Z_{(2)} \to \Mellbar\times \Spec \Z_{(2)}$, we obtain a synthetic transfer map $p_!\colon \Smf_1(3)_{(2)} \to \Smf_{(2)}$, which in turn induces the algebraic transfer map on mod $\tau$ bigraded homotopy groups
\[p_!\colon \uH^s(\Mbar_1(3)\times \Spec \Z_{(2)},\, \omega^\ast) \to \uH^s(\Mellbar\times \Spec \Z_{(2)},\, \omega^\ast).\]
To compute this map in a specific range, we will compare this to a transfer map defined using Hopf algebroids.

\begin{definition}
    Let $(A',\Ga')$ be the \defi{$2$-primary cubic Hopf algebroid}, defined by the data
    \[A' = \Z_{(2)}[a_1,a_3], \qquad \Ga' = A'[s,t]/{\sim},\]
    where the latter denotes the quotient by the relations
    \[s^4-6st+a_1s^3-3a_1t-3a_3s = 0,\]
    \[s^6-27t^2+3a_1s^5-9a_1s^2t+3a_1^2s^4 - 9a_1^2st + a_1^3s^3 - 27a_3t = 0,\]
    and with left unit, right unit, and comultiplication defined such that the evident quotient
    \[(A_{(2)},\Ga_{(2)}) \to (A',\Ga')\]
    is a map of Hopf algebroids.
\end{definition}

By \cite[Section~7]{bauer_tmf}, the above quotient map of Hopf algebroids induces an equivalence of stacks, and hence induces an isomorphism on cohomology groups.
The left $A'$-module $\Ga'$ corresponds to the pushforward of the structure sheaf of $\Mbar_1(3)\times \Spec \Z_{(2)}$ to $\Mellbar\times \Spec \Z_{(2)}$ along the map $p$ above; see \cite[Section~3]{mathew_homology_tmf}.
In particular, our computations of the algebraic transfer along $f$ are equivalent to that of the transfer $\Tr\colon \Ga' \to A'$.

In more detail, first note that the left $A'$-module $\Gamma'$ is free of rank 8 with basis
\[\set{1,\, s,\, s^2,\,s^3,\,t,\,st,\,s^2t,\,s^3t}.\]
For $x\in\Gamma'$, multiplication by $x$ on $\Gamma'$ is an $A'$-linear endomorphism.
The associated trace map $\Tr\colon \Gamma' \to A'$ agrees with the algebraic transfer associated to $p_!$, as these are determined by their effect on a finite flat cover, where they are both given by taking sums.
One can compute the traces of multiplication by $s$ and $t$ with respect to the above basis, resulting in the following.

\begin{proposition}\label{pr:hopfalgeboircomputation}
    The transfer map $\mathrm{Tr}\colon \Gamma'\to A'$ is the map of left $A'$-modules determined by the following formulas:
    \begin{align*}
        \mathrm{Tr}(1)&=8&\mathrm{Tr}(t)&=\tfrac{1}{3}a_1^3-4a_3\\
        \mathrm{Tr}(s)&=-4a_1&\mathrm{Tr}(st)&=-2a_1a_3\\
        \mathrm{Tr}(s^2)&=2a_1^2&\mathrm{Tr}(s^2t)&=-\tfrac{1}{3}a_1^5+9a_1^2a_3\\
        \mathrm{Tr}(s^3)&=-a_1^3&\mathrm{Tr}(s^3t)&=\tfrac{1}{3}a_1^6-7a_1^3a_3-27a_3^2.
    \end{align*}
\end{proposition}

At the prime $2$, the formulas for $c_4$, $c_6$ and $\Delta$ simplify to
\begin{equation}\label{expressionsforcfour}
    c_4 = a_1^4 - 24 a_1 a_3, \qquad c_6 = -a_1^6 +36 a_1^3 a_3 -216 a_3^2, \qquad \Delta = a_3^3 (a_1^3 - 27 a_3).
\end{equation}
The formulas for $\eta_R$ also simplify: we have
\begin{equation}
    \label{formulasforetaR}
    \eta_R(a_1)=a_1+2s \qquad \text{and} \qquad \eta_R(a_3)=a_3+\tfrac{1}{3}a_1s^2+\tfrac{1}{3}a_1^2s+2t.
\end{equation}
This immediately implies the following corollary.

\begin{corollary}\label{survivalofmodularformsattwo}
    The ideal
     \[(8,2c_4,2c_6)\subseteq \uH^0(\Mellbar\times\Spec\Z_{(2)}, \, \omega^{\otimes *})\]
    consists of permanent cycles in the DSS for $\Tmf_{(2)}$.
\end{corollary}

\begin{proof}
    The map of synthetic spectra $p_!\colon \Smf_1(3)_{(2)} \to \Smf_{(2)}$ induces a natural map of DSSs, which by \cref{etwopagebracketsprimethree} can be computed in the connective region of the target using \cref{pr:hopfalgeboircomputation}.    
    As $\Mbar_1(3)$ has cohomological dimension $1$, the DSS for $\Tmf_1(3)_{(2)}$ collapses on the $\uE_2$-page, so all classes are permanent cycles.
    In particular, any class in the image of the algebraic transfer map is a permanent cycle.
    The formulas from \cref{pr:hopfalgeboircomputation} together with \eqref{expressionsforcfour} and \eqref{formulasforetaR} then show that the elements $8$, $2c_4$, and $c_6$ lie in the image of the transfer map:
    \[\Tr(1)= 8, \qquad \Tr(3\eta_R(a_1a_3)) = 2c_4, \qquad \Tr(-\eta_R(a_3^2)) = 2c_6. \qedhere\]
\end{proof}

A similar analysis at the prime~$3$ for the DSS for $\TMF$ can be found in \cite[Section~5.6]{Meier_TMF}. We only need the $2$-local argument for this article.

\begin{remark}
    The precomposition of the transfer $\Tr\colon \Ga' \to A'$ with the right unit $\eta_R\colon A' \to \Ga'$ is a kind of $\Ga_1(3)$-Hecke operator. In other words, the computations above are a computation of the effect of the spectral Hecke operator $\mathrm{T}_{\Ga_1(3)}$ on $\TMF_{(2)}$ defined in \cite[Definition~2.6]{davies_hecke_tmf} on the $\uE_2$-page of its DSS on the submodule of holomorphic modular forms. 
\end{remark}


\section{Information from the sphere and detection}\label{detectionsection}

By construction, the synthetic $\E_\infty$-ring $\Smf$ comes with a map $\S\to \Smf$, which in particular gives us a map from the ANSS for $\S$ to the DSS for $\Tmf$.
In this subsection, we prove that various elements in the ANSS for $\S$ are detected in the DSS for $\Tmf$ via this map.

In what follows, we will implicitly use the fact that we know the bigraded homotopy groups of $\Smf/\tau$: this is the $\uE_2$-page of the DSS for $\Tmf$, which is computed by Konter \cite{konter_Tmf}.


\subsection{Height one detection results}

We begin by identifying $v_1$.

\begin{proposition}\label{codetectionone}
    \leavevmode
    \begin{numberenum}
        \item The map $\S/\tau\to \Smf_1(3)/\tau$ detects the classes
        \[
            \widetilde{v}_1\in \oppi_{2,0}\S/(\tau,2)\cong \F_2 \quad\text{and}  \quad \widetilde{v}_1^2 \in \oppi_{4,0} \S/(\tau,4)\cong \F_2,
        \]
        where $\widetilde{v}_1$ and $\widetilde{v}_1^2$ are generators.
        \item The natural map of synthetic $\E_\infty$-rings $\S/\tau\to \Smf_1(2)/\tau$ detects the nonzero class $\widetilde{v}_1\in \oppi_{4,0}\S/(\tau,3)$.
    \end{numberenum}
\end{proposition}

\begin{proof}
    The DSS for $\Tmf_1(3)$ collapses on the $\uE_2$-page, as the stack $\Mbar_1(3)$ is a weighted projective stack (see \cite[Example~2.1]{meier_additive_decomp22}) so their DSSs are concentrated in filtrations $0$ and~$1$.
    This yields the calculation $\oppi_\ast \Tmf_1(3) \cong A\oplus \widehat{A}$, where
   \[
        A=\Z[\tfrac{1}{2}][a_1, a_3], \qquad \text{where }\abs{a_i}=2i,
    \]
    and where $\widehat{A}=A/(a_1^\infty, a_3^\infty)$ is the torsion $A$-module with $\Z[\frac{1}{2}]$-basis given by
    \[
        \frac{1}{a_1^i a_3^j} \quad \text{with }i,j>0, \qquad \text{where }\abs*{\frac{1}{a_1^i a_3^j}}=-2i-6j-1;
    \]
    this is essentially \cite[Corollary~3.3]{mahowald_rezk_tmf03}.
    In positive degrees, these elements admit a description in terms of modular forms of level $\Ga_1(3)$.
    By \cite[Theorem 6.2]{hill_lawson_level_structure}, evaluation at the cusp gives a map of $\E_\infty$-rings
    \[
        \Tmf_1(3)\to \KU[\tfrac{1}{3}]
    \]
    that sends $a_1\in \oppi_2\Tmf_1(3)$ to $u$, where $u$ denotes the Bott periodicity generator of $\oppi_2 \KU[\frac{1}{3}]$.
    Let $a_1$ denote the nonzero class in $\oppi_{2,0}\Smf_1(3)/(\tau,2)\cong \F_2$.
    Because the mod~$2$ reduction of $u$ detects $\widetilde{v}_1$ on Adams--Novikov $\uE_2$-pages, see \cite[Theorem 5.4]{christian_jack_synthetic_j} for example, the reduction of the synthetic lift of $a_1$, must also detect $\widetilde{v}_1$. Similar arguments show that $\widetilde{v}_1^2$ is detected by $\Smf_1(2)/(\tau,4)$. The arguments at the prime $3$ are the same.
\end{proof}

Next, we have the corresponding $v_1$-periodic elements in the divided $\al$-family.

\begin{proposition}\label{heightonedetection}
    The natural map of synthetic $\E_\infty$-rings $\S/\tau\to \Smf/\tau$ detects the following classes:
    \[\al\in \oppi_{3,1}\S_{(3)}/\tau,\qquad \eta\in \oppi_{1,1}\S_{(2)}/\tau,\qquad \nu \in \oppi_{3,1}\S_{(2)}/\tau, \qquad \eps \in \oppi_{8,2}\S/\tau.
    \]
\end{proposition}

We assume the reader is familiar with this Adams--Novikov notation, as well as the ANSS for $\S$ in low degrees; see, e.g., \cite[Table~2]{ravenel_novice_guide_ANSS}.

\begin{proof}
    The class $\eta$ can be defined as $\partial(\widetilde{v}_1)$, where $\partial\colon \S/2\to \Sigma\S_{(2)}$ is the boundary map associated to the mod $2$ Moore spectrum, and $\widetilde{v}_1\in \oppi_2 \S/2$ is the class inducing multiplication by $v_1$ on $\K(1)$-homology.
    The class $\widetilde{v}_1$ is detected on the ANSS for $\S$ in filtration~$0$, i.e., in the bigraded homotopy group $\oppi_{2,0}\S/(\tau,2)$.
    
    We now consider the following commutative diagram of abelian groups
    \[\begin{tikzcd}
         &   {\oppi_{2,0}\S/(\tau,2)}\ar[r, "\partial"]\ar[d, "h"] &   {\oppi_{1,1}\S_{(2)}/\tau}\ar[d]  &     \\
     {\oppi_{2,0}\Smf_{(2)}/\tau}\ar[r, "q"]    &{\oppi_{2,0}\Smf/(\tau,2)}\ar[r, "\partial"]\ar[d] &   {\oppi_{1,1}\Smf_{(2)}/\tau}\ar[r, "2"]   &  {\oppi_{1,1}\Smf_{(2)}/\tau}\\
     &    {\oppi_{2,0}\Smf_1(3)/(\tau,2)}    &&
    \end{tikzcd}\]
    where the vertical maps are induced by the maps of synthetic $\E_\infty$-rings $\S\to \Smf\to \Smf_1(3)$.
    First, notice that $h(v_1)\neq 0$ as the image of $v_1$ is nonzero in $\oppi_{2,0}\Smf_1(3)/(\tau,2)$ by \cref{codetectionone}.

    Next, by \cref{etwoprime2} and the exactness of the rows above, we see that $\oppi_{2,0}\Smf_{(2)}/\tau$ is zero and $\pi_{1,1} \Smf_{(2)}/\tau$ has order $2$.
    In particular, the multiplication by $2$ map in the diagram above is zero, and we see that $\partial$ in the middle row is an isomorphism.
    From the commutativity of the above square, this implies that the nonzero class in $\oppi_{1,1}\Smf_{(2)}/\tau$ detects $\eta=\partial(\widetilde{v}_1)$.
    
    For $\nu$, we use that $\nu = \partial(\widetilde{v}_1^2)$ coming from $\S/4$, and in the $\eta$-case above. The $\alpha$-case also follows \emph{mutatis mutandis}.
    For the claim about $\eps$, note that in $\S_{(2)}/\tau$, we have the relation $h_2^3=h_1\cdot c$.
    As both $h_2^3$ and $h_1$ are nonzero in $\Smf_{(2)}/\tau$, the conclusion follows.
\end{proof}

\begin{remark}
    It follows immediately from the existence of a map of $\E_\infty$-rings $\Tmf \to \KO$ that $\eta$ is detected by $\Tmf$, and the detection of $\nu$ can also be eked out of this fact. This map is the composite of the topological $q$-expansion map $\Tmf \to \KO\llbracket q\rrbracket$ and the maps $\KO\llbracket q\rrbracket \to \KO$ sending $q$ to $0$. One can construct the topological $q$-expansion map either via obstruction theory \cite[Section~A]{hill_lawson_level_structure} or spectral algebraic geometry \cite[Theorem~6.18]{Davies_Linskens_Tatecurve}.
\end{remark}

As $\Smf$ detects $\eta$ and $\nu$, it also detects Toda brackets formed from these elements. Note that the Toda bracket $\angbr{\nu,\eta,\nu}\subseteq \pi_{8,2}\S$ is nonempty and has zero indeterminacy as $\nu \pi_{5,1}\S = 0$ by inspection of the ANSS for $\S$. By \cite[Corollary 3.4.12]{kochman_stable}, the $\tau$-inversion of this Toda bracket detects the classical class $\eps\in \pi_{8}\Sph$, sometimes written as $c_0=c$ in Adams notation or $\beta_2$ in Adams--Novikov notation.





\subsection{The classes \texorpdfstring{$\kappa$}{kappa} and \texorpdfstring{$\kappabar$}{kappa-bar}}\label{sssec:kappabar}

The height $2$ classes $\kappa$ and $\kappabar$ are more subtle than the height $1$ elements encountered above.

\begin{proposition}
    \label{prop:kappa_detection}
    The natural map of synthetic $\E_\infty$-rings $\S_{(2)}/\tau\to \Smf_{(2)}/ \tau$ detects the mod~$\tau$ reduction of the class $\kappa \in \oppi_{14,2}\S_{(2)}$.
    In particular, the image of the class $\kappa \in \oppi_{14,2} \S_{(2)}$ in $\Smf_{(2)}$ is also nonzero.
\end{proposition}

\begin{proof}
    In \cite[Table~23]{isaksen_stable_stems}, Isaksen shows that the nonzero class $\eta\kappa$ lies in $\angbr{2\nu, \nu, \eps}$ in $\oppi_{14}\Sph$.
    As $\oppi_{14,2}\Sph$ is $\tau$-torsion free, we can lift this containment to $\eta\kappa\in \angbr{2\nu, \nu, \eps}$ in the synthetic sphere.
    Using \cref{prop:todagoestomassey,etwopagebracketsprimethree}, we see that modulo $\tau$ we have a nonzero class $h_1d \in \angbr{2h_2, h_2, c}$ in the cohomology of the $2$-primary cubic Hopf algebroid; see \cite[Appendix~A]{bauer_tmf}.
    
    As $\oppi_{12,2}\Smf_{(2)}/\tau$ and $\oppi_{7,1}\Smf_{(2)}/\tau$ both vanish (see \cref{etwoprime2}) the above Massey product has zero indeterminacy and hence cannot contain zero. In particular, we see that $\eta\kappa$ cannot vanish in $\oppi_{15,3}\Smf_{(2)}$. As $\eta$ is detected in $\oppi_{1,1}\Smf_{(2)}$, this implies that $\kappa \in\oppi_{14,2}\Smf_{(2)}$ is also nonzero and reduces to the generator $d$ in $\oppi_{14,2}\Smf_{(2)}/\tau \cong \F_2$.
\end{proof}

Next, we turn to $\kappabar$.
In the sphere, this class is only well defined up to a factor of $\nu^2\kappa$.

\begin{definition}
    We write $\kappabar$ for any choice of element in the Toda bracket $\angbr{\kappa, 2,\eta,\nu}$ in the (non-synthetic) $\oppi_{20}\Sph_{(2)}$.
    Similarly, we write $\kappabar$ for any choice of element inside the synthetic Toda bracket $\angbr{\kappa, 2, \eta, \nu}$ in $\oppi_{20,2} \S_{(2)}$.
\end{definition}

A priori, it is not clear that either of these brackets are nonempty or nonzero.
This follows from some classical facts.

\begin{lemma}\label{lem:keykappabarbracket}
    The indeterminacy of the Toda bracket $\angbr{\kappa, 2, \eta, \nu}$ in the non-synthetic sphere spectrum $\oppi_{20}\S_{(2)}$ is equal to the subgroup generated by $\nu^2\kappa$. Moreover, this bracket does not contain zero. 
\end{lemma}

\begin{proof}
    Using the formula for the indeterminacy of a four-fold Toda bracket as presented by Kochman \cite[Theorem~2.3.1\,(b)]{kochman_stable}, one can compute the desired indeterminacy.
    We verified this using the Massey product calculator of \cite{sseq}.

    For the moreover statement, we use the computation of Kochman \cite[Lemma~5.3.8\,(e)]{kochman_stable}, which states that $\angbr{\kappa, 2, \eta, \nu}$ contains a generator of $\oppi_{20} \S_{(2)} \cong \Z/8$.
    This, combined with the fact that $\nu^2\kappa$ is the nonzero $2$-torsion element in this group \cite[Theorem~5.3.1\,(a)]{kochman_stable}, shows that this bracket does not contain zero.
\end{proof}

The naturality of Toda brackets as in \cref{lem:functorspreservebrackets} then implies that the synthetic Toda bracket $\angbr{\kappa, 2, \eta, \nu}$ also cannot contain zero.
Moreover, the fact that $\oppi_{20,2} \S_{(2)}$ is $\tau$-power torsion free allows us to lift relations on any choice of $\kappabar$ too.

\begin{lemma}\label{lem:kappabarkeyrelation1}
    For any choice of $\kappabar \in \angbr{\kappa, 2, \eta, \nu} \subseteq \oppi_{20,2} \S_{(2)}$, we have $4\kappabar = \tau^2 \nu^2 \kappa$.
\end{lemma}

\begin{proof}
    The proof of \cref{lem:keykappabarbracket} shows that the equation $4\kappabar = \nu^2 \kappa$ holds in $\oppi_{20}\Sph_{(2)}$.
    As the $\tau$-inversion map is injective in degree $(20,2)$, we obtain the desired equation in $\oppi_{20,2}\S_{(2)}$.
\end{proof}

\begin{lemma}\label{lem:kappabarkeyrelation2}
    For any choice of $\kappabar \in \angbr{\kappa, 2, \eta, \nu} \subseteq \oppi_{20,2} \S_{(2)}$, the class $\nu^3\kappabar \in \oppi_{29,5} \S_{(2)}$ is $\tau$-power torsion.
\end{lemma}

\begin{proof}
    This follows from the classical fact that $\oppi_{29}\Sph_{(2)}=0$; see \cite[Figure~A3.2]{ravenel_green_book}.
\end{proof}

The ill-definedness of $\kappabar$ persists in $\Smf$, where there is yet another issue. 
In the sphere, the class $\kappabar$ has Adams--Novikov filtration $2$, while in $\Smf$, it turns out to have filtration $4$.
We now make precise what we mean by this, which will allow us to pin down what what we mean by $\kappabar$ as an element of $\oppi_{*,*}\Smf$ going forward.

\begin{proposition}
    \label{prop:checking_kappabar}
    For every choice of $\kappabar \in \angbr{\kappa, 2, \eta, \nu} \subseteq \oppi_{20,2} \S_{(2)}$, there is a unique element $y$ in $\oppi_{20,4}\Smf_{(2)}$ such that $\tau^2 \cdot y \in\oppi_{20,2}\Smf_{(2)}$ is the image of $\kappabar \in \oppi_{20,2}\S_{(2)}$ under $\S_{(2)} \to \Smf_{(2)}$.
    Moreover, the reduction of $y$ to $\Smf/\tau$ is a generator for the group $\oppi_{20,4} \Smf_{(2)}/\tau \cong \Z/8$.
\end{proposition}
\begin{proof}
We will implicitly localise at the prime $2$ in this proof.
Write $x \in \oppi_{20,2}\Smf$ for the image of $\kappabar \in \oppi_{20,2}\S$ under the unit map $\S\to \Smf$.
We claim that $x$ is uniquely $\tau^2$-divisible.
Indeed, we look at the long exact sequence induced by the cofibre sequence
\[
    \begin{tikzcd}
        \opSigma^{-1,1} \Smf/\tau^2 \ar[r] & \opSigma^{0,-2} \Smf \ar[r,"\tau^2"] & \Smf \ar[r] & \Smf/\tau^2.
    \end{tikzcd}
\]
Part of this long exact sequence reads
\[
    \begin{tikzcd}
        \oppi_{21,1} \Smf/\tau^2 \ar[r] & \oppi_{20,4} \Smf \ar[r,"\tau^2"] & \oppi_{20,2} \Smf \ar[r] & \oppi_{20,2} \Smf/\tau^2.
    \end{tikzcd}
\]
The homotopy of $\Smf/\tau$ vanishes in bidegrees $(21,1)$, $(21,2)$, $(20,2)$, and $(20,3)$; see \cref{etwoprime2}.
\Cref{cor:easy_lifting_higher_power_tau} then implies that the homotopy of $\Smf/\tau^2$  vanishes in $(21,1)$ and $(20,2)$, so that multiplication by $\tau^2$ induces an isomorphism
\[
    \tau^2 \colon \oppi_{20,4} \Smf \congto \oppi_{20,2} \Smf.
\]
Write $y$ for the $\tau^2$-division of $x$.
From \cref{lem:kappabarkeyrelation1}, we learn that
\[
    4x = \tau^2 \cdot \kappa \nu^2.
\]
Because multiplication by $\tau^2$ in this bidegree in $\Smf$ is injective, we deduce that
\[
    4y = \kappa \nu^2.
\]
From the multiplicative structure of $\Smf/\tau$, we see that $\kappa\nu^2$ reduces to 4 times a generator of $\oppi_{20,4}\Smf/\tau\cong\Z/8$.
Write $g$ for a choice of a generator, and write $\bar{y}$ for the reduction mod $\tau$ of $y$.
The relation above then tells us that $4g = 4\bar{y}$, i.e., $4(g-\bar{y})=0$.
As such, $g-\bar{y}$ is an even multiple of $g$, i.e., $\bar{y}$ is an odd multiple of $g$, which says that it is a unit in $\Z/8$ away from $g$.
This means that $\bar{y}$ is also a generator, as claimed.
\end{proof}

\begin{notation}
\label{not:kappabar_and_g}
\leavevmode
\begin{itemize}
    \item From now on, we fix a choice of preferred element in $\angbr{\kappa, 2, \eta, \nu} \subseteq \oppi_{20,2}\S_{(2)}$.
    \item We write $\kappabar \in \oppi_{20,4} \Smf_{(2)}$ for the element uniquely determined by this choice using \cref{prop:checking_kappabar}.
    As a result, the chosen element in $\oppi_{20,2}\S_{(2)}$ maps to $\tau^2 \kappabar$ in $\oppi_{20,2}\Smf_{(2)}$.
    \item We write $g \in \oppi_{20,4}\Smf_{(2)}/\tau$ for the reduction of $\kappabar$ to $\Smf/\tau$.
    By \cref{prop:checking_kappabar}, this is a generator of $\oppi_{20,4}\Smf_{(2)}/\tau\cong\Z/8$.
\end{itemize}
\end{notation}

From this definition of $\kappabar$ in $\Smf_{(2)}$, we immediately obtain the following facts.

\begin{corollary}\label{cor:todabracketforkappabar}
    \leavevmode
    \begin{itemize}
        \item \textup{(20,2)} $\tau^2\kappabar \in \angbr{\kappa, 2, \eta, \nu}$ in $\Smf_{(2)}$.
        \item \textup{(20,4)} $4\kappabar = \nu^2 \kappa$ in $\Smf_{(2)}$.
        \item \textup{(29,7)} $\nu^3 \kappabar$ is $\tau$-power torsion in $\Smf_{(2)}$.
    \end{itemize}
\end{corollary}

All of these facts will be crucial in our computations.


\subsection{Toda bracket relations from the sphere}\label{sssec:todabracketsfromsphere}

There are also a handful of standard Toda bracket computations from the non-synthetic sphere which we would like to lift to $\Smf$.

\begin{proposition}\label{prop:toadyfromsphere}
    \leavevmode
    \begin{itemize}
        \item \textup{(8,2)} $\eps \in \angbr{\nu,2\nu,\eta}$ 
        in $\Smf_{(2)}$.
        \item \textup{(15,3)} $\eta \kappa \in \angbr{\nu, 2\nu, \eps}$ and $\eta \kappa \in \angbr{2\nu, \nu, \eps}$ in $\Smf_{(2)}$.
        \item \textup{(21,3)} $\tau^2 \eta \kappabar \in \angbr{\nu, 2\nu, \kappa}$ in $\Smf_{(2)}$.
    \end{itemize}
\end{proposition}

All of these relations will follow from the analogous relations in the sphere. We only state them in $\Smf$ to be consistent with our use of $\kappabar$ from \cref{not:kappabar_and_g}.

\begin{proof}
It suffices to prove these statements in the synthetic sphere.
By \cite[Table~10]{isaksen_wang_xu_dimension_90}, all of these statements hold in the non-synthetic sphere, except for those for $\eta\kappa$, which follow from \cite[Table~23]{isaksen_stable_stems}.\footnote{Although \cite{isaksen_stable_stems,isaksen_wang_xu_stable_stems} as a whole rely on the homotopy groups of $\tmf$, this is only for computations in stems higher than those considered here. One can independently check these computations using the Massey product calculator of \cite{sseq}, for example.}
\end{proof}


\section{Computations at the prime 2}\label{sec:prime2}

At last, we can begin to study the DSS for $\Tmf$ at the prime $2$. In \cref{theappendix}, we provide spectral sequence charts that aid in reading the computation in this section. We also include several tables with information from this section, including relevant information from the sphere, important lifts of $E_2$ elements to $\Smf/\tau^k$ for various $k$, relations and extensions in $\pi_{*,*}\Smf/\tau^k$ for various $k$, important values of the (truncated) total differential, and key Toda brackets.

Throughout this section, we are implicitly working $2$-locally.

\begin{theorem}\label{thm:dssattwo}
    The signature spectral sequence of $\Smf$, i.e., the DSS for $\Tmf$, is determined below; see \cref{efiveprime2_part1,efiveprime2_part2,efiveprime2_part3,efiveprime2_part4}.
\end{theorem}

Another diagram of the same spectral sequence is given in \cite[Figures~26--27]{konter_Tmf} with only minor typographical changes and omissions.

To prove \cref{thm:dssattwo}, we work page by page through the DSS. We will only explicitly state those differentials which propagate all others by the Leibniz rule.

\begin{definition}\label{def:atomicdiffs}
    An \defi{atomic differential} is a differential $d_r$ in a multiplicative spectral sequence, of the form $d_r(x)=y$, where $x$ is indecomposable in the dga $E_r$ and $y\neq 0$.
\end{definition}

It follows from the Leibniz rule that all of our atomic differentials yield all of the differentials in the \emph{connective region} (\cref{def:regions}); this is the region of \cref{efiveprime2_part1,efiveprime2_part2,efiveprime2_part3,efiveprime2_part4} below the blue line.

The \emph{meta-arguments} of \cref{ssec:meta} ensure both that all of the differentials above the blue line and to the right of the orange line follow from our atomic differentials, and that there are no differentials that cross this blue line.

Working page by page therefore amounts to ticking-off the following checklist:
\begin{itemize}
    \item compute the atomic $d_r$-differentials;
    \item check the conditions of the {meta-arguments} to deduce all other differentials whose sources lie in stems $n\geq -20$; 
    \item calculate as many lifts of elements, total differentials, extensions, and synthetic Toda brackets as will be necessary for future pages.
\end{itemize}
At the very end of the computation (\ref{ssec:nonconnectiveregion}), we come back to the computation of differentials in stems $n\leq -21$, which amounts to a careful check of total differentials.

We begin with the $\uE_2$-page, so the groups $\pi_{\ast,\ast}\Smf/\tau$, which is the classical computation of the $2$-local cohomology of the compactification of the moduli stack of elliptic curves; see \cref{etwoprime2} or \cite[Figure~25]{konter_Tmf}.
There is also a connective region (\cref{e2pagesection}) of this diagram which can be found in \cite[Section~7]{bauer_tmf} and within which we understand some of the higher multiplicative structure through \cref{prop:todagoestomassey} and \cref{etwopagebracketsprimethree}.

We also import the following elements from the ANSS for $\Sph$.

\begin{notation}
    \label{not:2_import_from_sphere}
    \leavevmode
    \begin{itemize}
        \item We write $\eta \in \oppi_{1,1}\Smf$ for the image of $\eta \in \oppi_{1,1}\S$ under $\S \to \Smf$.
        \item We write $\nu \in \oppi_{3,1}\Smf$ for the image of $\nu \in \oppi_{3,1}\S$ under $\S \to \Smf$.
        \item We write $\eps \in \oppi_{8,2}\Smf$ for the image of $\eps \in \oppi_{8,2}\S$ under $\S \to \Smf$.
        \item We write $\kappa \in \oppi_{14,2}\Smf$ for the image of $\kappa \in \oppi_{14,2}\S$ under $\S \to \Smf$.
        \item We write $\kappabar \in \oppi_{20,4}\Smf$ for the element defined in \cref{not:kappabar_and_g}.
    \end{itemize}
\end{notation}

Modulo $\tau$, these elements reduce to $h_1$, $h_2$, $c$, $d$ and $g$, respectively; see \cref{heightonedetection,sssec:kappabar}.
In particular, these classes in $\Smf/\tau$ are permanent cycles.

\subsection{Meta-arguments}\label{ssec:meta}

There are two steps in our computation that amount to a repeated check on each page, so we condense these checks into a pair of meta-arguments here. For clarity of exposition we separate these from the main page-by-page argument. However, we will need to make some forward references to the $\uE_4$-page. We will therefore be explicit in computing the $\uE_4$-page in our main argument without using the meta-arguments here in a circular way. This should not cause much confusion as the $\uE_3$ page is quite simple, with only one atomic differential.

In this section, we will frequently use terminology based on the way the relevant information appears in the spectral sequence chart, so we define these terms carefully first.

\begin{definition}\label{def:regions}
    The \defi{connective region} on the $\uE_r$-page of the descent spectral sequence for $\Tmf$ is the region in the plane consisting of those values of $(n,s)$ on or below the line
    \[s = \tfrac{1}{5}n + \tfrac{12}{5}.\]
    The \defi{nonconnective region} refers to the region above this line. A \defi{line-crossing differential} is a differential $d_r(x)=y$ with the property that $x$ is in the connective region and $y$ is in the nonconnective region. 
\end{definition}

In the charts of Appendix~\ref{ssec:ss}, we have drawn the line $s=\frac{1}{5}n+\frac{12}{5}$ in blue.

\subsubsection{Line-crossing differentials}\label{sssec:linecrossing}
Our first meta-argument allows us to efficiently rule out the possibility of line-crossing differentials. There is a simple condition on atomic differentials we must check on each page --- this guarantees that we only have to rule out line-crossing differentials through a finite range on a given page, which we check on each page by hand.

\begin{proposition}\label{prop:gdivisibilitycheck}
    Let $r>4$.
    Suppose we know that for all $r'<r$ and all atomic differentials $d_{r'}(x)=y$ in the connective region that the implication
    \[
        gy=0 \text{ in } \uE_{r'} \implies gx=0 \text{ in } \uE_{r'}
    \]
    holds.
    Then any element in the connective region on $\uE_r$ in filtration $\ge 4$ is divisible by $g$.
\end{proposition}
\begin{proof}
    Every element in the connective region on $\uE_4$ in filtration $\ge 4$ is divisible by $g$, as follows from \cref{prop:connectiveregiondivisiblebyg}. Inductively, an element $z\in \uE_r$ can be written as $z=gx$ on $\uE_{r'}$. If the element $z$ fails to be divisible by $g$ on $\uE_{r}$, there must be a differential $d_{r'}(x)=y$ with $gy=0$, but the assumptions preclude this possibility. 
\end{proof}

\begin{proposition}\label{prop:gdivisibilitycheck2}
    Let $r>4$. Suppose that any element in the connective region on $\uE_r$ in filtration $\ge 4$ is divisible by $g$. Let $n$ be the largest stem on $\uE_r$ such that the nonconnective region has a nonzero class in bidegrees $(n,s)$ for $r\le s\le r+3$. If there are no line-crossing $d_r$-differentials through the $(n+1)$-stem, then there are no line-crossing $d_r$-differentials.
\end{proposition}
\begin{proof}
    Any possible line-crossing differential whose source $x$ is in filtration $\ge 4$ has $x=gx'$, hence if $d_r(x')=0$, then $d_r(x)=0$. We may therefore assume without loss of generality that $x$ is in filtration $0\le s\le 3$, and the result follows.
\end{proof}

\subsubsection{Differentials in the $S$-region}\label{sssec:linecrossing_noncon}
Our second meta-argument allows us to conclude that all differentials in the $S$-region of \cref{not:S_region} are accounted for by application of the Leibniz rule. For this we will use the following.

\begin{proposition}\label{prop:Delta8torsfree}
    If the $\uE_r$-page is $\Delta^8$-torsion free in stems $\ge n-1$, and $\Delta^8$ is a $d_r$-cycle, then all $d_r$-differentials in the nonconnective region with source in the $n$-stem are uniquely determined by the $d_r$-differentials in the connective region via the Leibniz rule.
\end{proposition}
\begin{proof}
    Let $x\in \uE_r$ be an element in the $n$-stem in the nonconnective region. For a suitable power $k$, the product $\Delta^{8k}x$ lies in the connective region, and one may divide a nonzero differential $d_r(\Delta^{8k}x)$ by $\Delta^{8k}$ to obtain a nonzero differential in the nonconnective region. Conversely, also by $\Delta^8$-torsion freeness, any nonzero differential with source in the $n$-stem in the nonconnective region determines a nonzero differential in the connective region.
\end{proof}

The torsion-free condition of the previous proposition also follows from the checks in \cref{prop:gdivisibilitycheck,prop:gdivisibilitycheck2}.

\begin{definition}
    \label{not:S_region}
    Let us denote by $S$ the region consisting of stems $n>-21$, together with the $(-21)$-stem in filtration $s>1$.
    In our spectral sequence charts of Appendix~\ref{ssec:ss}, we indicate this region as the one to the right of the orange line.
\end{definition}

\begin{proposition}\label{prop:Sregion}
    Let $r>4$. Suppose the conditions of \cref{prop:gdivisibilitycheck,prop:gdivisibilitycheck2} have been verified for $d_{\le r}$ and suppose that $\Delta^8$ is a $d_{r+1}$-cycle. Then $\uE_{r+1}$ is $\Delta^8$-torsion free in the region $S$.
\end{proposition}

\begin{remark}
    The region $S$ has been chosen precisely to yield \cref{thm:gap}. Much of the nonconnective region is actually $\Delta^8$-torsion free, but this does not always hold. For example, $8$ times the generator of $\oppi_{-21,1}\Smf/\tau$ is $\Delta$-torsion; see \cref{etwoprime2}. We will have to come back to these more subtle cases in \cref{ssec:nonconnectiveregion}, where we finish the proof of \cref{thm:dssattwo}.
\end{remark}

\begin{proof}[Proof of \cref{prop:Sregion}]
    We proceed by induction, noting that $\uE_4$ is $\Delta^8$-torsion free in $S$, that the conditions of \cref{prop:gdivisibilitycheck,prop:gdivisibilitycheck2} hold for $d_{\le 4}$, and finally, that by \cref{prop:E4metachecks} the following pair of conditions holds for $r\le 4$.
    \begin{itemize}
        \item A class $a\in \uE_r^{n,s}$ satisfies
        \[s < \tfrac{1}{5}(n-192) + \tfrac{12}{5}\]
        if and only if $a=\Delta^8b$, where $b$ is in the connective region.
        \item A class $a\in \uE_r^{n,s}$ in $S$ satisfies $a=\Delta^8b$ for $b$ in the $S$-region if and only if $n\ge 171$.
    \end{itemize}
    We therefore assume by induction that $\uE_r$ is $\Delta^8$-torsion free, and  that the two conditions above hold for all $r'\le r$. Since $\uE_r$ is $\Delta^8$-torsion free, the $\uE_{r+1}$-page is $\Delta^8$-torsion free in $S$ unless there is a $d_r$-cycle $x$ in $S$ such that $d_r(y)=\Delta^8x$ for some $y$, and $y$ is not divisible by $\Delta^8$. Given that the conditions of \cref{prop:gdivisibilitycheck,prop:gdivisibilitycheck2} hold for $d_{\le r}$, such a differential cannot be line-crossing, so it has both source and target in either the connective or nonconnective region. In the former case, it follows that $\Delta^8y$ lives in a bidegree $(n,s)$ satisfying
    \begin{equation}\label{eq:inequality}
     s<\tfrac{1}{5}(n-192)+\tfrac{12}{5}
    \end{equation}
    and hence so must $x$. This leads to a contradiction as then $x$ is also divisible by $\Delta^8$. In the latter case, it follows that $\Delta^8y$ lives in a bidegree $(n,s)$ satisfying $n\ge 192$, hence the same is true for $x$, showing again that $x$ is $\Delta^8$-divisible.
 
    Therefore $\uE_{r+1}$ is $\Delta^8$-torsion free in $S$.
    To complete the induction, we need to establish the above two conditions for $\uE_{r+1}$.
    If $b$ is in the connective region and $a=\Delta^8b$, then it is clear from $\uE_2$ that \eqref{eq:inequality} holds.
    Conversely, if $a\in \uE_{r+1}^{n,s}$ satisfies \eqref{eq:inequality}, then $a=\Delta^8b$ on $\uE_r$ by induction, and $b$ is in the connective region for degree reasons. Since $\uE_r$ is $\Delta^8$-torsion free in $S$, and $a$ is a $d_r$-cycle, so is $b$, so the division $a=\Delta^8b$ carries to the $\uE_{r+1}$-page unless $b$ is hit by a differential. But since $\Delta^8$ is a $d_r$-cycle, it would also follow that $a$ is hit, a contradiction. The condition in the nonconnective region is established by an analogous argument.
\end{proof}

\subsection{Page 3}

\subsubsection{Atomic differentials}

\begin{proposition}[3,1]\label{prop:d3differential}
$d_3(b) = h_1^4$.
\end{proposition}

\begin{proof}
    By \cref{heightonedetection}, the map
\[ \Sph\to \Smf\to \Smf/\tau\]
sends $\eta$ to the unique nonzero class $h_1\in\oppi_{1,1}\Smf/\tau$. There is a differential in the ANSS for $\Sph$ which hits $\eta^4$, see \cite[Table~2]{ravenel_novice_guide_ANSS} for example, so there also has to be a differential in the signature spectral sequence of $\Smf$ hitting $\eta^4$. For degree reasons, the only things that can do this is a $d_3$ killing $h_1^4$, and $b \in\oppi_{5,1}\Smf/\tau$ is the only potential source.
\end{proof}

All other differentials follow from the Leibniz rule and the fact that $\eta^4$ is $\tau^2$-torsion. For example, we have $d_3(c_6)= \eta^3 c_4$. 

\subsubsection{Meta-arguments}
Computing homology with respect to the $d_3$-differentials, we have the following facts about the $\uE_4$-page.
These serve as base cases for the induction arguments used in the meta-arguments of \cref{ssec:meta}.

\begin{proposition}\label{prop:connectiveregiondivisiblebyg}
    Any element $x\in \uE_4$ in the connective region of filtration $\ge 4$ is divisible by~$g$.
\end{proposition}
\begin{proof}
    The computation of the connective region of the $\uE_2$-page is done by Bauer in \cite[Section 7]{bauer_tmf}. A straightforward consequence is that every $d_3$-cycle $x$ in the connective region of $\uE_2$ of filtration $\ge 4$ is divisible by either $g$ or $h_1^4$. The claim then follows from the differential of \cref{prop:d3differential}.
\end{proof}

\begin{proposition}\label{prop:E3nocrossing}
    There are no line-crossing $d_3$ differentials.
\end{proposition}
\begin{proof}
    For degree reasons, the only possible $d_3$'s crossing the line have source of the form $g^kh_1^3$, which is a permanent cycle. 
\end{proof}

\begin{proposition}\label{prop:E4metachecks}
    The class $\Delta^8$ is a $d_3$-cycle, and $\uE_4$ is $\Delta^8$-torsion free in the region $S$ of \cref{prop:Sregion}. Moreover, the following properties hold.
    \begin{itemize}
        \item A class $a\in \uE_4^{n,s}$ satisfies
        \[s<\tfrac{1}{5}(n-192)+\tfrac{12}{5}\]
        if and only if $a=\Delta^8b$, where $b$ is in the connective region.
        \item A class $a\in \uE_4^{n,s}$ in the $S$-region satisfies $a=\Delta^8b$ for $b$ in the $S$-region if and only if $n\ge 171$.
    \end{itemize}
\end{proposition}
\begin{proof}
    The class $\Delta^8$ is a $d_3$-cycle for degree reasons. It follows from \cite[Section 5.1]{konter_Tmf} that the region $S$ is $\Delta^8$-torsion free. The claim for $\uE_4$ then follows from the fact there is a single atomic $d_3$ that does not introduce $\Delta^8$-torsion.

    For the latter claims, it helps to consult a chart \cref{etwoprime2}; a larger version of this chart appears in \cite[Figure 25]{konter_Tmf}. The $\uE_4$-page is divided into $g$-periodic strips of width $24$ and slope $1/5$. Multiplication by $\Delta$ maps one strip isomorphically to the next within the region $S$, which implies the above claims.
\end{proof}

\subsubsection{Hidden extensions}

A crucial $2$-extension is also generated on the $\uE_3$-page.

\begin{lemma}[3,1]
    \label{lem:4nu_tau2_eta3}
    We have an isomorphism
    \[
        \oppi_{3,1} \Smf/\tau^{14} \cong \Z/8\angbr{\nu} \qquad \text{where} \quad \tau^2 \eta^3 = 4 \nu.
    \]
\end{lemma}
\begin{proof}
    This follows from the sphere: we claim that
    \[
        \oppi_{3,1} \S \cong \Z/8\angbr{\nu}.
    \]
    Indeed, using the $\F_2$-Adams spectral sequence for the sphere, we learn that $4\nu=\eta^3$ in the non-synthetic homotopy group $\oppi_3\S$.
    Because the Adams--Novikov spectral sequence for the sphere has no differentials hitting the 3-stem, we learn that $\oppi_{3,*}\S$ (referring to the $\MU$-synthetic sphere) is $\tau$-torsion free.
    For degree reasons, one therefore has the relation $4\nu=\tau^2\eta^3$.
    
    We defined the elements $\eta$ and $\nu$ in $\Smf$ to be the images of these respective elements in $\nu\S \to \Smf$, so we learn that $4\nu=\tau^2\eta^3$ holds in $\Smf$.
    It remains to show that $\nu$ generates $\oppi_{3,1}\Smf/\tau^{14}$. This follows from the fact that $\Smf/\tau^{14} \to \Smf/\tau$ induces an isomorphism on $\pi_{3,1}$ using the $d_3$'s already computed and \cref{prop:precise_lifting_higher_powers_tau}. Since we know that $\nu$ reduces to $h_2$ by \cref{heightonedetection}, this finishes the argument.
\end{proof}

\subsubsection{Lifts}

On later pages, we will need to work with precisely defined lifts of elements from $\Smf/\tau$ to higher $\Smf/\tau^k$.
These will serve as the way to express hidden extensions and total differentials.

We begin by lifting $\Delta$ from $\Smf/\tau$ to $\Smf/\tau^4$.
It does not lift to $\Smf/\tau^5$, as $\Delta$ turns out to support a $d_5$.
As such, the existence of this lift to $\Smf/\tau^4$ is a purely synthetic phenomenon.
This lift will be absolutely crucial to all of our computations going forward; see \cref{prop:dsevens}, for example.

\begin{lemma}[24,0]
    \label{lem:Delta_lift}
    The reduction map $\Smf/\tau^4 \to \Smf/\tau$ is an isomorphism on homotopy groups in degree $(24,0)$.
\end{lemma}
\begin{proof}
    The element $\Delta$ generates $\oppi_{24,0}\Smf/\tau$.
    We showed above that $\Delta$ is a $d_3$-cycle, so by evenness, it is also a $d_4$-cycle.
    This means that it lifts to $\oppi_{24,4}\Smf/\tau^4$, so that the reduction map is surjective.
    By \cref{cor:easy_lifting_higher_power_tau}, it is also injective: the homotopy groups of $\Smf/\tau$ vanish in bidegrees $(24,1)$, $(24,2)$ and $(24,3)$.
\end{proof}

\begin{notation}
    \label{not:Delta_lift}
    We write $\Delta \in \oppi_{24,0}\Smf/\tau^4$ for the unique lift of $\Delta \in \oppi_{24,0}\Smf/\tau$ guaranteed by \cref{lem:Delta_lift}.
\end{notation}

Our abuse of notation is mild, due to the uniqueness of the lift.
We will freely consider $\Delta$ as an element of $\oppi_{24,0}\Smf/\tau^4$ going forward.



\subsection{Page 5}

\subsubsection{Atomic differentials}

\begin{proposition}[24,0]
    \label{prop:d5_Delta}
    $d_5(\Delta) = \pm h_2g$.
\end{proposition}

\begin{proof}
We know that $\nu^3\kappabar$ is $\tau$-power torsion in $\Smf$ by \cref{cor:todabracketforkappabar}. Since $\nu$ and $\kappabar$ project to $h_2$ and $g$ mod $\tau$ respectively, and $h_2^3g\neq0$, it follows that $h_2^3g$ must be the target of a differential. The only possibility is $d_5(\Delta h_2^2) = g h_2^3$. In particular, by the Leibniz rule, this gives us $d_5(\Delta)= \pm g h_2$. 
\end{proof}

\subsubsection{Meta-arguments}

\begin{proposition}\label{prop:E5metachecks}
    The condition of \cref{prop:gdivisibilitycheck} holds for $d_5$. Moreover, $\Delta^8$ is a $d_5$-cycle.
\end{proposition}

\begin{proof}
    The condition of \cref{prop:gdivisibilitycheck} may be checked directly for the atomic $d_5$'s.
    Since $8d_5(\Delta)=0$, the Leibniz rule implies that $\Delta^8$ is a $d_5$-cycle.
\end{proof}

\begin{proposition}\label{cor:nolinecrossingE5}
    There are no line-crossing $d_5$-differentials.
\end{proposition}
\begin{proof}
    By \cref{prop:E4metachecks}, we may invoke the meta-argument of \cref{prop:gdivisibilitycheck2}, which implies that we only need to check this through the $17$-stem.
    The only possible atomic $d_5$'s crossing the line in this range have source $h_2$ or $d$, which are permanent cycles.
\end{proof}

\subsubsection{Lifts}

We now pick up our task of lifting elements with more vigour.
Unlike the class $\Delta$ in $\Smf/\tau^4$ from \cref{not:Delta_lift}, most of the classes here turn out to lift all the way to $\Smf$, but we will not need this.

\begin{lemma}[25,1), (97,1), (121,1]
    \label{lem:eta1_eta5_lift_tau8}
    The reduction maps $\Smf/\tau^8 \to \Smf/\tau$ induces an isomorphism on bigraded homotopy groups in degrees $(25,1), (97,1)$, and $(121,1)$.
\end{lemma}

The argument is a synthetic version of the statement that $h_1\Delta$ is a $d_{\leq 8}$-cycle for degree reasons: all potential targets either support or are hit by a shorter differential.

\begin{proof}
    We start with the first map.
    The target $\oppi_{25,1}\Smf/\tau$ is generated by $h_1\Delta$; we will first show that this generator lifts.
    This is equivalent to $\delta_1^8(h_1\Delta)=0$, so it suffices to show that $\oppi_{24,3}\Smf/\tau^7 = 0$.
    This follows from \cref{prop:precise_lifting_higher_powers_tau}, as the class in $(24,4)$ supports a $d_3$, while the classes in $(24,8)$ are hit by a $d_3$.
    
    It remains therefore only to show that the reduction map is injective.
    This too follows from an application of \cref{prop:precise_lifting_higher_powers_tau}: the nonzero elements in $\oppi_{25,5}\Smf/\tau$ support a $d_3$.

    The other two cases follow in the exact same way.
\end{proof}

\begin{notation}
    \label{not:lift_eta1_eta_4_eta5}
    \leavevmode
    \begin{itemize}
        \item \textup{(25,1)} We write $\eta_1 \in \oppi_{25,1}\Smf/\tau^8$ for the unique lift of $h_1\Delta \in \oppi_{25,1}\Smf/\tau$.
        \item \textup{(97,1)} We write $\eta_4 \in \oppi_{97,1}\Smf/\tau^8$ for the unique lift of $h_1\Delta^4 \in \oppi_{97,1}\Smf/\tau$.
        \item \textup{(121,1)} We write $\eta_5 \in \oppi_{121,1}\Smf/\tau^8$ for the unique lift of $h_1\Delta^5 \in \oppi_{121,1}\Smf/\tau$.
    \end{itemize}
\end{notation}

\begin{warning}
    The element $\eta_5$ does not lift beyond $\Smf/\tau^{22}$: as we will see later in \cref{prop:diff23}, the element $h_1\Delta^5$ supports a $d_{23}$.
\end{warning}

Next, we turn to $\Delta$-multiples of $h_2$.
Here we run into two problems.
First, the lift from $\Smf/\tau$ to $\Smf/\tau^{14}$ is not uniquely defined.
Second, in order to describe the group structure, we need to take the relation $4\nu=\tau^2\eta^3$ from \cref{lem:4nu_tau2_eta3} into account, but this is invisible to $\Smf/\tau$.

Both problems are solved by working with $\Smf/\tau^4$ instead of $\Smf/\tau$.
For the first problem, we instead lift $\Delta$-multiples of $\nu$ in $\Smf/\tau^4$.
As the class $\Delta$ in $\Smf/\tau^4$ is uniquely determined, this procedure also uniquely specifies these lifts.
For the second problem, we note that the relation $4\nu=\tau^2\eta^3$ is visible in $\Smf/\tau^4$.
We can transport this by multiplying by powers of $\Delta$.

\begin{lemma}
    \label{lem:27_1_mod_tau4}
    We have isomorphisms of abelian groups
    \begin{align*}
        \oppi_{27,1}\Smf/\tau^{4} &\cong \Z/8\angbr{\nu\Delta},\\
        \oppi_{51,1}\Smf/\tau^{4} &\cong \Z/8\angbr{\nu\Delta^2},\displaybreak[0]\\
        \oppi_{123,1}\Smf/\tau^{4} &\cong \Z/8\angbr{\nu\Delta^5},\\
        \oppi_{147,1}\Smf/\tau^{4} &\cong \Z/8\angbr{\nu\Delta^6}.
    \end{align*}
\end{lemma}
\begin{proof}
    \Cref{lem:4nu_tau2_eta3} implies that $\oppi_{3,1}\Smf/\tau^4 \cong \Z/8\angbr{\nu}$.
    We claim that multiplication by $\Delta \in \oppi_{24,0}\Smf/\tau^4$ induces an isomorphism
    \[
        \Delta \colon \oppi_{3,1}\Smf/\tau^4 \congto \oppi_{27,1}\Smf/\tau^4.
    \]
    To see this, note that multiplication by $\Delta$ induces an injection on $\Smf/\tau$ in degrees $(3,1+m)$ for $0\leq m\leq 3$.
    Moreover, there are no $d_{\leq 4}$-differentials entering $\oppi_{3,\,1+m}\Smf/\tau$, and the $\Delta$-multiples of these elements in $\oppi_{27,\,1+m}\Smf/\tau$ are also not hit by $d_{\leq4}$-differentials.
    It follows that multiplication by $\Delta$ induces an injection on the relevant groups appearing in \eqref{eq:cycles_mod_boundaries} in \cref{prop:precise_lifting_higher_powers_tau}.
    As a result, \cref{prop:precise_lifting_higher_powers_tau} implies that multiplication by $\Delta$ on $\Smf/\tau^4$ is an injection on degree $(3,1)$.
    Using the same \cref{prop:precise_lifting_higher_powers_tau}, we see that the target group is of the same size as the source, showing the map is indeed an isomorphism.
    The other cases follow similarly.
\end{proof}

\begin{lemma}
    \label{lem:checking_nu_i_lift}
    The reduction map $\Smf/\tau^{14} \to \Smf/\tau^4$ is injective on homotopy groups in bidegrees $(27,1)$, $(51,1)$, $(91,1)$, $(123,1)$ and $(147,1)$.
    Moreover, these maps can be identified, respectively, with
    \begin{align*}
        \SwapAboveDisplaySkip
        \Z/4 \angbr{x} \to \Z/8\angbr{\nu\Delta}, \quad &x \mapsto 2\nu\Delta,\\
        \Z/8 \angbr{y} \to \Z/8\angbr{\nu\Delta^2}, \quad &y \mapsto \nu\Delta^2,\displaybreak[0]\\
        \Z/8 \angbr{z} \to \Z/8\angbr{\nu\Delta^2}, \quad &z \mapsto \nu\Delta^4,\displaybreak[0]\\
        \Z/4 \angbr{w} \to \Z/8\angbr{\nu\Delta^5}, \quad &w \mapsto 2\nu\Delta^5,\\
        \Z/8 \angbr{t} \to \Z/8\angbr{\nu\Delta^6}, \quad &t \mapsto \nu\Delta^6.
    \end{align*}
    Finally, the reduction map $\Smf/\tau^{14}\to\Smf/\tau^{10}$ is an isomorphism in bidegrees $(51,1)$ and $(147,1)$.
\end{lemma}
\begin{proof}
    We study the case of the first of the five maps; the other ones are similar.
    We have an exact sequence
    \[
        \begin{tikzcd}
            \oppi_{27,5} \Smf/\tau \ar[r,"\tau^4"] & \oppi_{27,1}\Smf/\tau^5 \ar[r] & \oppi_{27,1} \Smf/\tau^4 \ar[r,"\delta_4^5"] & \oppi_{26,6} \Smf/\tau.
        \end{tikzcd}        
    \]
    The group on the left vanishes.
    The last map is equal to the $d_5$-differential by \cref{prop:total_differential_versus_differentials}\,(3): indeed, there are no shorter differentials entering or leaving bidegree $(26,6)$, so $\uE_5 = \uE_2$ in this bidegree.
    Since $d_5(2h_2\Delta) = 0$, this means that there is a unique lift of $2\nu\Delta$ from $\Smf/\tau^4$ to $\Smf/\tau^5$.
    Since $d_5(h_2\Delta) \neq 0$, the element $\nu \Delta$ does not lift.
    This means that $\oppi_{27,1}\Smf/\tau^5$ is isomorphic to $\Z/4$ and is generated by this unique lift of $2\nu\Delta$.
    
    Next, we have an exact sequence
    \[
        \begin{tikzcd}
            \oppi_{27,6} \Smf/\tau^9 \ar[r,"\tau^5"] & \oppi_{27,1}\Smf/\tau^{14} \ar[r] & \oppi_{27,1} \Smf/\tau^5 \ar[r,"\delta_5^{14}"] & \oppi_{26,7} \Smf/\tau^9.
        \end{tikzcd}
    \]
    The outer two groups vanish by \cref{prop:precise_lifting_higher_powers_tau}, so the middle map is an isomorphism.
    
    Finally, the last claim is checked using \cref{prop:precise_lifting_higher_powers_tau}.
\end{proof}

As before, we use a subscript to denote the power of $\Delta$ present in its mod $\tau$ reduction.
In some cases, the element cannot truly reduce to $h_2\Delta^i$, but has to differ by an element of small degree (in this case, 2).
Note that our notation does not indicate this.

\begin{notation}
    \label{not:lift_nu1_nu2_nu4_nu5_nu6}
    \leavevmode
    \begin{itemize}
        \item \textup{(27,1)} We write $\nu_1 \in \oppi_{27,1}\Smf/\tau^{14}$ for the unique lift of $2\nu\Delta \in \oppi_{27,1}\Smf/\tau^4$.
        \item \textup{(51,1)} We write $\nu_2 \in \oppi_{51,1}\Smf/\tau^{14}$ for the unique lift of $\nu\Delta^2 \in \oppi_{51,1}\Smf/\tau^4$.
        \item \textup{(99,1)} We write $\nu_4 \in \oppi_{99,1}\Smf/\tau^{14}$ for the unique lift of $\nu\Delta^4 \in \oppi_{99,1}\Smf/\tau^4$.
        \item \textup{(123,1)} We write $\nu_5 \in \oppi_{123,1}\Smf/\tau^{14}$ for the unique lift of $2\nu\Delta^5 \in \oppi_{123,1}\Smf/\tau^4$.
        \item \textup{(147,1)} We write $\nu_6 \in \oppi_{147,1}\Smf/\tau^{14}$ for the unique lift of $\nu\Delta^6 \in \oppi_{147,1}\Smf/\tau^4$.
    \end{itemize}
\end{notation}

\subsubsection{Relations}

\begin{lemma}[23,5]
    \label{lem:computation_23_5}
    Multiplication by $\kappabar$ induces an isomorphisms
    \begin{align*}
        \kappabar \colon \oppi_{3,1}\Smf/\tau^{14} &\congto \oppi_{23,5}\Smf/\tau^{14}\\
        \kappabar \colon \oppi_{99,1}\Smf/\tau^{14} &\congto \oppi_{119,5}\Smf/\tau^{14}
    \end{align*}
    In particular, the reduction maps $\Smf/\tau^{14} \to \Smf/\tau$ in degrees $(23,5)$ and $(119,5)$ can be identified with
    \begin{align*}
        \SwapAboveDisplaySkip
        \Z/8\angbr{\nu\kappabar} \twoheadto \Z/4\angbr{h_2g}, \qquad &\nu\kappabar \mapsto h_2g,\\
        \Z/8\angbr{\nu_4\kappabar} \twoheadto \Z/4\angbr{h_2g\Delta^4}, \qquad &\nu_4\kappabar \mapsto h_2g\Delta^4.
    \end{align*}
\end{lemma}
\begin{proof}
    In the same way as in \cref{lem:27_1_mod_tau4}, one can show that $\kappabar$ induces an isomorphism on $\Smf/\tau^4$ in degrees $(3,1)$ and $(99,1)$.
    Using \cref{prop:precise_lifting_higher_powers_tau}, we see that the reduction map $\Smf/\tau^{14}\to\Smf/\tau^4$ is injective in the relevant degrees, proving the claim.
\end{proof}

\subsubsection{Total differentials}

For many of the later pages, it will be crucial to compute a truncated total differential on~$\Delta$.
Given the $d_5$ of \cref{prop:d5_Delta}, a natural guess for $\delta_1^\infty(\Delta)$ would be $\nu \kappabar$.
Because we at this point do not know the fate of the elements in very high filtration, we cannot compute the entire differential, but only a truncated version.
While computing the $8$-truncated version would be enough to deduce $d_7(\Delta^4)$, it is hardly any more work to at this point record the $14$-truncated version, and this will be needed later on.
As $\Delta$ lifts uniquely to $\Smf/\tau^4$, we can work with $\delta_4^{14}$ instead.
This is both easier to compute, taking values in $\Smf/\tau^{10}$ rather than $\Smf/\tau^{13}$, and also gives more information, as we will see in our later computations.

In what follows, we think of $\nu\kappabar$ as defining an element in $\Smf/\tau^{10}$ via the reduction map.

\begin{proposition}[24,0]
    \label{lem:total_differential_14_Delta}
    We have $\delta_4^{14} (\Delta) = u \cdot \nu \kappabar$ in $\oppi_{23,5}\Smf/\tau^{10}$, where $u \in (\Z/8)^\times$.
\end{proposition}
\begin{proof}
    First we show that
    \[
        \delta_4^{8}(\Delta) = u \cdot \nu \kappabar
    \]
    where $u \in (\Z/8)^\times$ is a unit.
    To show this, we first claim that we have a commutative diagram
    \[
        \begin{tikzcd}
            \oppi_{24,0}\Smf/\tau^4 \ar[r,"\delta_4^8"] \ar[d] & \oppi_{23,5}\Smf/\tau^4\ar[d] \\
            \oppi_{24,0}\Smf/\tau \ar[r,"d_5"'] & \oppi_{23,5} \Smf/\tau,
        \end{tikzcd}
    \]
    where the vertical maps are the reductions mod $\tau$.
    Indeed, we use \cref{prop:total_differential_versus_differentials} together with the fact that the differentials $d_2,d_3,d_4$ vanish on $(24,0)$ and hit no elements in $(23,5)$, so that $\uE_5$ is the same as $\uE_2$ in these bidegrees.
    The group $\oppi_{24,0}\Smf/\tau$ is the free (2-local) abelian group generated by $\Delta$ and $c_4^3$.
    The element $c_4$ is a $d_{\leq10}$-cycle for degree reasons, so we can ignore this summand in further analysis.

    The left vertical map is an isomorphism by \cref{lem:Delta_lift}, and the right vertical map is surjective by \cref{lem:computation_23_5}.
    In other words, restricting to the $\Delta$-summand, the commutative diagram is of the form
    \[
        \begin{tikzcd}
            \Z_{(2)} \ar[r] \ar[rd,two heads] & \Z/8 \ar[d,two heads]\\
            & \Z/4.
        \end{tikzcd}
    \]
    This means the top horizontal map must send $1$ to a unit in $\Z/8$, which is exactly the claim about $\delta_4^8(\Delta)$.

    Using \cref{prop:precise_lifting_higher_powers_tau}, we see that the reduction $\oppi_{23,5}\Smf/\tau^{10} \to \oppi_{23,5}\Smf/\tau^4$ is injective.
    It is also surjective, as $h_2g$ is a $d_5$-boundary, so in particular a permanent cycle.
    The claim about $\delta_4^{14}(\Delta)$ therefore follows.
\end{proof}

\subsection{Page 7}\label{ssec:page7}

\subsubsection{Atomic differentials}

\begin{proposition}\label{prop:dsevens}
    \leavevmode
    \begin{itemize}
        \item \textup{(24,0)} $d_7(4\Delta) = h_1^3 g$.
        \item \textup{(48,0)} $d_7(2\Delta^2) = h_1^3 g \Delta$.
        \item \textup{(72,0)} $d_7(4\Delta^3) = h_1^3g\Delta^2$.
        \item \textup{(96,0)} $d_7(\Delta^4) = h_1^3 g \Delta^3$.
        \item \textup{(120,0)} $d_7(4\Delta^5) = h_1^3 g \Delta^4$.
        \item \textup{(144,0)} $d_7(2\Delta^6) = h_1^3 g \Delta^5$.
        \item \textup{(168,0)} $d_7(4\Delta^7) = h_1^3 g \Delta^6$.
    \end{itemize}
\end{proposition}
The atomic $d_7$-differentials on powers of $\Delta$ are difficult to deduce directly.
Our approach essentially deduces these from the $d_5$ on $\Delta$, through the use of the total differential on $\Delta$ from \cref{lem:total_differential_14_Delta}.
In all of this, the lift of $\Delta$ to $\Smf/\tau^4$ (\cref{not:Delta_lift}) is the crucial input to make the following arguments work.

\begin{proof}
    Let us start with $d_7(\Delta^4)$.
    First of all, from \cref{lem:total_differential_14_Delta}, we learn in particular that $\delta_4^8(\Delta) = u\cdot \nu\kappabar$.
    The synthetic Leibniz rule of \cref{syntheticleibnizrule} tells us that
    \[
        \delta_4^8(\Delta^4) = 4\Delta^3\cdot \delta_4^8(\Delta) = u\cdot 4\nu\kappabar\Delta^3. 
    \]
    Combining this with the relation $4\nu=\tau^2\eta^3$ from \cref{lem:4nu_tau2_eta3}, we learn that
    \[
        \delta_4^8(\Delta^4) = u\cdot\tau^2 \eta^3 \kappabar\Delta^3 = \tau^2 \eta^3 \kappabar \Delta^3,
    \]
    where we use that $\eta$ is 2-torsion to ignore the unit.
    From this, the differential $d_7(\Delta^4)=h_1^3g\Delta^3$ follows by \cref{prop:total_differential_versus_differentials}.

    The remaining differentials follow similarly using the synthetic Leibniz rule for $\delta_4^8$ and the same relation $4\nu = \tau^2\eta^3$.
\end{proof}

\subsubsection{Meta-arguments}

\begin{proposition}\label{prop:E7metachecks}
    The condition of \cref{prop:gdivisibilitycheck} holds for $d_7$. Moreover, $\Delta^8$ is a $d_7$-cycle.
\end{proposition}

\begin{proof}
    The condition of \cref{prop:gdivisibilitycheck} is checked directly as before.
    Since $2d_7(\Delta^4)=0$, the Leibniz rule implies that $\Delta^8$ is a $d_7$-cycle.
\end{proof}

\begin{proposition}\label{cor:nolinecrossingE7}
    There are no line-crossing $d_7$-differentials.
\end{proposition}
\begin{proof}
    By \cref{prop:E5metachecks}, we may invoke the meta-argument of \cref{prop:gdivisibilitycheck2}, which implies we only need to check for line-crossing differentials through the $25$-stem. The only possible atomic $d_7$ crossing the line in this range has source $2c_6$, which is a permanent cycle because it is a transfer; see \cref{survivalofmodularformsattwo}. 
\end{proof}

\subsubsection{Total differentials}

In the proof of \cref{prop:dsevens}, we computed the total differentials $\delta_4^8$ on powers of $\Delta$ using the synthetic Leibniz rule.
The synthetic Leibniz rule does not apply to the higher truncations $\delta_4^N$ for $N>8$, so we have to manually compute these total differentials.

\begin{proposition}
    \label{prop:total_diff_14_Delta2}
    There is a unit $u\in(\Z/8)^\times$ such that
    \begin{itemize}
        \item \textup{(48,0)} $\delta_4^{14}(\Delta^2) = u \cdot \nu_1 \kappabar$,
        \item \textup{(72,0)} $\delta_4^{14}(\Delta^3) = u\cdot 3\nu_2\kappabar$,
        \item \textup{(144,0)} $\delta_4^{14}(\Delta^6) = u \cdot 3\nu_5\kappabar$,
        \item \textup{(168,0)} $\delta_4^{14}(\Delta^7) = u \cdot 7\nu_6 \kappabar$.
    \end{itemize}
\end{proposition}
\begin{proof}
    Using \cref{prop:precise_lifting_higher_powers_tau}, we see that $\Smf/\tau^{14} \to \Smf/\tau^4$ is injective in the relevant bidegrees.
    As $\delta_4^{14}$ reduces to $\delta_4^8$ when taken mod $\tau^4$, we are reduced to the computation of $\delta_4^8$ in these degrees. This now follows from \cref{lem:total_differential_14_Delta} and the synthetic Leibniz rule \cref{syntheticleibnizrule}.
    For example, we have
    \[\delta_4^8(\Delta^2) = 2\Delta \cdot \delta_4^8(\Delta) = 2\Delta \cdot u \cdot \nu \kappabar = u \cdot 2\nu\Delta \kappabar \qquad \text{in } \oppi_{47,5} \Smf/\tau^4,\]
    and $\nu_1$ is defined as a lift of $2\nu\Delta \in \oppi_{27,1}\Smf/\tau^4$.
    The other cases are proved in the exact same way.
\end{proof}

\subsubsection{Relations}

To later be able to define Toda brackets, we will need relations like $\tau^4\nu\kappabar = 0$.
Such a relation is plausible because $h_2g$ is the target of a $d_5$-differential.
However, knowing this differential is not enough to deduce that $\tau^4\nu \kappabar = 0$ holds (see \cref{warn:diff_not_imply_tau_torsion}), so we have to check this relation by hand.

\begin{lemma}
    \label{lem:tau4_nu_kappabar_is_zero}
        \leavevmode
    \begin{itemize}
        \item \textup{(23,1)} $\tau^4\nu\kappabar = 0$ in $\Smf/\tau^{12}$.
        \item \textup{(119,1)} $\tau^4\nu_4\kappabar = 0$ in $\Smf/\tau^{12}$.
    \end{itemize}
\end{lemma}
\begin{proof}
    Recall from \cref{prop:d5_Delta} that $d_5(\Delta) = \pm h_2g$.
    By Omnibus \cref{thm:omnibus}\,(3), this means that there exists a $\tau^4$-torsion lift of $h_2g$ to $\oppi_{23,5}\Smf/\tau^{12}$.
    We claim that this must be $\nu\kappabar$ itself.
    To see this, we first use \cref{lem:computation_23_5} to conclude that $\nu\kappabar$ and $5\nu\kappabar$ are the only lifts of $h_2g$.
    At least one of them is therefore $\tau^4$-torsion; we claim that this implies that both are $\tau^4$-torsion.
    
    Since their difference is $4\nu\kappabar = \tau^2 \eta^3 \kappabar$, it suffices to establish that $\eta^3\kappabar$ in $\oppi_{23,6}\Smf/\tau^{12}$ is $\tau^6$-torsion.
    The element $h_1^3 g$ is hit by a $d_7$, so by Omnibus \cref{thm:omnibus}, there exists a $\tau^6$-torsion lift of it to $\Smf/\tau^{12}$.
    But $\eta^3\kappabar$ is the only element in $\oppi_{23,7}\Smf/\tau^{12}$ that lifts $h_1^3 g$: the mod $\tau$ reduction map is injective in this bidegree by \cref{prop:precise_lifting_higher_powers_tau}.
    We learn that $\eta^3 \kappabar$ must therefore be this $\tau^6$-torsion lift. 

    A similar argument applied to the differential $d_5(\Delta^5)=\pm h_2 g \Delta^4$ yields $\tau^4\nu_4\kappabar=0$.
\end{proof}

\begin{lemma}
    \label{lem:2_times_nu1}
    \leavevmode
    \begin{itemize}
        \item \textup{(27,1)} $2\nu_1 = \tau^2\eta^2 \eta_1$ in $\Smf/\tau^8$.
        \item \textup{(123,1)} $2\nu_5 = \tau^2\eta^2 \eta_5$ in $\Smf/\tau^8$.
    \end{itemize}
\end{lemma}

\begin{proof}
    Using \cref{lem:4nu_tau2_eta3}, we have the relation $4\nu\Delta = \tau^2\eta^3 \Delta$ in $\Smf/\tau^4$.
    This means that $2\nu_1 \in \oppi_{27,1}\Smf/\tau^{14}$ and $\tau^2 \eta^2 \eta_1 \in \oppi_{27,1}\Smf/\tau^8$ reduce to the same element in $\Smf/\tau^4$.
    It is therefore enough to establish that $\oppi_{27,1}\Smf/\tau^8 \to \oppi_{27,1}\Smf/\tau^4$ is injective, which follows from \cref{prop:precise_lifting_higher_powers_tau}. The exact same arguments apply to the second relation.
\end{proof}

\begin{corollary}
    \label{cor:tau4_nu1_kappabar_is_zero}
        \leavevmode
    \begin{itemize}
        \item \textup{(47,1)} $\tau^4\nu_1\kappabar=0$ in $\Smf/\tau^{12}$.
        \item \textup{(143,1)} $\tau^4\nu_5\kappabar=0$ in $\Smf/\tau^{12}$.
    \end{itemize}
\end{corollary}
\begin{proof}
    It suffices to prove this statement in $\Smf/\tau^6$, as the reduction map is injective in this bidegree by \cref{prop:precise_lifting_higher_powers_tau}.
    Similar to the argument of \cref{lem:27_1_mod_tau4}, we see that multiplication by $\kappabar$ induces an isomorphism from $\oppi_{27,1}\Smf/\tau^{4}$ to $\oppi_{47,5}\Smf/\tau^{4}$, and that $\oppi_{47,5}\Smf/\tau^{6}\to\oppi_{47,5}\Smf/\tau^4$ is injective.
    Since $\oppi_{47,5}\Smf/\tau^4$ is generated by $\kappabar\nu_1$, which lifts to $\Smf/\tau^{6}$, we therefore conclude that
    \[
        \oppi_{47,5}\Smf/\tau^{6} \cong \Z/4\angbr{\kappabar\nu_1}.
    \]
    The only lifts of $2gh_2\Delta$ are therefore $\kappabar\nu_1$ and $3\kappabar\nu_1$.
    Because $2gh_2\Delta$ is the target of a $d_5$-differential, it must have a $\tau^4$-torsion lift to $\Smf/\tau^{6}$.
    Using the relation from \cref{lem:2_times_nu1}, we see that their difference is $\tau^2$-divisible.
    In the same way as in \cref{lem:tau4_nu_kappabar_is_zero}, we can deduce from this that both lifts are $\tau^4$-torsion, so in particular, $\tau^4\kappabar\nu_1=0$. The proof of the equality $\tau^4 \nu_5 \kappabar = 0$ follows from the same arguments.
\end{proof}

\subsubsection{Lifts}

We now turn to lifting powers of $\Delta$ times $c$.
This is more difficult, as it involves ruling out a $d_9$ on them.
We cannot use degree arguments: it turns out that the potential target for these $d_9$'s each support a $d_{11}$.
Instead we give a Toda bracket argument to show that these elements lift.

\begin{lemma}\label{lem:epsexists}
    \leavevmode
    \begin{itemize}
        \item \textup{(32,2)} $\angbr{\nu_1, \nu, \eta}$ consists of one element in $\Smf/\tau^{14}$, which lifts $c\Delta$ in $\Smf/\tau$.
        \item \textup{(104,2)} $\angbr{\nu, 2\nu_4, \eta}$ consists of one element in $\Smf/\tau^{14}$, which lifts $c\Delta^4$ $\Smf/\tau$.
        \item \textup{(128,2)} $\angbr{\nu_1, \nu_4, \eta}$ consists of one element in $\Smf/\tau^{14}$, which lifts $c\Delta^5$ in $\Smf/\tau$.
    \end{itemize}
\end{lemma}

\begin{proof}
    We use \cref{prop:todagoestomassey} to verify these claims. One can easily check that the indeterminacy of the mod $\tau$ reductions of these brackets vanishes, as well as compute the values of the corresponding Massey products.
    An application of \cref{prop:precise_lifting_higher_powers_tau} tells us that the reduction map $\Smf/\tau^{14}\to\Smf/\tau$ is injective in these bidegrees.
    This means that the Toda brackets in $\Smf/\tau^{14}$ therefore also consist of singletons, ending the argument.
\end{proof}

\begin{notation}
    \label{not:lifts_eps1_eps_4_eps5}
    \leavevmode
    \begin{itemize}
        \item \textup{(32,2)} We write $\eps_1 = \angbr{\nu_1, \nu, \eta}$ in $\oppi_{32,2}\Smf/\tau^{12}$.
        \item \textup{(104,2)} We write $\eps_4 = \angbr{\nu, 2\nu_4, \eta}$ in $\oppi_{104,2}\Smf/\tau^{12}$.
        \item \textup{(128,2)} We write $\eps_5 = \angbr{\nu_1, \nu_4,\eta}$ in $\oppi_{128,2}\Smf/\tau^{12}$.
    \end{itemize}
\end{notation}

\subsubsection{Toda brackets}

\begin{lemma}
    \label{lem:toda_brackets_nu1_nu2}
    \leavevmode
    \begin{itemize}
        \item\textup{(27,1)} $\nu_1 = \angbr{\kappabar, \tau^4 \nu,2\nu}$ in $\Smf/\tau^{12}$.
        \item\textup{(51,1)} $\nu_2 = \angbr{\kappabar,\tau^4\nu_1,\nu}$ in $\oppi_{51,1}\Smf/\tau^{12}$.
        \item\textup{(123,1)} $\nu_5 = \angbr{\kappabar, \tau^4 \nu,2\nu_4}$ in $\Smf/\tau^{12}$.
        \item\textup{(147,1)} $\nu_6 = \angbr{\kappabar,\tau^4\nu_1,\nu_4}$ in $\Smf/\tau^{12}$.
        \item\textup{(39,3)} $\eta_1\kappa = \angbr{\nu_1, \nu, \eps}$ in $\Smf/\tau^{10}$.
        \item\textup{(111,3)} $\eta_4\kappa = \angbr{\nu, 2\nu_4, \eps}$ in $\Smf/\tau^{12}$.
        \item\textup{(135,3)} $\eta_5\kappa = \angbr{\nu_1, \nu_4, \eps}$ in $\Smf/\tau^{10}$.
        \item\textup{(117,3)} $\tau^2\eta_4\kappabar = \angbr{\nu, 2\nu_4, \kappa}$ in $\Smf/\tau^{12}$.
        \item\textup{(20,4)} $\pm 2\kappabar \in \angbr{\nu, \eta, \eta\kappa}$ in $\Smf/\tau^{10}$.
    \end{itemize}
\end{lemma}
\begin{proof}
    The brackets for $\nu_1$ and $\nu_5$ are nonempty by \cref{lem:tau4_nu_kappabar_is_zero}. These brackets also have zero indeterminacy as $\oppi_{7,-3}\Smf/\tau^{12}$ and $\oppi_{103,-3} \Smf/\tau^{12}$ both vanish and $\oppi_{24,0}\Smf/\tau^{12}$ and $\oppi_{120,0}\Smf/\tau^{12}$ are both $2\nu$-torsion. The associated Massey product on $\uE_5$ contains $2\nu\Delta$, so the synthetic version of Moss' theorem \cref{thm:synthetic_moss} shows that $\angbr{\kappabar, \tau^4\nu, 2\nu} = \nu_1$, our chosen lift of $2\nu\Delta$. Similarly, we see $\nu_5 = \angbr{\kappabar, \tau^4\nu, 2\nu_4}$. The brackets for $\nu_2$ and $\nu_6$ follow similarly, except to see they are nonempty one refers to \cref{cor:tau4_nu1_kappabar_is_zero} and the proof of \cref{lem:epsexists}.

    For the bracket expressions for $\eta_1\kappa$, $\eta_4\kappa$, and $\eta_5\kappa$, we use \cref{prop:todagoestomassey}. One uses \cref{prop:precise_lifting_higher_powers_tau} to see that the appropriate reduction map to $\Smf/\tau$ is injective in these degrees, so it suffices to work in $\Smf/\tau$. In this case, these brackets are easily seen to have no indeterminacy, and therefore follow from \cref{prop:toadyfromsphere} by multiplication by a power of $\Delta$. For $\tau^2\eta_4\kappabar$, we make the same arguments but only reduce to $\Smf/\tau^4$.

    For the last Toda bracket for $\pm 2\kappabar$, we again use \cref{prop:todagoestomassey} and the fact that the reduction map
    \[\oppi_{20,4} \Smf/\tau^{10} \to \Smf/\tau\]
    is injective by \cref{prop:precise_lifting_higher_powers_tau}, and surjective as the generator $g$ is hit by $\kappabar$ by \cref{not:kappabar_and_g}. We can now use Bauer's computation of $2g \in \angbr{h_2,h_1,h_1d}$ from \cite[Appendix~A]{bauer_tmf} together with \cref{prop:todagoestomassey} and \cref{etwopagebracketsprimethree} which validates this relation in $\Smf/\tau$.
\end{proof}

\subsubsection{Hidden extensions}

\begin{lemma}
    \label{lem:hidden_extensions_nu1_nu2}
    \leavevmode
    \begin{itemize}
        \item \textup{(28,2)} $\nu_1 \eta = \tau^4 \eps \kappabar$ in $\Smf/\tau^{12}$.
        \item \textup{(35,3)} $\nu_1 \eps = \tau^4 \eta \kappa \kappabar$ in $ \Smf/\tau^{12}$.
        \item \textup{(41,3)} $\nu_1 \kappa = \tau^6 \eta \kappabar^2$ in $\Smf/\tau^{12}$.
        \item \textup{(52,2)} $\nu_2 \eta = \tau^4 \eps_1 \kappabar$ in $\Smf/\tau^{12}$.
        \item \textup{(59,3)} $\nu_2 \eps = \tau^4 \eta_1\kappa \kappabar$ in $\Smf/\tau^{10}$.
        \item \textup{(124,2)} $\nu_5 \eta = \tau^4 \eps_4 \kappabar$ in $\Smf/\tau^{12}$.
        \item \textup{(131,3)} $\nu_5 \eps = \tau^4 \eta_4 \kappa \kappabar$ in $ \Smf/\tau^{12}$.
        \item \textup{(137,3)} $\nu_5 \kappa = \tau^6 \eta_4 \kappabar^2$ in $\Smf/\tau^{12}$.
        \item \textup{(148,2)} $\nu_6 \eta = \tau^4 \eps_5 \kappabar$ in $\Smf/\tau^{12}$.
        \item \textup{(155,3)} $\nu_6 \eps = \tau^4 \eta_5\kappa \kappabar$ in $\Smf/\tau^{10}$.
    \end{itemize}
\end{lemma}
\begin{proof}
The Toda bracket $\nu_1 = \angbr{\kappabar,\tau^4\nu, 2\nu}$ of \cref{lem:toda_brackets_nu1_nu2}, the shuffling formulas of \cref{prop:shufflingformulas}, the naturality of Toda brackets of \cref{lem:functorspreservebrackets}, and the Toda bracket $\eps \in \angbr{\nu,2\nu,\eta}$ of \cref{prop:toadyfromsphere} yield the relations
\[\nu_1 \eta = \angbr{\kappabar, \tau^4\nu,2\nu} \eta = \kappabar \angbr{\tau^4\nu,2\nu,\eta} \supseteq \kappabar \tau^4 \angbr{\nu,2\nu,\eta} \ni \tau^4 \eps \kappabar\]
in $\oppi_{28,2} \Smf/\tau^{12}$. Similarly, using the brackets $\nu_5 = \angbr{\kappabar, \tau^4\nu_4, 2\nu}$ and $\eps_4 = \angbr{\nu,2\nu_4,\eta}$ of \cref{lem:toda_brackets_nu1_nu2} and \cref{not:lifts_eps1_eps_4_eps5}, we obtain $\nu_5\eta = \tau^4\eps_4 \kappabar$.

The other extensions on $\nu_1$ and $\nu_2$ follow similarly, referring to \cref{prop:toadyfromsphere} and \cref{lem:toda_brackets_nu1_nu2} when necessary. The same goes for all of the extensions on $\nu_2$ and $\nu_6$, referring to \cref{lem:toda_brackets_nu1_nu2} and the definition of $\eps_1$ and $\eps_5$ from \cref{not:lifts_eps1_eps_4_eps5}. 
\end{proof}


\subsection{Page 9}

\subsubsection{Atomic differentials}

\begin{proposition}\label{prop:dnines}
    \leavevmode
    \begin{itemize}
        \item \textup{(49,1)} $d_9(h_1\Delta^2) = c g^2$.\footnote{As an alternative to the proof of this differential provided below, one can observe that
        \[\tau^8 \eps \kappabar^2 = \tau^8 \kappabar^2 \angbr{\nu, \eta, \nu} = \tau^4 \kappabar \angbr{\tau^4 \kappabar, \nu, \eta} \nu =0,\]
        and then apply \cref{thm:omnibus}.}
        \item \textup{(56,2)} $d_9(c\Delta^2) = h_1dg^2$.
        \item \textup{(73,1)} $d_9(h_1\Delta^3) = cg^2\Delta$.
        \item \textup{(81,2)} $d_9(c\Delta^3) = h_1dg^2\Delta$.
        \item \textup{(145,1)} $d_9(h_1\Delta^6) = cg^2\Delta^4$.
        \item \textup{(152,2)} $d_9(c\Delta^6) = h_1dg^2\Delta^4$.
        \item \textup{(169,1)} $d_9(h_1\Delta^7) = cg^2\Delta^5$.
        \item \textup{(176,2)} $d_9(c\Delta^7) = h_1dg^2\Delta^6$.
    \end{itemize}
\end{proposition}
\begin{proof}
    From \cref{prop:total_diff_14_Delta2} it follows that also $\delta_4^{10}(\Delta^2) = u \cdot \nu_1\kappabar$ for some $u\in(\Z/8)^\times$.
    Using \cref{lem:hidden_extensions_nu1_nu2}, we learn that
    \begin{gather*}
        \delta_4^{10}(\eta\Delta^2) = \eta \delta_4^{10}(\Delta^2) = u \cdot \eta \cdot \nu_1\kappabar = u \cdot \tau^4 \eps \kappabar^2\\
    \shortintertext{and}
        \delta_4^{10}(\eps\Delta^2) = u\cdot \eps \nu_1 \kappabar = u\cdot \tau^4 \eta\kappa\kappabar^2.
    \end{gather*}
    We obtain our desired $d_9$ on $h_1\Delta^2$ courtesy of \cref{prop:total_differential_versus_differentials}.
    The other differentials follow similarly using the hidden extensions of \cref{lem:hidden_extensions_nu1_nu2}.
\end{proof}

\subsubsection{Meta-arguments}\label{sssec:9nonconn}

\begin{proposition}\label{prop:E9metachecks}
    The condition of \cref{prop:gdivisibilitycheck} holds for $d_9$. Moreover, $\Delta^8$ is a $d_9$-cycle.
\end{proposition}

\begin{proof}
    The condition of \cref{prop:gdivisibilitycheck} is checked directly as before. The class $\Delta^8$ is a $d_9$-cycle for degree reasons.
\end{proof}

\begin{proposition}\label{cor:nolinecrossingE9}
    There are no line-crossing $d_9$-differentials.
\end{proposition}
\begin{proof}
    By \cref{prop:E7metachecks}, we may invoke the meta-argument of \cref{prop:gdivisibilitycheck2}, which implies we only need to check for line crossing differentials through the $32$-stem.
    The only possible atomic $d_9$'s crossing the line in this range have source $h_1$, $c$, $h_1\Delta$, or $c\Delta$.
    The first two are permanent cycles.
    The third is a $d_9$-cycle because the only possible target supports a $d_9$ differential.
    Lastly, the $d_9$ on $c\Delta$ is excluded by the earlier \cref{lem:epsexists} (combined with \cref{thm:omnibus}\,(1)).
\end{proof}

\subsubsection{Lifts}

In \cref{lem:eta1_eta5_lift_tau8}, we provided lifts $\eta_1$, $\eta_4$ and $\eta_5$ to $\Smf/\tau^8$.
We will need further lifts of two of these elements.

\begin{lemma}[25,1), (121,1]
    \label{lem:eta1_eta5_lift_tau10}
    The reduction maps $\Smf/\tau^{10} \to \Smf/\tau^8$ induces an isomorphism on bigraded homotopy groups in degrees $(25,1)$ and $(121,1)$.
\end{lemma}
\begin{proof}
    Our computation of the $d_9$-differentials shows that $h_1\Delta$ is a $d_{\leq 10}$-cycle.
    By \cref{thm:omnibus}, it therefore lifts to $\Smf/\tau^{10}$; since $\eta_1$ is by \cref{lem:eta1_eta5_lift_tau8} the unique lift of $h_1\Delta$ to $\Smf/\tau^8$, it follows that $\eta_1$ lifts to $\Smf/\tau^{10}$.
    It follows from \cref{prop:precise_lifting_higher_powers_tau} that the reduction map $\Smf/\tau^{10}\to\Smf/\tau^8$ is injective on bidegree $(25,1)$, proving the claim.
    The case for bidegree $(121,1)$ and lifting $h_1\Delta^5$ is the same.
\end{proof}

\begin{notation}
    \label{not:lift_eta1_eta5_validity_10}
    \leavevmode
    \begin{itemize}
        \item We write $\eta_1 \in \oppi_{25,1}\Smf/\tau^{10}$ for the unique lift of $\eta_1 \in \oppi_{25,1}\Smf/\tau^8$.
        In particular, it is also the unique lift of $h_1\Delta\in \oppi_{25,1}\Smf/\tau$.
        \item We write $\eta_5 \in \oppi_{121,1}\Smf/\tau^{10}$ for the unique lift of $\eta_5 \in \oppi_{121,1}\Smf/\tau^8$.
        In particular, it is also the unique lift of $h_1\Delta^5\in \oppi_{121,1}\Smf/\tau$.
    \end{itemize}
\end{notation}

\begin{lemma}[116,4]
    \label{lem:kappabar4_exists}
    The reduction maps $\Smf/\tau^{20} \to \Smf/\tau^{10}$ and $\Smf/\tau^{10}\to \Smf/\tau$ are injective on homotopy groups in degree $(116,4)$.
    Moreover, the map $\oppi_{116,4}\Smf/\tau^{20}\to\oppi_{116,4}\Smf/\tau$ can be identified with
    \[
        \Z/4\angbr{x} \hookto \Z/8\angbr{g\Delta^4}, \quad x \mapsto 2g\Delta^4.
    \]
\end{lemma}
\begin{proof}
    It follows from \cref{prop:precise_lifting_higher_powers_tau} that the maps are injective, so it remains to be shown that $2g\Delta^4$ lifts to $\Smf/\tau^{20}$, while $g\Delta^4$ does not lift to $\Smf/\tau^{10}$.
    The latter claim is clear, as it supports a $d_7$.
    To see that $2g\Delta^4$ lifts to $\Smf/\tau^{20}$, it suffices to show that $\delta_4^{20}(2\kappabar\Delta^4)=0$, for which it in turn suffices to show that $\oppi_{115,9}\Smf/\tau^{16}$ vanishes.
    This follows from another application of \cref{prop:precise_lifting_higher_powers_tau}.
\end{proof}

\begin{notation}
    \label{not:kappabar4_lift}
    We write $\kappabar_4\in\oppi_{116,4}\Smf/\tau^{20}$ for the unique lift of $2g\Delta^4 \in \oppi_{116,4}\Smf/\tau$.
    Note that it is also the unique lift to $\Smf/\tau^{10}$ of $2g\Delta^4$.
\end{notation}

\subsubsection{Toda brackets}

\begin{lemma}
    \label{lem:toda_brackets_etai}
    \leavevmode
    \begin{itemize}
        \item\textup{(25,1)} $\eta_1 \in \angbr{\kappabar,\tau^4\nu,\eta}$ in $\Smf/\tau^{10}$. This Toda bracket has indeterminacy given by the subgroup of $\oppi_{25,1} \Smf/\tau^{10}$ spanned by the $\kappa$-torsion classes.
        \item\textup{(121,1)} $\eta_5 \in \angbr{\kappabar,\tau^4\nu_4,\eta}$ in $\Smf/\tau^{10}$. This Toda bracket has indeterminacy given by the subgroup of $\oppi_{121,1} \Smf/\tau^{10}$ spanned by the $\kappa$-torsion classes.
        \item\textup{(116,4)} $\pm \kappabar_4 \in \angbr{\nu_4, \eta, \eta\kappa}$ in $\Smf/\tau^{10}$.
    \end{itemize}
\end{lemma}

\begin{proof}
    The first two brackets are nonempty by \cref{lem:tau4_nu_kappabar_is_zero}, and it straightforward to compute the indeterminacies.
    The Massey products on the $\uE_5$-page associated with these Toda brackets contains $h_1\Delta$ and $h_1\Delta^5$, respectively, hence the synthetic Moss' theorem \cref{thm:synthetic_moss} shows that $\angbr{\kappabar, \tau^4\nu, \eta}$ contains $\eta_1$ and $\angbr{\kappabar, \tau^4 \nu_4, \eta}$ contains $\eta_5$.

    For the Toda bracket expression of $\pm \kappabar_4$, we note that by \cref{not:kappabar4_lift}, it suffices to show that this Toda bracket contains a lift of $\pm 2g\Delta^4$.
    We will show that $\pm2g\Delta^4 \subseteq \angbr{h_2\Delta^4, h_1, h_1d}$.
    It follows from \cref{prop:todagoestomassey}and the bracket $\pm 2\kappabar \in \angbr{\nu, \eta, \eta\kappa}$ of \cref{prop:toadyfromsphere} that $\pm 2g \in \angbr{h_2, h_1, h_1d}$, and it follows that 
    \[\pm 2 g \Delta^4 \in \pm \Delta^4 \angbr{h_2, h_1, h_1d} \subseteq \angbr{h_2\Delta^4, h_1, h_1 d}.\qedhere\]
\end{proof}

\subsubsection{Hidden extensions}

\begin{lemma}
    \label{lem:hidden_extensions_etai}
    \leavevmode
    \begin{itemize}
        \item \textup{(40,4)} $\eta \eta_1 \kappa = \pm \tau^4 2\kappabar^2$ in $\Smf/\tau^{10}$.
        \item \textup{(136,4)} $\eta \eta_5 \kappa = \pm \tau^4 \kappabar \kappabar_4$ in $\Smf/\tau^{10}$.
    \end{itemize}
\end{lemma}

\begin{proof}
The formula for $\eta\eta_1\kappa$ follows from the Toda bracket expressions of \cref{prop:toadyfromsphere} and \cref{lem:toda_brackets_nu1_nu2}:
\[\eta\eta_1 \kappa = \angbr{\kappabar, \tau^4\nu, \eta}\eta \kappa = \kappabar \angbr{\tau^4\nu, \eta, \eta\kappa} \supseteq  \tau^4\kappabar \angbr{\nu, \eta, \eta\kappa} \ni \pm \tau^4 2\kappabar^2.\]
Here the second equality follows from the fact that the indeterminacy of $\angbr{\kappabar, \tau^4\nu, \eta}$ is all $\kappa$-torsion. The equality for $\eta\eta_5\kappa$ is the same, except we refer only to \cref{lem:toda_brackets_nu1_nu2}.
\end{proof}


\subsection{Page 11}\label{ssec:page11}

\subsubsection{Atomic differentials}

\begin{proposition}
    \leavevmode
    \begin{itemize}
        \item \textup{(62,2)} $d_{11}(d\Delta^2) = h_1 g^3$.
        \item \textup{(158,2)} $d_{11}(d\Delta^6) = h_1 g^3\Delta^4$.
    \end{itemize}
\end{proposition}
\begin{proof}
    This follows from the total differentials in \cref{prop:total_diff_14_Delta2} and the $\kappa$-extensions of \cref{lem:hidden_extensions_nu1_nu2}.
\end{proof}

\subsubsection{Meta-arguments}

\begin{proposition}\label{prop:E11metachecks}
    The conditions of \cref{prop:gdivisibilitycheck} hold for $d_{11}$. Moreover, $\Delta^8$ is a $d_{11}$-cycle.
\end{proposition}

\begin{proof}
    The condition of \cref{prop:gdivisibilitycheck} is checked directly as before.
    The class $\Delta^8$ is a $d_{11}$-cycle for degree reasons.
\end{proof}

\begin{proposition}\label{cor:nolinecrossingE11}
    There are no line-crossing $d_{11}$-differentials.
\end{proposition}
\begin{proof}
    By \cref{prop:E9metachecks}, we may invoke the meta-argument of \cref{prop:gdivisibilitycheck2}, which implies we only need to check for line crossing differentials through the $40$-stem. The only possible atomic $d_{11}$'s crossing the line in this range have sources $c_4$ or $h_1^3\Delta$, but the only possible target supports a $d_{11}$ differential in both cases. 
\end{proof}

\subsubsection{Hidden extensions}

To establish the $d_{13}$-differentials, we need some hidden extensions.
These turn out to require four-fold Toda brackets, and computing these is a delicate matter.
We provide a very detailed and careful analysis of these in \cref{sec:large_carrick_brackets}.
Using these, we now apply various shuffling formulas to obtain the following.

\begin{lemma}
    \label{prop:nu2_times_2_nu_hidden_extension}
    \leavevmode
    \begin{itemize}
        \item \textup{(54,2)} $\nu_2\cdot 2\nu = \tau^8\, \kappabar^2\, \widetilde{d}$ in $\Smf/\tau^{10}$, where $\widetilde{d}\in\oppi_{14,2}\Smf/\tau^{10}$ is an element that is sent to $d$ under the map $\Smf/\tau^{10}\to\Smf/\tau$.
        \item \textup{(150,2)} $\nu_6\cdot 2\nu = \tau^8 \, \kappabar^2\, \widetilde{d}_4$ in $\Smf/\tau^{10}$, where $\widetilde{d}_4\in\oppi_{110,2}\Smf/\tau^{10}$ is an element that is sent to $d\Delta^4$ under the map $\Smf/\tau^{10}\to\Smf/\tau$.
    \end{itemize}
\end{lemma}

\begin{proof}
    In \cref{prop:nu2_fourfold_carrick_bracket}, we show that
    \[
        \nu_2 = \angbr{\nu,\, 2\nu\tau^4,\, \nu\tau^4,\, \kappabar^2} \quad \text{in }\oppi_{51,1}\Smf/\tau^{10}.
    \]
    The Toda brackets $\angbr{2\nu,\nu,2\nu\tau^4}$ and $\angbr{\nu,2\nu,\nu\tau^4}$ are strictly zero, so that the four-fold Toda bracket $\angbr{2\nu,\nu,2\nu\tau^4,\nu\tau^4}$ is nonempty and we may apply the shuffling formula of \cref{prop:shufflingformulas} to see that
    \[
        2\nu\cdot\nu_2=2\nu\angbr{\nu,\, 2\nu\tau^4,\,\nu\tau^4,\,\kappabar^2}=\angbr{2\nu,\, \nu,\, 2\nu\tau^4,\, \nu\tau^4}\kappabar^2.
    \]
    In particular, we also see that the set $\angbr{2\nu,\nu,2\nu\tau^4,\nu\tau^4}\kappabar^2$ is a singleton.
    Applying the shuffling formula of \cref{prop:shufflingformulas}, one has
    \[\angbr{2\nu,\nu,2\nu,\nu}\tau^8\subseteq \angbr{2\nu,\nu,2\nu,\nu\tau^8}\subseteq \angbr{2\nu,\nu,2\nu\tau^4,\nu\tau^4}.\]
    Since $\angbr{2\nu,\nu,2\nu\tau^4,\nu\tau^4}\kappabar^2$ is a singleton, we are reduced to showing that every element of $\angbr{2\nu,\nu,2\nu,\nu}$ projects to $d$ mod $\tau$.
    However, using \cref{prop:todagoestomassey}, any class in $\angbr{2\nu,\nu,2\nu,\nu}$ projects to $d\in\oppi_{14,2}\Smf/\tau$, as the latter has zero indeterminacy and contains $d$ by \cite[Equation~(7.14)]{bauer_tmf}.

    The expression for $2\nu \cdot \nu_6$ follows by similar arguments. First, we use the Toda bracket expression
    \[\nu_6 = \angbr{\nu_4, 2\nu\tau^4, \nu\tau^4, \kappabar^2}\]
    of \cref{prop:nu2_fourfold_carrick_bracket} and the shuffling formula of \cref{prop:shufflingformulas}
    \[2\nu \cdot \nu_6 = 2\nu \angbr{\nu_4, 2\nu\tau^4, \nu\tau^4, \kappabar^2} = \angbr{2\nu, \nu_4, 2\nu\tau^4, \nu\tau^4} \kappabar^2.\]
    Combining this with the the containments
    \[d\Delta^4 \in \Delta^4 \angbr{2h_2, h_2, 2h_2, h_2} \subseteq \angbr{2h_2\Delta^4, h_2, 2h_2, h_2} \subseteq \angbr{2h_2, h_2\Delta^4, 2h_2, h_2}\]
    in $\Smf/\tau$ yields the result.
\end{proof}


\subsection{Page 13}\label{ssec:page13}

\subsubsection{Atomic differentials}

\begin{proposition}\label{prop:d13}
        \leavevmode
    \begin{itemize}
        \item \textup{(75,1)} $d_{13}(2h_2\Delta^3) = dg^3$.
        \item \textup{(81,3)} $d_{13}(h_2^3\Delta^3) = \pm 2g^4$.
        \item \textup{(171,1)} $d_{13}(2h_2\Delta^7) = dg^3\Delta^4$.
        \item \textup{(177,3)} $d_{13}(h_2^3\Delta^7) = \pm 2g^4\Delta^4$.
    \end{itemize}
\end{proposition}

\begin{proof}
    Recall the total differentials of \cref{prop:total_diff_14_Delta2}
    \begin{align*}
        \delta^{14}_4(\Delta^3) = u \cdot 3\nu_2\kappabar \qquad &\text{and}\qquad \delta^{14}_4(\Delta^7) = u \cdot 7\nu_6\kappabar
    \intertext{and the extensions of \cref{prop:nu2_times_2_nu_hidden_extension}}
        \nu_2 \cdot 2\nu = \tau^8\, \kappabar^2\, \widetilde{d} \qquad &\text{and}\qquad \nu_6 \cdot 2\nu = \tau^8\, \kappabar \, \widetilde{d}_4,
    \end{align*}
    where $\widetilde{d}$ is some element whose mod~$\tau$ reduction is $d$, and where $\widetilde{d}_4$ is some element whose mod~$\tau$ reduction is $d\Delta^4$.
    Together these give
    \[
        \delta_4^{14}(2\nu \Delta^3) = 3u\cdot 2\nu \cdot \nu_2 \kappabar = 3u\cdot \tau^8\,\widetilde{d} \,\kappabar^3,
    \]
    and similarly $\delta_4^{14}(2\nu\Delta^7) = 7u\cdot \tau^8 \, \widetilde{d}_4 \, \kappabar^3$.
    Combining these equalities with \cref{prop:total_differential_versus_differentials} gives the differentials supported in degrees $(75,1)$ and $(171,1)$.

    For the other two, first recall the classical relation $\nu^3 = \eta\eps$ in the non-synthetic sphere; 
    see \cite[Theorem~3.3.15\,(a)]{kochman_stable}.
    This relation immediately lifts to the synthetic sphere $\Sph$ as there is no $\tau$-power torsion in $\oppi_{8,2}\S$, so we also have $\eta\eps = \nu^3$ in $\Smf$ as well as $h_2^3 = h_1 c$ in $\Smf/\tau$. Using the extensions
    \[\eps \nu_2 = \tau^4\eta_1 \kappa\kappabar\qquad \text{and}\qquad \eta \eta_1 \kappa = \pm \tau^4 2\kappabar\]
    of \cref{lem:hidden_extensions_nu1_nu2} and \cref{lem:hidden_extensions_etai}, respectively, we obtain the total differential
    \[\delta_4^{14}(\nu^3 \Delta^3) = 3u\cdot \eta \eps \cdot \nu_2 \kappabar = 3u \cdot (\tau^4 \eta\eta_1 \kappa \kappabar) \kappabar = 3u \cdot \tau^4 \kappabar^2 (\eta \eta_1 \kappa) 
    = \pm 2 \tau^8 \kappabar^4.\]

    Similarly, we have
    \[\delta_4^{14}(\nu^3 \Delta^7) = 7u\cdot \eta \eps \cdot \nu_6 \kappabar = 7u \cdot (\tau^4 \eta\eta_5 \kappa \kappabar) \kappabar = 7u \cdot \tau^4 \kappabar^2 (\eta \eta_5 \kappa) 
    = \pm \tau^8 \kappabar_4\kappabar^3\]
    using \cref{lem:hidden_extensions_nu1_nu2} and \cref{lem:hidden_extensions_etai}. These two total differentials combined with \cref{prop:total_differential_versus_differentials} yield the remaining two atomic $d_{13}$'s.
\end{proof}

\subsubsection{Meta-arguments}

\begin{proposition}\label{prop:E13metachecks}
    The condition of \cref{prop:gdivisibilitycheck} holds for $d_r$ for $13\le r\le 21$. Moreover, $\Delta^8$ is a $d_{21}$-cycle.
\end{proposition}

\begin{proof}
    The condition of \cref{prop:gdivisibilitycheck} is checked directly as before. The class $\Delta^8$ is a $d_{21}$-cycle for degree reasons.
\end{proof}

\begin{proposition}\label{cor:nolinecrossingE13}
    There are no line-crossing $d_{r}$-differentials for $13\le r\le 21$.
\end{proposition}
\begin{proof}
    By \cref{prop:E11metachecks}, we may invoke the meta-argument of \cref{prop:gdivisibilitycheck2} for the case $r=13$, which implies we only need to check for line-crossing differentials through the $46$-stem, and there are no possibilities in this range.

    For $15\le r\le 21$, \cref{prop:E13metachecks} implies we only need to check for line-crossing differentials through the $96$-stem, and again there are no possibilities in this range.
\end{proof}

\subsubsection{Lifts}

Before we start the computation of the $\uE_{23}$-page, we need to lift a relation from $\Smf/\tau^{12}$ to $\Smf/\tau^{24}$.

\begin{lemma}[23,1]
\label{lem:tau4_nu_kappabar_is_zeroweiter}
    $\tau^4\nu\kappabar = 0$ in $\Smf/\tau^{24}$.
\end{lemma}

\begin{proof}
    We know this relation holds in $\Smf/\tau^{12}$ by \cref{lem:tau4_nu_kappabar_is_zero}, so it suffices to show that the reduction map
    \[\oppi_{23,1} \Smf/\tau^{24} \to \oppi_{23,1} \Smf/\tau^{12}\]
    is injective.
    This now follows from \cref{prop:precise_lifting_higher_powers_tau}; the nontrivial item to check is that there is a $d_{9}$ hitting the class in filtration 19, and this is a consequence of \cref{sssec:9nonconn}.
\end{proof}


\subsection{Page 23}\label{ssec:page23}

\subsubsection{Atomic differentials}

\begin{proposition}[121,1]\label{prop:diff23}
    $d_{23}(h_1\Delta^5) = g^6$.
\end{proposition}
\begin{proof}
    First, we claim that by \cref{thm:omnibus}, it suffices to show that $\tau^{22}\kappabar^6 = 0$ in $\Smf/\tau^{24}$.
    Indeed, the element $g^6$ is a $d_{\leq 22}$-cycle and the only potential source of a differential to hit $g^6$ is $h_1\Delta^5$.
    Next, we claim that all of the proper sub-brackets of the four-fold Toda brackets
    \[
        \angbr{\kappa,\, 2,\, \eta,\, \nu} \qquad \text{and} \qquad \angbr{\tau^{16}\kappabar^4,\, \kappa,\, 2,\, \eta}
    \]
    are equal to zero in $\Smf/\tau^{24}$.
    This computation for the first bracket is straightforward. 
    For the second bracket, we need to use the key fact that $\oppi_{95,1}\Smf/\tau^{24}=0$.
    This follows from \cref{prop:precise_lifting_higher_powers_tau}, using the $d_7$ on $\Delta^4$ of \cref{prop:dsevens}.
    Lastly, we have that $\tau^2 \kappabar \in \angbr{\kappa, 2, \eta, \nu}$ in $\Smf/\tau^{24}$, which follows from \cref{cor:todabracketforkappabar}.
    
    The fact that these sub-brackets are strictly zero allows us to apply the shuffling formula of \cref{prop:shufflingformulas}\,(1), and we obtain
    \[
        \tau^{22}\kappabar^6 = \tau^{16}\kappabar^4\cdot(\tau^2\kappabar)\cdot\tau^4\kappabar \in \tau^{16}\kappabar^4 \cdot \angbr{\kappa, 2, \eta, \nu} \cdot \tau^4 \kappabar = \angbr{\tau^{16}\kappabar^4, \kappa, 2, \eta} \cdot \tau^4 \nu \kappabar = 0,
    \]
    where for the last equality we used the relation $\tau^4 \nu\kappabar = 0$ from \cref{lem:tau4_nu_kappabar_is_zeroweiter}.
\end{proof}

All other $d_{23}$'s in the connective region follow from the Leibniz rule as $g$, $[\Delta^8]$, and $[h_1\Delta]$ are $d_{\leq23}$-cycles.

\subsubsection{Meta-arguments}

\begin{proposition}\label{prop:E23metachecks}
    The conditions of \cref{prop:gdivisibilitycheck} hold for $d_{23}$. 
\end{proposition}

\begin{proof}
    The condition is checked directly.
\end{proof}

\begin{proposition}\label{cor:nolinecrossingE23}
    There are no line-crossing $d_{23}$-differentials.
\end{proposition}
\begin{proof}
    By \cref{prop:E13metachecks}, we may invoke the meta-argument of \cref{prop:gdivisibilitycheck2}, which implies we only need to check for line crossing differentials through the $110$-stem. There are no possible line-crossing $d_{23}$'s in this range.
\end{proof}

\begin{proposition}\label{prop:vanishingE24}
    The groups $\uE_{24}^{n,s}$ vanish for $-21<n<0$ and all $s$, and also for all $(n,s)$ with $n\ge0$ and $s>23$. In particular, there is no nontrivial differential of length $>23$ in the DSS whose source lives in bidegree $(n,s)$ for $n>-21$, and $\Delta^8$ is a permanent cycle.
\end{proposition}
\begin{proof}
    The region $S$ of \cref{not:S_region} on $\uE_{24}$ is $\Delta^8$-torsion free by \cref{prop:Sregion,prop:E23metachecks}. It therefore suffices to check the claim in the connective region, by multiplying any class with a power of $\Delta^8$. However, by inductively applying \cref{prop:gdivisibilitycheck}, every element in the connective region on $\uE_{24}$ of filtration $\ge 24$ is divisible by $g^6$, which is zero by \cref{prop:diff23}.
\end{proof}


\subsection{Stems below \texorpdfstring{$-21$}{minus 21}}\label{ssec:nonconnectiveregion}
The arguments above yield the DSS for $\Tmf$ in stems $n>-21$.
In particular, the $2$-primary Gap \cref{thm:gap} follows from \cref{prop:vanishingE24}.
In this section we compute the spectral sequence in stems $n\leq -21$.
We do this in less detail, as this has no bearing on the Gap Theorem.

Strictly speaking, the differentials in all negative stems follow from the $\Delta^{8}$- and $\kappabar$-linearity of differentials.
However, one needs to take care, as crossing the line defined by $n+s\leq -12$ can cause $d_r$-differentials to turn into $d_{r\pm 2}$-differentials. 
This phenomenon is only briefly touched upon in \cite{konter_Tmf}, so we give some more details here as well as a synthetic interpretation of this stretching.

Outside the connective region --- in particular, away from the image of $\Phi$ of \cref{etwopagebracketsprimethree} --- we will use the notation of Konter \cite{konter_Tmf}, which we now recall.
From the decomposition of the stack $\Mellbar$ as the pushout
\[
    \Mellbar = D(\Delta) \cup D(c_4)
\]
and the discussion in \cref{e2pagesection}, we obtain a pullback of $\E_\infty$-$C\tau$-algebras
\[\begin{tikzcd}
    {\Smf/\tau}\ar[r]\ar[d] \pullback & {(\Smf/\tau)[\Delta^{-1}]}\ar[d] \\
    {(\Smf/\tau) [c_4^{-1}]}\ar[r] & (\Smf/\tau)[c_4^{-1},\Delta^{-1}].
\end{tikzcd}\]

\begin{notation}
The above pullback yields a fibre sequence of synthetic spectra
\[
    \begin{tikzcd}
        \Smf/\tau \rar["(i{,}j)"] & \Smf/\tau[c_4^{-1}]\oplus \Smf/\tau[\Delta^{-1}] \rar & \Smf/\tau[c_4^{-1},\Delta^{-1}].
    \end{tikzcd}
\]
The notation $[x]$ refers to the image of $x$ under the boundary map
\[
    \oppi_{n,s} \Smf/\tau[c_4^{-1},\Delta^{-1}] \to \oppi_{n-1,\, s+1} \Smf/\tau
\]
in the above cofibre sequence.
We use the notation $\angbr{x}$ to denote an element of $\oppi_{n,s}\Smf$ whose image under $(i,j)$ is the element $(0,x)$, where $x$ is $c_4$-torsion.
This means the symbol $\angbr{x}$ in general does not denote a uniquely defined element, but it will in our cases of interest.
In particular, notice that
\[
    \abs{[x]}=\abs{x}+(-1,1) \qquad \text{and} \qquad \abs{\angbr{x}} = \abs{x}.
\]
There are also generators in filtration $1$ which double or quadruple to classes of the form $[x]$. For these generators we write $[\tfrac{1}{n}x]$, where $n=2$ or $4$, where appropriate.
\end{notation}

\subsubsection{Page 3}

We begin with the atomic $d_3$'s.

\begin{proposition}\label{noncon:d3}
    For $x\in \oppi_{n,s}\Smf/\tau$ of the form
    \[x = h_1^j[c_4^{-k}c_6^l \Delta^{-m}] \qquad \text{for } j,k,m\geq 0 \text{ and } l = 0,1, \]
    such that $h_1^4 x \neq 0$ and $bx=0$ \textbr{where $b$ is the generator of $\oppi_{5,1}\Smf/\tau$}, then
    \[d_3(x) = h_1^{j+3}[c_4^{-k+1}\Delta^{-m}].\]
\end{proposition}

Konter states these differentials implicitly in \cite[page~34]{konter_Tmf}. The proof below corrects some notation from his justification.

\begin{proof}
    If $x$ is a class of the above form and $b$ in the nonzero class in degree $(5,1)$, then the Leibniz rule states that
    \[d_3(bx) = d_3(b)x + bd_3(x).\]
As $d_3(b)=h_1^4$ by \cref{prop:d3differential}, we see that for such classes $x$ with $h_1^4 x\neq 0$ and $bx=0$, we have $d_3(x)\neq 0$. In particular, $x$ supports a $d_3$ whenever $xb\neq 0$ and $l=1$ and $xb\neq 0$ as $c_6 b=0$. This yields the desired $d_3$-differential.
\end{proof}

All other $d_3$'s in the nonconnective region follow from the Leibniz rule. As there are no charts of this, let us also note that classes of the form $h_1^j[c_4^{-k}\Delta^{-m}]$ for various $j\geq 2$ do not support $d_3$'s. The prototypical examples of this are $h_1^2[c_4^{-1}\Delta^{-1}]$ in bidegree $(-31,3)$ and $h_1^6[c_4^{-2}\Delta^{-1}]$ in bidegree $(-35,7)$. Indeed,  if these classes did support differentials, then the Leibniz rule would imply that $d_3(h_1[c_4^{-k}\Delta^{-m}]) = h_1^4 [c_4^{-k-2}c_6\Delta^{-m}]$, a contradiction as these potential targets support differentials by \cref{noncon:d3}.

This yields the $\uE_5$-page of \cref{efiveprime2_part1} and \cite[Figure~27]{konter_Tmf}. 

\subsubsection{Hidden extensions}

The key to all higher differentials is a precise interpretation of the following nonconnective analogue of the extension $\tau^2 \eta^3 = 4\nu$ from \cref{lem:4nu_tau2_eta3}.

\begin{lemma}\label{lem:noncon_2ext}
    For $0\leq a, b$, there are are isomorphisms of abelian groups
    \[
        \oppi_{-28-4a-24b,\ 2-4a} \Smf/\tau^4 \cong \Z/8\angbr{[\eta c_4^{-2}c_6\Delta^{-1-a}\kappabar^{a}]},
    \]
    where $[\eta c_4^{-2}c_6\Delta^{-1-a-b}\kappabar^{a}]$ is the unique lift to $\Smf/\tau^4$ of $[h_1 c_4^{-2}c_6\Delta^{-1-a-b}g^{a}]$.
\end{lemma}

The $b$-direction is $\Delta^{-b}$-multiplication and the $a$-direction follows $\angbr{g\Delta^{-1}}$-multiplication along the line of slope $-1$.

\begin{proof}
    Let us focus on the case where $a=b=0$ for simplicity; the other cases also follow by $(\kappabar\Delta^{-1})^a$-multiplication.
    In this case, an application of \cref{cor:easy_lifting_higher_power_tau} tells us $\Smf/\tau^4 \to \Smf/\tau$ is injective in degree $(-28,4)$.
    Moreover, this reduction map is surjective, as the generator $\angbr{2g\Delta^{-2}}$ is a $d_{\leq 4}$-cycle.
    We write $\angbr{2\kappabar\Delta^{-2}}$ for the unique generator of $\oppi_{-28,\,4}\Smf_{(2)}/\tau^4$ whose mod $\tau$ reduction is $\angbr{2g\Delta^{-2}}$.
    Since $h_1^3[c_4^{-2}c_6\Delta^{-1}] = 2h_2\angbr{2g\Delta^{-2}}$ holds in $\Smf/\tau$, see \cite[page~32]{konter_Tmf}, we learn that
    \[
        \eta^3 [\eta c_4^{-2}c_6\Delta^{-1}] = 2\nu\angbr{2\kappabar\Delta^{-2}}.
    \]
    Combining this with the relation $\tau^2 \eta^3 = 4\nu$ of \cref{lem:4nu_tau2_eta3}, we learn that in $\oppi_{-25,3}\Smf/\tau^4$, we have
    \[4\nu [\eta c_4^{-2}c_6\Delta^{-1}] = \tau^2 \eta^3 [\eta c_4^{-2}c_6\Delta^{-1}] = \tau^2 2\nu\angbr{2\kappabar\Delta^{-2}}.\]
    This implies $2 [\eta c_4^{-2}c_6\Delta^{-1}] = \tau^2\angbr{2\kappabar\Delta^{-2}}$ in $\oppi_{-28,2}\Smf/\tau^4$, as desired.
\end{proof}

\subsubsection{Coatomic strip}

All of the higher differentials in negative stems now boil down to the following region.

\begin{definition}
    \label{def:coatomic_strip}
    The \defi{coatomic strip} is the region of the $\uE_5$-page of the DSS for $\Tmf_{(2)}$ given by $\uE_5^{n,s}$ where
    \[
        -32 \leq n+s\leq -12 \qquad \text{and} \qquad 4s-n\leq 192.
    \]
    In terms of a chart, this is the shape between the lines emanating from $(-32,0)$ and $(-12,0)$ of slope $-1$ and under the connective region shifted by $\Delta^{-8}$.
    See \cite[Figure~27]{konter_Tmf} for a chart of the negative stems for $-221 < n < 0$.
\end{definition}

The coatomic strip is defined precisely so that all differentials in the DSS for $\Tmf$ in stems $n\leq -21$ follow from differentials in this region by iterated $\Delta^8$- or $\kappabar$-multiplication; see \cref{prop:coatomickeyfact}.

The key computations in the nonconnective region occurs within the coatomic strip.

\begin{proposition}\label{noncon:coatomicdiffs}
The differentials of \cref{tab:coatomicdiffs} determine all differentials in the coatomic strip by $h_1$- or $h_2$-multiplication.
\end{proposition}

\begin{table}[ht]
    \centering
    \caption{Differentials in the coatomic strip. The column labelled \emph{Change} indicates the length of the original differential the current one is obtained from. The notation $\nu_2(n)$ denotes the $2$-adic valuation of $n$.}
    \label{tab:coatomicdiffs}
    \vspace{1em}
    \adjustbox{scale=0.9}{
\begin{tabular}{@{ }lccccc@{ }}
    \toprule
    Diff. & Source & Bidegree & Target & Range & Change \\
    \midrule
    $d_5$ & $[\tfrac{1}{4}c_4^{-1}c_6\Delta^{-1}]$ & $(-21,1)$ & $h_2 \angbr{h_2g\Delta^{-2}}$ & & \\
    $d_5$ & $\angbr{2g^i\Delta^{-1-i}}$ & $(-24-4i,4i)$ & $4\angbr{h_2g^{1+i}\Delta^{-2-i}}$ & $i=1,5$ & \\
    $d_5$ & $\angbr{2g^j\Delta^{-1-j}}$ & $(-24-4j,4j)$ & $\pm 2\angbr{h_2g^{1+j}\Delta^{-2-j}}$ & $j=2,4,6$ & \\
    $d_5$ & $\angbr{h_2g^k\Delta^{-1-k}}$ & $(-21-4k,1+4k)$ & $h_2\angbr{h_2g^{1+k}\Delta^{-2-k}}$ & $k=2,4,6$ & \\
    $d_7$ & $h_1[c_4^{-2}c_6 g^i \Delta^{-1-i}]$ & $(-4i-28,2+4i)$ & $2^{\nu_2(i+2)}\angbr{h_2g^{2+i}\Delta^{-3-i}}$ & $0\leq i\leq 5$ & \\
    $d_9$ & $\angbr{cg^j\Delta^{-1-j}}$ & $(-16-4j,2+4j)$ & $\angbr{h_1dg^{j+2}\Delta^{-3-j}}$ & $j=0,1,4,5$ & \\
    $d_9$ & $\angbr{dg^k\Delta^{-2-k}}$ & $(-34-4k,2+4k)$ & $h_1^2[c_4^{-2}c_6g^{2+k}\Delta^{-3-k}]$ & $k=0,4$ & $d_{11}$\\
    $d_9$ & $\angbr{h_1dg^3\Delta^{-5}}$ & $(-45,15)$ & $h_1^3[c_4^{-2}c_6g^{5}\Delta^{-6}]$ & & $d_{11}$\\
    $d_{11}$ & $h_1^2[c_4^{-2}c_6g^i\Delta^{-1-i}]$ & $(-27-4i,3+4i)$ & $\angbr{cg^{3+i}\Delta^{-4-i}}$ & $i=0,3,4$ & $d_{9}$\\
    $d_{13}$ & $h_2^2[\tfrac{1}{4}c_4^{-1}c_6g^j\Delta^{-1-j}]$ & $(-15-4j,3+4j)$ & $2\angbr{g^{4+j}\Delta^{-4-j}}$ & $j=0,4$ & \\
    $d_{13}$ & $[\tfrac{1}{2}c_4^{-1}c_6\Delta^{-1}]$ & $(-21, 1)$ & $\angbr{dg^3\Delta^{-4}}$ & & \\
    $d_{13}$ & $2\angbr{h_2g^4\Delta^{-5}}$ & $(-37,17)$ & $\angbr{dg^7\Delta^{-8}}$ & & \\
    $d_{25}$ & $[c_4^{-1}c_6\Delta^{-1}]$ & $(-21,1)$ & $h_1^2 \angbr{g^6\Delta^{-6}}$ & & $d_{23}$\\
    $d_{25}$ & $h_1^{3-k}[c_4^{-2}c_6g^k\Delta^{-1-k}]$ & $(-26-5k,4+3k)$ & $h_1^{1-k} \angbr{g^{7+k}\Delta^{-7-k}}$ & $k=0,1$ & $d_{23}$\\
    \bottomrule
\end{tabular}}
\end{table}

\begin{proof}
    The key to obtain all of the above differentials is to carefully analyse the commutative diagram
    \begin{equation}\label{dia:nonconntotaldiff}\begin{tikzcd}
            {\oppi_{n,s}\Smf/\tau^4}\ar[d, "{\delta_4^{14}}"']\ar[r, "{\kappabar^t}"]   &   {\oppi_{n+20t,\, s+4t}\Smf/\tau^4}\ar[d, "{\delta_4^{14}}"]    \\
            {\oppi_{n-1,\, s+5}\Smf/\tau^{10}}\ar[r, "{\kappabar^t}"]   &   {\oppi_{n-1+20t,\, s+5+4t}\Smf/\tau^{10}.}
        \end{tikzcd}
    \end{equation}
Combining this diagram together with \cref{prop:total_differential_versus_differentials} and \cref{lem:noncon_2ext} yields all of the differentials in the coatomic strip.

In more detail, for the $d_5$'s, $d_7$'s, $d_{13}$'s, and also the $d_{9}$'s without any change, if $(n,s)$ indicates the bidegree of the source of the differential in question, then the horizontal maps of (\ref{dia:nonconntotaldiff}) are isomorphisms for all $t\geq 0$ by \cref{lem:noncon_2ext}. We then lift the elements in question from $\Smf/\tau$ to $\Smf/\tau^4$, and use the fact that we know the right-hand total differential in (\ref{dia:nonconntotaldiff}) from the prior computations in the positive stems by letting $t=1$ or~$2$. One then immediately obtains the desired differentials using \cref{prop:total_differential_versus_differentials}.

All other differentials also follow from this argument, but various multiplications by $\tau^2$ lead to some $d_r$-differentials being deduced from $d_{r\pm 2}$-differentials; this is indicated in \cref{tab:coatomicdiffs} in the column labelled \emph{Change}. Let us give the two prototypical examples of this phenomenon; the rest follow from the same arguments.

Consider $x = \angbr{d\Delta^{-2}}$ in bidegree $(n,s)=(-34,2)$, and consider \eqref{dia:nonconntotaldiff} with $t=1$.
First, we lift $x$ and $\angbr{dg\Delta^{-2}}$ to classes $\angbr{\kappa\Delta^{-2}}$ and $\angbr{\kappa\kappabar\Delta^{-2}}$ in $\Smf/\tau^4$. We already have the total differential expression $\delta_4^{14}(\angbr{\kappa\kappabar\Delta^{-2}}) = \tau^6\eta \angbr{\kappabar^4 \Delta^{-4}}$ on the right from the $d_{11}$'s of \cref{ssec:page11}. As $\kappabar \cdot \eta^2[\kappabar^2 c_4^{-2}c_6 \Delta^{-3}] = \tau^2 \eta \angbr{\kappabar^4 \Delta^{-4}}$ from \cref{lem:noncon_2ext}, we obtain the total differential
\[
    \delta_4^{14}(\angbr{\kappa \Delta^{-2}}) = \tau^4 \eta^2 [\kappabar^2 c_4^{-2}c_6 \Delta^{-3}],
\]
which implies that $x$ supports a $d_{9}$.  

In the other direction, consider the class $y=h_1^2[c_4^{-2}c_6\Delta^{-1}]$ and its lift to $\eta^2[c_4^{-2}c_6\Delta^{-1}]$, which generates the group $\oppi_{-27,3} \Smf/\tau^4$.
We know from \cref{lem:noncon_2ext} that
\[
    \kappabar \cdot \eta^2[c_4^{-2}c_6\Delta^{-1}] = \tau^2 \angbr{\eta\kappabar^2\Delta^{-2}},
\]
and we also have the total differential $\delta_{4}^{14}(\angbr{\eta\kappabar^2\Delta^{-2}}) = \tau^4 \angbr{\eps\kappabar^4\Delta^{-4}}$.
We then compute
\[\kappabar\cdot \delta_4^{14}(\eta^2[c_4^{-2}c_6\Delta^{-1}]) = \tau^2\cdot \delta_4^{14}(\angbr{\eta\kappabar^2\Delta^{-2}}) = \tau^6 \angbr{\eps\kappabar^4\Delta^{-4}},\]
which implies that $y$ supports a $d_{11}$ hitting $\angbr{\eps\kappabar^3\Delta^{-4}}$, as desired.

For the $d_{25}$'s we cannot reuse exactly \eqref{dia:nonconntotaldiff}, but rather use the analogous diagram with $\delta_{4}^{26}$ replacing $\delta_4^{14}$.
It is possible to compute such a total differential as once we are up to the $d_{25}$'s the spectral sequence is incredibly sparse. 
\end{proof}

The differentials of \cref{noncon:coatomicdiffs} also appear in \cite[Figure 27]{konter_Tmf} although there is no proof. We do not know how to obtain these differentials in stems $\leq -21$ without using synthetic techniques.

\subsubsection{Propagation}

All other differentials in the nonconnective region follow.

\begin{proposition}\label{prop:coatomickeyfact}
    Let $r\geq 5$. All $d_r$-differentials in the DSS for $\Tmf$ with source of bidegree $(n,s)$ with $n\leq -21$ follow from those of \cref{noncon:coatomicdiffs} by the Leibniz rule.
\end{proposition}

The moral is that all differentials whose source has bidegree $(n,s)$ with $n+s+10\geq 0$ follow from the differentials in the connective region by $g^a\Delta^{8b}$-multiplication, and all of those with $n+s+10 \leq 0$ follow from the differentials in the coatomic region for the same reason.

\begin{proof}
    Suppose we want to compute $d_r(x)$ for a class $x$ on the $\uE_r$-page of the DSS for $\Tmf$ in stem $n\leq -21$. If $x$ has filtration $s\leq 27$, then for some $a,b\geq 0$, the element $xg^a\Delta^{8b}$ lies in the coatomic strip. If $x$ has filtration $2\leq s\leq 27$, for this choice of $a$ and $b$, the multiplication by $g^a\Delta^b$-map is an isomorphism on the $\uE_r$-page in degrees $(n,s)$ and $(n-1,\, s+r)$. Using that $g$ and $\Delta^{8}$ are permanent cycles, the latter by \cref{prop:vanishingE24}, this allows us to important the differentials straight from those in the coatomic strip of \cref{noncon:coatomicdiffs} by the Leibniz rule. If $x$ has filtration $s=1$, the same argument works, except the multiplication by $g^a\Delta^b$-map might only be an surjection. Regardless, the argument persists.

    For classes $x$ with filtration $s\geq 28$, we use the fact that there is a natural number $c$ such that $x = g^c y$ for some unique class $y$ of filtration $2\leq s'\leq 21$. One then obtains the value of $d_r(x)$ from $d_r(y)$ using the argument above and the Leibniz rule.
\end{proof}

This finishes our computation of the DSS for $\Tmf$, so we have proven \cref{thm:dssattwo}.


\subsection{Rainchecked four-fold Toda brackets in Smf}\label{sec:large_carrick_brackets}

Logically speaking, this section appears just after we have finished with the $\uE_{11}$-page computations above and just before our discussion of the $\uE_{13}$-page.
To establish our earlier $d_{13}$'s, we needed to establish some hidden extensions.
This used that the element $\nu_2$ (\cref{not:lift_nu1_nu2_nu4_nu5_nu6}) is contained in the bracket
\[\angbr{\nu,\,2\nu\tau^4,\,\nu\tau^4,\,\kappabar^2}\subseteq\oppi_{51,1}\Smf/\tau^{10},\]
which allowed us to determine the product $2\nu\cdot\nu_2$.
We now compute this four-fold bracket.
In the following discussion we work in the symmetric monoidal $\infty$-category of modules in $\Syn$ over $\Smf/\tau^{10}$.
In particular, we write $\1$ for $\Smf/\tau^{10}$.

A simple long exact sequence argument shows there is a diagram as follows
\[
\begin{tikzcd}
    \1^{7,-7}\arrow[r,dashed,"\exists!\overline{\nu\tau^4}"]\arrow[dr,"\nu\tau^4"']&C(2\nu\tau^4)\arrow[d]\\
    &\1^{4,-4},
\end{tikzcd}
\]
where the vertical map is projection onto the top cell, and similarly with the roles of $\nu$ and $2\nu$ replaced. Taking the cofibers with respect to the dashed maps gives the unique forms of the $3$-cell complexes $C(\nu\tau^4,2\nu\tau^4)$ and $C(2\nu\tau^4,\nu\tau^4)$, in the sense of \cref{def:formsofcones}.

\begin{proposition}
\label{prop:nu2_fourfold_carrick_bracket}
    \leavevmode
\begin{itemize}
    \item \textup{(51,1)} The set
\[
    \angbr{\nu,\, 2\nu\tau^4,\, \nu\tau^4,\, \kappabar^2} \subseteq \oppi_{51,1}\Smf/\tau^{10}
\]
is a singleton consisting of \textbr{the mod $\tau^{10}$ reduction of} the element $\nu_2$ from \cref{not:lift_nu1_nu2_nu4_nu5_nu6}.
    \item \textup{(147,1)} The set
\[
    \angbr{\nu_4,\, 2\nu\tau^4,\, \nu\tau^4,\, \kappabar^2} \subseteq \oppi_{51,1}\Smf/\tau^{10}
\]
is a singleton consisting of \textbr{the mod $\tau^{10}$ reduction of} the element $\nu_6$ from \cref{not:lift_nu1_nu2_nu4_nu5_nu6}.
\end{itemize}
\end{proposition}

In a diagram, this proposition says the composition
\[
\begin{tikzcd}
    \boxed{48,0}\arrow[r,"\kappabar^2"]&\boxed{8,-8}\arrow[d,"\nu\tau^4",dash]\\
&\boxed{4,-4}\arrow[d,"2\nu\tau^4",dash]\\
&\boxed{0,0}\arrow[r,"\underline{\nu}"]&\boxed{-3,-1}
\end{tikzcd}
\]
can be constructed uniquely and that it is sent to $\nu\Delta^2$ in $\Smf/\tau^4$, which uniquely specifies $\nu_2 \in \oppi_{51,1}\Smf/\tau^{10}$; see \cref{lem:checking_nu_i_lift}.

\begin{proof}
To construct $\underline{\nu}$, we must extend the map $\nu$ over $C(2\nu\tau^4,\nu\tau^4)$, which by \cref{prop:conesandbracketsexistence} exists if
\[0\in\angbr{\nu,2\nu\tau^4,\nu\tau^4}.\]
This bracket is nonempty since $2\nu\cdot\nu=0$ in $\oppi_{*,*}\1$, and a degree check shows it has zero indeterminacy. Shuffling shows it contains $\tau^8\angbr{\nu,2\nu,\nu}$, hence it suffices to show $0\in\angbr{\nu,2\nu,\nu}$. This also has zero indeterminacy, and for degree reasons is either $0$ or $\eta^2c_4$. However, in the latter case, the shuffling formula
\[\angbr{\nu,2\nu,\nu}\eta=\nu\angbr{2\nu,\nu,\eta}=\nu\epsilon=0\]
would give a contradiction, as $\eta^3c_4\neq0$. Two extensions $\underline{\nu}$ so constructed differ by an element in $\oppi_{11,-7}\1=0$, so this constructs $\underline{\nu}$ uniquely.

We construct $\overline{\kappabar^2}$ analogously by constructing its dual. As before, we must show that
\[0\in \angbr{2\nu\tau^2,\nu\tau^4,\kappabar^2}.\]
This bracket is nonempty since $2\nu\cdot\nu=0$ and $\nu\tau^4\kappabar^2=0$ by \cref{cor:todabracketforkappabar}, and it lives in $\oppi_{47,1}\1$, which is zero. A degree argument shows there is a unique extension of $\kappabar^2$ over $C(\nu\tau^4)$, and two further extensions over $C(\nu\tau^4,2\nu\tau^4)$ differ by an element in $\oppi_{48,0}\1\neq 0$, via the exact sequence
\[
    \begin{tikzcd}
        \oppi_{48,0} \1 \ar[r,"\iota"] & \oppi_{48,0} C(\nu\tau^4,2\nu\tau^4) \ar[r] & \oppi_{48,0}\Sigma^{4,-4}C(2\nu\tau^4)
    \end{tikzcd}
\]
so $\overline{\kappabar^2}$ is not unique. However, the composite $\underline{\nu}\circ\iota=\nu$, and $\nu$ kills $\oppi_{48,0}\1$, hence $\underline{\nu}\circ\overline{\kappabar^2}$ is unique, which proves the set
\[\angbr{\nu,2\nu\tau^4,\nu\tau^4,\kappabar^2}\]
is a singleton.

To determine the image of $\angbr{\nu,2\nu\tau^4,\nu\tau^4,\kappabar^2}$ under the map $\Smf/\tau^{10}\to\Smf/\tau^4$, we note that the primary attaching maps $2\nu\tau^4$ and $\nu\tau^4$ in $C(2\nu\tau^4,\nu\tau^4)$ are sent to zero in $\Smf/\tau^4$.
Therefore the attaching map for the middle cell in $C(2\nu\tau^4,\nu\tau^4)$ is sent to zero, and the attaching map for the top cell in $C(2\nu\tau^4,\nu\tau^4)$ factors through a map in $\oppi_{7,-7}\Smf/\tau^4=0$.
Applying $-\otimes_{\1}\Smf/\tau^4$ to the composite  $\underline{\nu}\circ\overline{\kappabar^2}$ therefore gives a diagram
\[
\begin{tikzcd}
    \boxed{48,0}\arrow[r,"\kappabar^2"]\arrow[dr,"f_2",dashed]\arrow[ddr,"f_1"',dashed]&\boxed{8,-8}\\
&\boxed{4,-4}\\
&\boxed{0,0}\arrow[r,"\nu"]&\boxed{-3,-1}
\end{tikzcd}
\]
in modules over $\Smf/\tau^4$, using that the components of $\underline{\nu}$ on the middle and top cell are zero for degree reasons. The composite is therefore given by $\nu \cdot f_1$, and it remains to identify $f_1$ with $\Delta^2\in\oppi_{48,0}\Smf/\tau^4$ up to terms in the kernel of multiplication by $\nu$.
 
We first construct a commutative square
\[
\begin{tikzcd}
    \opSigma^{4,-4}C(\nu\tau^4)\arrow[d,"\underline{2\nu}"]\arrow[r]&\1^{1,-1}\arrow[d,equal]\\
    \1^{1,-5}\arrow[r,"\tau^4"]&\1^{1,-1}.
\end{tikzcd}
\]
Here we choose the left-hand vertical map to be the unique extension of $2\nu\colon \1^{4,-4}\to\1^{1,-5}$ over $C(\nu\tau^4)$. Taking fibers produces a commutative square
\begin{equation}\label{eq:MossSquareFourFold}
\begin{tikzcd}
    C(2\nu\tau^4,\nu\tau^4)\arrow[r]\arrow[d]&\opSigma^{4,-4}C(\nu\tau^4)\arrow[d]\\
    C(\tau^4)\arrow[r]&\1^{1,-5}
\end{tikzcd}
\end{equation}
where each of the horizontal maps crush the bottom cell. Applying the counterclockwise composite in this diagram to $\overline{\kappabar^2}$ gives a composition that we may represent with the following diagram.
\[
\begin{tikzcd}
    \boxed{48,0}\arrow[r,"\overline{\kappabar^2}"]&\boxed{8,-8}\arrow[d,"\nu\tau^4",dash]\\
&\boxed{4,-4}\arrow[d,"2\nu\tau^4",dash]\arrow[r,"2\nu"]&\boxed{1,-5}\arrow[d,"\tau^4",dash]\arrow[r,equal]&\boxed{1,-5}\\
&\boxed{0,0}\arrow[r,equal]&\boxed{0,0}
\end{tikzcd}
\]
After tensoring down to $\Smf/\tau^4$, the composite of the first two morphisms above has bottom component equal to $f_1$ and the diagram shows that $\delta_4^8(f_1)=2\nu \cdot f_2$. 

To determine $2\nu \cdot f_2$, we apply the clockwise composite in \eqref{eq:MossSquareFourFold} to $\overline{\kappabar^2}$, which we may represent with the following diagram.
\[
\begin{tikzcd}
    \boxed{48,0}\arrow[r,"\kappabar^2"]&\boxed{8,-8}\arrow[d,"\nu\tau^4",dash]\arrow[r,equal]&\boxed{8,-8}\arrow[d,"\nu\tau^4",dash]\\
&\boxed{4,-4}\arrow[d,"2\nu\tau^4",dash]\arrow[r,equal]&\boxed{4,-4}\arrow[r,"2\nu"]&\boxed{1,-5}\\
&\boxed{0,0}
\end{tikzcd}
\]
After tensoring down to $\Smf/\tau^4$, the composite $\phi$ of the first two morphisms above has bottom component equal to $f_2$, so we analyze $\phi$ carefully. One may construct a diagram
\[
\begin{tikzcd}
    \boxed{48,0}\arrow[r,"\kappabar^2"]&\boxed{8,-8}\arrow[d,"\nu\tau^4",dash]\arrow[r,equal]&\boxed{8,-8}\arrow[d,"\nu\tau^4",dash]\arrow[r,"\nu"]&\boxed{5,-9}\arrow[d,dash,"\tau^4"]\\
&\boxed{4,-4}\arrow[d,"2\nu\tau^4",dash]\arrow[r,equal]&\boxed{4,-4}\arrow[r,equal]&\boxed{4,-4}\\
&\boxed{0,0}
\end{tikzcd}
\]
where the composite of the first two maps is $\phi$. After tensoring down to $\Smf/\tau^4$, the component of this composite onto the cell of dimension $(4,-4)$ remains $f_2$, so this diagram shows that $\delta_4^8(f_2)=\kappabar^2\nu$, which guarantees that $f_2=\kappabar\Delta$.

Putting these facts together we conclude that  $\delta_4^8(f_1)=2\nu\kappabar\Delta$, which guarantees that $f_1=\Delta^2\in\oppi_{*,*}\Smf/\tau^4$ modulo $\nu$-torsion, and we conclude that $\nu \cdot f_1=\nu_2$.

The bracket for $\nu_6$ follows from analogous arguments.
\end{proof}

\section{Computations away from the prime 2}
\label{sec:computations_away_from_2}

Above we computed (most of) the signature spectral sequence of $\Smf_{(2)}$, and here we would like to do the same for $\Smf_{(3)}$ and $\Smf[\tfrac{1}{6}]$.
All together, these three results yield the signature spectral sequence for $\Smf$.

For any collection of primes $J$, the signature spectral sequence associated to $\Smf[J^{-1}]$ is naturally identified with the DSS for $\Tmf[J^{-1}]$. Indeed, by \cite[Proposition 4.16]{CDvN_part1}, the natural map of synthetic $\E_\infty$-rings
\[\Smf[J^{-1}] = \calO^\syn(\Mellbar)[J^{-1}] \simeqto \calO^\syn(\Mellbar\times \Spec \Z[J^{-1}])\]
is an equivalence. This is implicitly used below.

\subsection{Computations at the prime 3}

As is often the case for $\Tmf$, the $3$-local DSS computation is a vast simplification of the $2$-local analogue discussed above. 
We implicitly work $3$-locally in this subsection.

\begin{theorem}\label{dssatthree}
    The signature spectral sequence of $\Smf$, i.e., the DSS for $\Tmf$, is determined below, and has precisely the form depicted in \cref{ssprime3}.
\end{theorem}

The $\uE_2$-page is computed using sheaf cohomology; see \cref{ssprime3} or \cite[Figure 10]{konter_Tmf}. We recommend the reader keeps these charts nearby throughout the following arguments.

There are only two atomic differentials.

\begin{proposition}
        \leavevmode
    \begin{itemize}
        \item \textup{(24,0)}   $d_5(\Delta) = \pm \al \be^2$.
        \item \textup{(51,1)}   $d_9(\al \Delta^2) = \pm \be^5$.
    \end{itemize}
\end{proposition}

\begin{proof}
    Recall from \cref{heightonedetection} that $\al\in \oppi_{3,1}\Sph$ has nonzero image in $\oppi_{3,1}\Smf$.
    The class in the sphere $\be \in \oppi_{10,2} \Sph$ is defined as the Toda bracket $\angbr{\al,\al,\al}$. From the Massey product structure of the $3$-local cubic Hopf algebroid, combined with \cref{prop:todagoestomassey} and \cref{etwopagebracketsprimethree}, we see that $\be$ hits the generator of $\oppi_{10,2}\Smf/\tau \cong \F_3$, which we also call $\beta$; see \cite[Equation~(5.1)]{bauer_tmf}.
    
    In the ANSS for $\Sph$, there is the classical Toda differential $d_5(\beta_{3/3}) = \pm\alpha \beta^3$; see \cite[Theorem~4.4.22]{ravenel_green_book}.
    As a result, the element $\alpha\beta^3$ in $\Smf/\tau$ must also be hit by a differential.
    For degree reasons, the only possibility is $d_5(\pm \beta \Delta) = \al\be^3$.
    From the Leibniz rule and the fact that $\beta$ is a permanent cycle, we obtain $d_5(\Delta) =\pm \al\beta^2$. The Leibniz rule gives all other $d_5$'s in this spectral sequence. 
    
    For degree reasons, the next possible differential is a $d_9$. To compute this atomic $d_9$, we show that the relation
    \begin{equation}
        \label{eq:tau4_alpha_beta2_zero}
        \tau^4 \al\beta^2 = 0
    \end{equation}
    holds in $\Smf/\tau^{14}$.
    As the mod $\tau$ reduction $\alpha\beta^2 \in \oppi_{23,5}\Smf/\tau$ is the target of a $d_5$, , \cref{thm:omnibus} tells us that there exists a $\tau^4$-torsion lift to $\oppi_{23,5}\Smf/\tau^{14}$.
    Using \cref{cor:easy_lifting_higher_power_tau}, we see that this lift is unique, proving that \eqref{eq:tau4_alpha_beta2_zero} holds.
    
    Next, recall that the defining Toda bracket expression $\be = \angbr{\al,\al,\al}$ also holds in $\Smf/\tau^{14}$.
    Applying the usual juggling formula \cref{prop:shufflingformulas} and the relation \eqref{eq:tau4_alpha_beta2_zero}, we find that
    \[
        \tau^8\be^5 = \tau^4\be^2 \angbr{\al,\al,\al} \tau^4\be^2 = \angbr{\tau^4\be^2,\al,\al} \tau^4 \al\be^2 = 0.
    \]
    (Note that \eqref{eq:tau4_alpha_beta2_zero} also justifies that the second bracket is nonempty.)
    In other words, we have learned that $\beta^5$ is $\tau^8$-torsion, which by \cref{thm:omnibus} means that its mod $\tau$ reduction $\beta^5 \in \oppi_{50,10}\Smf/\tau$ is hit by a $d_{\leq 9}$-differential.
    As this class is a $d_{\leq 8}$-cycle for degree reasons, it must be hit by $d_9$.
    The only possibility is the desired $d_9(\al\Delta^2) = \pm \be^5$. All other $d_9$'s follow from the Leibniz rule.\footnote{As an alternative to the computation of $d_9$ above, one can apply the synthetic Leibniz rule to $\delta_4^8(\Delta) = \pm \alpha \be^2$ and use the extension $\al[\al\Delta] = \tau^4 \beta^3$ to compute
    \[\delta_4^8(\al\Delta^2) = \pm \alpha [\al\Delta] \beta^2 = \pm \tau^4 \beta^5.\]
    This directly shows $d_9(\al\Delta^2) = \pm \beta^5$.}
\end{proof}

This spectral sequence then collapses with a horizontal vanishing line at $s=8$. This yields the homotopy groups of $\Tmf$, which can be read off from \cref{ssprime3}.
In other words, we have proved \cref{dssatthree}.

\subsection{Computations away from 6}

\begin{theorem}\label{rationalcomputation}
    There is an isomorphism of bigraded $\Z[\tfrac{1}{6}, \tau]$-modules
    \[
        \oppi_{\ast,\ast}\Smf[\tfrac{1}{6}] \cong \Z[\tfrac{1}{6}, \tau][c_4,\Delta]\otimes E(c_6) \oplus \Z[\tfrac{1}{6}, \tau]\set*{ c_4^i c_6^j \Delta^k }_{\substack{i,k\leq -1 \\ 0\leq j\leq 1}}
    \]
    where $E(\blank)$ denotes exterior algebra and
    \[
        \abs{c_4^i c_6^j \Delta^k} = \begin{cases}
            \ (8i+12j+24k,\ 0) & \text{if }i \geq 0,\\
            \ (8i+12j+24k-1,\ 1) & \text{if }i\leq -1.
        \end{cases}
    \]
    Moreover, in nonnegative degrees, this is an isomorphism of rings.
\end{theorem}

This $\Z[\tau]$-module structure on the synthetic homotopy groups shows that the signature spectral sequence of $\Smf[\tfrac{1}{6}]$ converging to $\Tmf[\tfrac{1}{6}]$ collapses on the $\uE_2$-page.

\begin{proof}
    By \cite[Construction~3.1]{CDvN_part1}, there is a synthetic DSS for this sheaf of synthetic $\E_\infty$-rings that takes the form
    \[\uE_2^{k,\ast,s} = \uH^s(\Mellbar[\tfrac{1}{6}],\, \omega^{k+s/2})[\tau] \implies \oppi_{k, \ast} \Smf[\tfrac{1}{6}].\]
    As $\Mellbar[\tfrac{1}{6}]$ is the weighted projective stack $\mathbf{P}(4,6)$, we see that in positive $s$-degree this spectral sequence is concentrated in filtration $0$ and in negative $s$-degree in filtration $1$; see \cite[Section~6]{konter_Tmf}. This spectral sequence then collapses and we obtain the desired result.
\end{proof}


\section{Proofs of main theorems}
\label{sec:proofs_main_theorems}

With all of the computations out of the way, we can now prove our main theorems and corollaries, including the Gap Theorem (\cref{thm:gap}) and a description of the DSS for $\Tmf$ (\cref{thm:dss}). 
This removes the circularity discussed in \cref{ssec:history} from the literature.
As corollaries, we also compute the homotopy groups of $\Tmf$ (\cref{thm:homotopyoftmf}), the ANSS for $\TMF$ (\cref{thm:sssTMF}), and the ANSS for $\tmf$ (\cref{thm:ansstmf}).
For the convenience of the reader, we repeat the statements of these results below.

As emphasised in the introduction, the computation of the DSS for $\Tmf$ should be seen as the fundamental computation towards the homotopy groups of $\Tmf$ and $\tmf$ --- there is no other path through the literature that is not circular.


We start with the Gap Theorem.

\thmgap*

\begin{proof}
    Using the usual fracture square for $\Tmf$, it suffices to prove the Gap Theorem for $\Tmf_{(p)}$ for each prime $p$. By \cref{rationalcomputation}, this holds for all primes $p\neq 2,3$. By \cref{dssatthree}, also see \cref{ssprime3}, this also holds at the prime $3$. At the prime $2$, we do not need the whole DSS for $\Tmf_{(2)}$; the Gap Theorem follows from \cref{prop:vanishingE24}, which only relies on computations in stems $n > -21$.
\end{proof}

Our proof of the Gap Theorem only requires the DSS for $\Tmf_{(2)}$ only in stems $n\geq -20$.
The additional computations of \cref{ssec:nonconnectiveregion} yield the entire DSS for $\Tmf$.

\thmdss*

\begin{proof}
By \cite[Proposition~4.16]{CDvN_part1}, the natural map of synthetic $\E_\infty$-rings
\[\Smf[J^{-1}] = \calO^\syn(\Mellbar)[J^{-1}] \simeqto \calO^\syn(\Mellbar\times \Spec \Z[J^{-1}])\]
is an equivalence for any set of primes $J$. In particular, we obtain the DSS for $\Tmf$ from the DSS for $\Tmf[\tfrac{1}{6}]$ of \cref{rationalcomputation}, which collapses, the DSS for $\Tmf_{(3)}$ of \cref{dssatthree}, and the DSS for $\Tmf_{(2)}$ of \cref{thm:dssattwo}.
\end{proof}

We deduce the following, which will be helpful for later results.

\begin{lemma}
    \label{lem:Delta24_lifts}
    \leavevmode
    \begin{itemize}
        \item The element $\Delta \in \oppi_{24,0}\Smf[\frac{1}{6}]/\tau$ lifts to $\oppi_{24,0}\Smf[\frac{1}{6}]$.
        \item The element $\Delta^3 \in \oppi_{72,0}\Smf_{(3)}/\tau$ lifts to $\oppi_{72,0}\Smf_{(3)}$.
        \item The element $\Delta^8 \in \oppi_{192,0}\Smf_{(2)}/\tau$ lifts to $\oppi_{192,0}\Smf_{(2)}$.
        \item The element $\Delta^{24} \in \oppi_{576,0}\Smf/\tau$ lifts to $\oppi_{576,0}\Smf$.
    \end{itemize}
\end{lemma}
\begin{proof}
Using \cref{prop:non_truncated_omnibus}, it suffices to show that the indicated powers of $\Delta$ are permanent cycles in the localised DSS for $\Tmf$.\footnote{Our computations of the $\uE_\infty$-pages for the $6$-invertible, $3$-primary and $2$-primary cases in particular shows that the conditional convergence of the DSS for $\Tmf$ is in fact strong, so we can indeed apply this proposition.}
This now follows from the computations of \cref{sec:prime2,sec:computations_away_from_2}.
Indeed, when $6$ is inverted, the class $\Delta$ is a permanent cycle; $3$-locally, the power $\Delta^3$ is a permanent cycle; $2$-locally, the power $\Delta^8$ is a permanent cycle.
Their lowest common multiple is $\Delta^{24}$, proving the final claim.
\end{proof}

With the whole of the DSS for $\Tmf$ at hand, we have almost computed the homotopy groups of $\Tmf$; all that is left is to compute some extension problems. Most of these extension problems follow from rudimentary algebraic arguments given our computations so far.
The lone exception is a $2$-extension in degree $110$, which instead follows from our knowledge of total differentials.
This method is used in, e.g., \cite[Proposition~A.20]{burklund_hahn_senger_manifolds}, and is further explained in \cite[Method~2.17]{isaksen_etal_motivic_ANSS_tmf}.
The application of this method to this hidden extension is due to \cite[Proposition 4.5]{isaksen_etal_motivic_ANSS_tmf} (using a slight modification to the method stated in Method~2.17 of op.\ cit.).

\htpytmf*

\begin{proof}
Away from $6$, these homotopy groups are obtained from \cref{rationalcomputation} by inverting~$\tau$.
Localised at the prime $3$, they follow immediately from \cref{dssatthree} as there are no extension problems.
At the prime $2$, we use the DSS for \cref{thm:dssattwo}, but there are some extension problems to solve.

First, let us deal with the positive stems.
The $2$-extensions which follow from \cref{lem:4nu_tau2_eta3} in $\Smf_{(2)}/\tau^4$ are clear, so we ignore these.
There can be no extensions between the $\ko$-patterns, indicated by solid diamonds on the $0$-line of \cref{efiveprime2_part1,efiveprime2_part2,efiveprime2_part3,efiveprime2_part4}; this can be checked on a case by case basis. 
Indeed, there cannot be any $2$-extensions with source in filtration $0$ for algebraic reasons, as these sources are torsion-free. All other parts of the $\ko$-patterns are divisible by $\eta$, which cannot support multiplication by $2$.
Similar arguments discount many other potential extensions: there cannot be a $2$-extension in stem $65$ from filtration $3$ to $9$ as the source is divisible by the $2$-torsion class $\kappa$.
Using the lift of $\Delta^8$ from \cref{lem:Delta24_lifts}, the above arguments reduce us to verifying $2$-extensions in the following stems:
\[54,\quad 110,\quad 130,\quad 150.\]
The extension in the $110$-stem implies the one in the $130$-stem by $\kappabar$-multiplication.
In the $54$-stem, this extension is precisely \cref{prop:nu2_times_2_nu_hidden_extension}.
This also gives the lower half of the extensions in the $150$-stem, and the upper half follows from the ones in the $130$-stem by $\kappabar$-multiplication.
We are reduced to the extension in degree $110$.
In this case, consider the exact sequence

\adjustbox{scale=0.93,center}{
   \begin{tikzcd}
       {\oppi_{*,*}\Smf_{(2)}/\tau^{28}} \rar & {\oppi_{*,*} \Smf_{(2)}/\tau^4} \rar["\delta_4^{28}"] & {\oppi_{*-1,\,*+4} \Smf_{(2)}/\tau^{28}} \rar["\tau^4"] & {\oppi_{*-1,\, *} \Smf_{(2)}/\tau^{28}}.
   \end{tikzcd}}

We will show that there is an equality
\begin{equation}\label{twoextensioninluine}
    2 \,\tau^8 \,\kappabar^3\, \kappa_4 = \tau^{20}\,  \eta_1^2 \, \kappabar^6 \qquad \text{in } \oppi_{170, 6} \Smf_{(2)}/\tau^{28},
\end{equation}
where all of the elements displayed are their unique lifts to $\Smf/\tau^{28}$; from this extension in the $170$-stem, our desired extension in $\Tmf_{(2)}$ in the $110$-stem follows by $\kappabar^3$-divisibility and from the fact that the DSS for $\Tmf_{(2)}$ collapses on the $\uE_{26}$-page.
The differentials
\[
    d_{13}(2 h_2\Delta^7) = dg^3 \Delta^4 \qquad \text{and}\qquad d_{23}(h_1^3\Delta^7) = h_1^2g^6\Delta^2
\]
of \cref{ssec:page13,ssec:page23}, respectively, show by \cref{thm:omnibus}\,(3) that $\tau^8 \,\kappa_4 \,\kappabar$ and $\tau^{20} \,\eta_1^2\,\kappabar^6$ are both $\tau^4$-torsion (as they are the unique lift of the respective $\uE_2$-elements).
One then uses the exact sequence above to compute
\[
    \delta_4^{28}(2 \nu \Delta^7) = \tau^8 \,\kappa_4 \,\kappabar \qquad \text{and} \qquad \delta_4^{28}(\tau^2 \eta^3 \Delta^7) = \tau^{20} \, \eta_1^2\, \kappabar^6.
\]
In $\Smf/\tau^4$ we have $\tau^2 \eta^3 = 4\nu$ courtesy of the mod $\tau^4$ reduction of \cref{lem:4nu_tau2_eta3}, which also yields $4 \nu\Delta^7 = \tau^2 \eta^3 \Delta^7$.
Combining what we have so far yields the desired equality \eqref{twoextensioninluine}:
\[
    2 \tau^8 \,\kappabar^3\, \kappa_4 = \delta_4^{28}(4 \nu \Delta^7) = \delta_4^{28}(\tau^2 \eta^3 \Delta^7) = \tau^{20}\, \eta_1^2 \, \kappabar^5.
\]

We are reduced to extension problems in negative degrees. Again, there are no extension problems away from $2$, so we are reduced to $\Tmf_{(2)}$. In this case, all of the extension problems in negative degrees follow from their counterparts in positive degrees by $\Delta^{8t}$-multiplication for large enough $t$.
\end{proof}

\begin{remark}\label{rmk:moreexts}
    The methods used to deduce the $2$-extension in stem $110$ from the proof of \cref{thm:homotopyoftmf} above, adopted the argument from \cite[Method~2.17, Proposition~4.5]{isaksen_etal_motivic_ANSS_tmf}, generalise to capture many hidden extensions in the DSS for $\Tmf$. For example, consider the total differentials
    \[\delta_4^{28}(\eta\Delta^5) = \tau^{18}\kappabar^6 \qquad \text{and} \qquad \delta_4^{28}(\Delta^5) = u \nu_4\kappabar\]
    in $\Smf_{(2)}/\tau^{28}$, where $u$ is a unit of $\Z/8$; these can be computed from the computations of \cref{sec:prime2}.
    From this one obtains
    \[\eta \nu_4 \kappabar = \eta \cdot \delta_4^{28}(\Delta^5) = \delta_4^{28}(\eta \Delta^5) = \tau^{18} \kappabar^6,\]
    and hence the hidden $\eta$-extension
    \[\eta \nu_4 = \kappabar^5 \qquad \text{in } \oppi_{100} \Tmf_{(2)}\]
    of \cite[Corollary~8.7\,(2)]{bauer_tmf}.
    These arguments completely avoid the use of six-fold Toda brackets seen in \cite{bauer_tmf}.
\end{remark}

Using the Gap Theorem, Mathew computed the Hopf algebroid computing the $\uE_2$-page of the ANSS of $\tmf$, which Bauer then used to compute the $\uE_2$-page.
We will now show how this, combined with our computation of the DSS for $\Tmf$, computes the ANSS of $\tmf$.
This recovers Bauer's differentials of \cite{bauer_tmf} without any circularity issues.

\ansstmf*

\begin{proof}
    By \cite[Corollary~5.3]{mathew_homology_tmf}, which in turn relies on the Gap Theorem, the $\uE_2$-page of ANSS for $\tmf$ is isomorphic to the cohomology of the cubic Hopf algebroid of \cref{e2pagesection}.
    By \cref{etwopagebracketsprimethree}, this means that the composition of natural maps of synthetic $\E_\infty$-rings
    \[\opnu \tmf \to \opnu \Tmf \to \Smf\]
    induces an isomorphism on $\oppi_{n,s}(C\tau\otimes \blank)$ for $5s \leq n+12$, and in general is a retract of bigraded abelian groups. At $p=2$, the atomic $d_3$ in the signature spectral sequence for $\Smf$ lifts uniquely to the signature spectral sequence for $\opnu\tmf$. This propagates to all other $d_3$'s using the Leibniz rule. From the $\uE_5$-page on, there are no differentials whose source is in the connective region and whose target is outside the connective region (i.e., there are no line-crossing differentials \`{a} la \cref{sssec:linecrossing}). Therefore, we can safely import the differentials in the connective region of the signature spectral sequence of $\Smf_{(2)}$ to differentials in the signature spectral sequence of $\opnu\tmf_{(2)}$.
    
    The same is true at $p=3$, i.e., there are no line-crossing differentials. Away from $6$ this is tautological as $\tmf[\tfrac{1}{6}]$ is complex-oriented. 
    It follows that the map $\opnu \tmf\to \Smf$ induces a retract of spectral sequences.
\end{proof}

The DSS for $\TMF$ follows easily from \cref{thm:dss} by inverting an appropriate power of~$\Delta$.
In \cite[Corollary~4.17]{CDvN_part1}, we showed that if $\Delta^{24}\in \oppi_{576,0}\Smf/\tau$ lifts in the manner of \cref{lem:Delta24_lifts}, then the natural map of synthetic $\E_\infty$-rings 
\begin{equation}\label{eq:invertingdelta}
    \Smf[\Delta^{-24}] \simeqto \SMF
\end{equation}
is an equivalence.
Using this, the computation of the DSS for $\TMF$ now follows.

\sssTMF*

\begin{proof}
    By \cite[Theorem~C]{CDvN_part1}, the signature spectral sequence of $\SMF$ is the DSS for $\TMF$.
    As tensoring with $C\tau$ preserves colimits, the map \eqref{eq:invertingdelta} therefore identifies the $\uE_2$-page of the DSS for $\TMF$ with the $\Delta$-inversion of the $\uE_2$-page of the DSS for $\Tmf$.
    The $\Z[\Delta^{\pm}]$-module $\oppi_{\ast,\ast}\SMF/\tau$ is generated by the image of the connective region in $\oppi_{\ast,\ast}\Smf/\tau$.
    This means that by inverting $\Delta$ on the DSS for $\Tmf$, we obtain all of the differentials in the DSS for $\TMF$.
    The ANSS for $\TMF$ then follows from the equivalence of synthetic $\E_\infty$-rings $\opnu \TMF \simeq \SMF$ of \cite[Theorem~C]{CDvN_part1}, as $\opnu \TMF$ implements the ANSS for $\TMF$; see \cite[Proposition~1.25]{CDvN_part1}.
\end{proof}

\begin{remark}
    One can also directly compute the DSS for $\TMF$ using the techniques of this article; simply replace the meta-arguments with the fact that at the prime $2$ the class $\Delta^8$ acts (resp.\ at the prime $3$ the class $\Delta^3$ acts) by isomorphisms on each $E_r$-page, a fact which is inductively determined page-by-page. We leave the details up to an interested reader.
\end{remark}


\appendix
\section{Tables and charts}
\label{theappendix}


\subsection{Tables}
\label{ssec:tables}

\Cref{tab:2_sphere_import,tab:2_lifts,tab:2_tau_power_torsion,tab:2_hidden_ext,tab:2_total_diffs,tab:2_toda_brackets} collect the lifts, hidden extensions, Toda brackets, and total differentials proved in the 2-primary computation of \cref{sec:prime2}.
Every entry in the table is accompanied with the location where the element is defined or the relation is proved.

In these tables, the term \emph{validity} refers to the number $k$ for which the element or relation lives in $\Smf/\tau^k$, where validity $\infty$ means it lives in $\Smf$.
However, we only list the validity that we prove (and that we require); listing a finite validity does not necessarily mean that it does not lift further.
For total differentials, the validity is a pair of numbers $(n,N)$; this refers to the total differential $\delta_n^N$.
We omit the unknown units in the formulas for the total differentials, and refer to the location where the differential is proved for the expression with units included.
For the Toda brackets, if no indeterminacy is listed, this means it is zero.

\begin{table}[H]
    \centering
    \caption{Elements imported from the sphere to $\Smf_{(2)}$.}
    \label{tab:2_sphere_import}
    \vspace{1em}
\begin{tabular}{@{}lcccc@{}}
    \toprule
    Name & Degree & Detected by & Location & Comment\\
    \midrule
    $\eta$ & $(1,1)$ & $h_1$ & \ref{heightonedetection}, \ref{not:2_import_from_sphere} & \\
    $\nu$ & $(3,1)$ & $h_2$ & \ref{heightonedetection}, \ref{not:2_import_from_sphere} & \\
    $\eps$ & $(8,2)$ & $c$ & \ref{heightonedetection}, \ref{not:2_import_from_sphere} & \\
    $\kappa$ & $(14,2)$ & $d$ & \ref{prop:kappa_detection}, \ref{not:2_import_from_sphere} & \\
    $\kappabar$ & $(20,4)$ & $g$ & \ref{not:kappabar_and_g}, \ref{not:2_import_from_sphere} & Determined up to $\nu^2\kappa$-multiples\\
    \bottomrule
\end{tabular}
\end{table}

\begin{table}[H]
    \centering
    \caption{Lifts of elements from $\Smf_{(2)}/\tau$ to $\Smf_{(2)}/\tau^k$.}
    \label{tab:2_lifts}
    \vspace{1em}
\begin{tabular}{@{}lcccc@{}}
    \toprule
    Name & Degree & Validity & Lift of & Location \\
    \midrule
    $\Delta$ & $(24,0)$ & 4 & $\Delta$ in $\Smf/\tau$ & \ref{not:Delta_lift} \\
    
    $\eta_1$ & $(25,1)$ & 8 & $h_1\Delta$ in $\Smf/\tau$ & \ref{not:lift_eta1_eta_4_eta5} \\
    & & 10 & & \ref{not:lift_eta1_eta5_validity_10} \\
    $\eta_4$ & $(97,1)$ & 8 & $h_1\Delta^4$ in $\Smf/\tau$ & \ref{not:lift_eta1_eta_4_eta5} \\
    $\eta_5$ & $(121,1)$ & 8 & $h_1\Delta^5$ in $\Smf/\tau$ & \ref{not:lift_eta1_eta_4_eta5}\\
    & & 10 & & \ref{not:lift_eta1_eta5_validity_10} \\
    
    $\nu_1$ & $(27,1)$ & 14 & $2\nu\Delta$ in $\Smf/\tau^4$ & \ref{not:lift_nu1_nu2_nu4_nu5_nu6} \\
    $\nu_2$ & $(51,1)$ & 14 & $\nu\Delta^2$ in $\Smf/\tau^4$ & \ref{not:lift_nu1_nu2_nu4_nu5_nu6}\\
    $\nu_4$ & $(99,1)$ & 14 & $\nu\Delta^4$ in $\Smf/\tau^4$ & \ref{not:lift_nu1_nu2_nu4_nu5_nu6}\\
    $\nu_5$ & $(123,1)$ & 14 & $2\nu\Delta^5$ in $\Smf/\tau^4$ & \ref{not:lift_nu1_nu2_nu4_nu5_nu6}\\
    $\nu_6$ & $(147,1)$ & 14 & $\nu\Delta^6$ in $\Smf/\tau^4$ & \ref{not:lift_nu1_nu2_nu4_nu5_nu6}\\
    
    $\eps_1$ & $(32,2)$ & 12 & $c\Delta$ in $\Smf/\tau$ & \ref{not:lifts_eps1_eps_4_eps5}\\
    $\eps_4$ & $(104,2)$ & 12 & $c\Delta^4$ in $\Smf/\tau$ & \ref{not:lifts_eps1_eps_4_eps5}\\
    $\eps_5$ & $(128,2)$ & 12 & $c\Delta^5$ in $\Smf/\tau$ & \ref{not:lifts_eps1_eps_4_eps5}\\
    
    $\kappabar_4$ & $(116,4)$ & 20 & $2g\Delta^4$ in $\Smf/\tau$ & \ref{not:kappabar4_lift}\\
    \bottomrule
\end{tabular}
\end{table}

\begin{table}[H]
    \centering
    \caption{A few $\tau$-power torsion relations in $\oppi_{*,*}\Smf_{(2)}/\tau^k$.}
    \label{tab:2_tau_power_torsion}
    \vspace{1em}
\begin{tabular}{@{}lccc@{}}
    \toprule
    Relation & Degree & Validity & Location \\
    \midrule
    $\tau^4 \nu \kappabar=0$ & $(23,1)$ & 12 & \ref{lem:tau4_nu_kappabar_is_zero}\\
    & & 24 & \ref{lem:tau4_nu_kappabar_is_zeroweiter}\\
    $\tau^4 \nu_1\kappabar = 0$ & $(47,1)$ & 12 & \ref{cor:tau4_nu1_kappabar_is_zero}\\
    $\tau^4 \nu_4 \kappabar=0$ & $(119,1)$ & 12 & \ref{lem:tau4_nu_kappabar_is_zero}\\
    $\tau^4 \nu_5\kappabar = 0$ & $(143,1)$ & 12 & \ref{cor:tau4_nu1_kappabar_is_zero}\\
    \bottomrule
\end{tabular}
\end{table}


\begin{table}[H]
    \centering
    \caption{Relations and hidden relations in $\oppi_{*,*}\Smf_{(2)}/\tau^k$.}
    \label{tab:2_hidden_ext}
    \vspace{1em}
\begin{tabular}{@{}lccc@{}}
    \toprule
    Relation & Degree & Validity & Location \\
    \midrule
    $4\nu = \tau^2 \eta^3$ & $(3,1)$ & 14 & \ref{lem:4nu_tau2_eta3}\\
    $2\nu_1=\tau^2 \eta^2\eta_1$ & $(27,1)$ & 8 & \ref{lem:2_times_nu1}\\
    $2\nu_5=\tau^2 \eta^2\eta_5$ & $(123,1)$ & 8 & \ref{lem:2_times_nu1}\\

    $\nu_1 \eta = \tau^4 \eps \kappabar$ & $(28,2)$ & 12 & \ref{lem:hidden_extensions_nu1_nu2}\\
    $\nu_1 \eps = \tau^4 \eta \kappa \kappabar$ & $(35,3)$ & 12 & \ref{lem:hidden_extensions_nu1_nu2}\\

    $\nu_1 \kappa = \tau^6 \eta \kappabar^2$ & $(41,3)$ & 12 & \ref{lem:hidden_extensions_nu1_nu2} \\
    $\nu_2 \eta = \tau^4 \eps_1 \kappabar$ & $(52,2)$ & 12 & \ref{lem:hidden_extensions_nu1_nu2} \\
    $\nu_2 \eps = \tau^4 \eta_1\kappa \kappabar$ & $(59,3)$ & 10 & \ref{lem:hidden_extensions_nu1_nu2} \\
    $\nu_5 \eta = \tau^4 \eps_4 \kappabar$ & $(124,2)$ & 12 & \ref{lem:hidden_extensions_nu1_nu2} \\
    $\nu_5 \eps = \tau^4 \eta_4 \kappa \kappabar$ & $(131,3)$ & 12 & \ref{lem:hidden_extensions_nu1_nu2} \\
    $\nu_5 \kappa = \tau^6 \eta_4 \kappabar^2$ & $(137,3)$ & 12 & \ref{lem:hidden_extensions_nu1_nu2} \\
    $\nu_6 \eta = \tau^4 \eps_5 \kappabar$ & $(148,2)$ & 12 & \ref{lem:hidden_extensions_nu1_nu2} \\
    $\nu_6 \eps = \tau^4 \eta_5\kappa \kappabar$ & $(155,3)$ & 10 & \ref{lem:hidden_extensions_nu1_nu2} \\

    $\nu_2\cdot 2\nu = \tau^8 \kappabar^2 \widetilde{d}$ & $(54,2)$ & 10 & \ref{prop:nu2_times_2_nu_hidden_extension}\\
    $\nu_6\cdot 2\nu = \tau^8 \kappabar^2 \widetilde{d}_4$ & $(150,2)$ & 10 & \ref{prop:nu2_times_2_nu_hidden_extension}\\
    \bottomrule
\end{tabular}
\end{table}


\begin{table}[H]
    \centering
    \caption{Total differentials on $\Smf_{(2)}/\tau$.}
    \label{tab:2_total_diffs}
    \vspace{1em}
\begin{tabular}{@{}lcccc@{}}
    \toprule
    Source & Source degree & Target & Validity & Location\\
    \midrule
    $\Delta$ & $(24,0)$ & $\nu\kappabar$ & $(4,14)$ & \ref{lem:total_differential_14_Delta} \\
    $\Delta^2$ & $(48,0)$ & $\nu_1\kappabar$ & $(4,14)$ & \ref{prop:total_diff_14_Delta2} \\
    $\Delta^3$ & $(72,0)$ & $3\nu_2\kappabar$ & $(4,14)$ & \ref{prop:total_diff_14_Delta2} \\
    $\Delta^6$ & $(144,0)$ & $3\nu_5\kappabar$ & $(4,14)$ & \ref{prop:total_diff_14_Delta2} \\
    $\Delta^7$ & $(168,0)$ & $7\nu_6\kappabar$ & $(4,14)$ & \ref{prop:total_diff_14_Delta2} \\
    \bottomrule
\end{tabular}
\end{table}


\begin{table}[H]
    \centering
    \caption{Toda brackets in $\Smf_{(2)}/\tau^k$.}
    \label{tab:2_toda_brackets}
    \vspace{1em}
\begin{tabular}{@{}lccccc@{}}
    \toprule
    Name & Degree & Toda bracket & Indeterminacy & Validity & Location \\
    \midrule

    $\eta_1$ & $(25,1)$ & $\angbr{\kappabar,\tau^4\nu,\eta}$ & $\kappa$-torsion classes & 10 & \ref{lem:toda_brackets_etai}\\
    $\eta_5$ & $(121,1)$ & $\angbr{\kappabar,\tau^4\nu_4,\eta}$ & $\kappa$-torsion classes & 10 & \ref{lem:toda_brackets_etai}\\

    $\nu_1$ & $(27,1)$ & $\angbr{\kappabar,\tau^4\nu,2\nu}$ & & 12 & \ref{lem:toda_brackets_nu1_nu2}\\
    
    $\nu_2$ & $(51,1)$ & $\angbr{\kappabar,\tau^4\nu_1,\nu}$ & & 12 & \ref{lem:toda_brackets_nu1_nu2}\\
     & & $\angbr{\nu,\tau^4 2\nu,\tau^4\nu, \kappabar^2}$ & & 10 & \ref{prop:nu2_fourfold_carrick_bracket}\\
    
    $\nu_5$ & $(123,1)$ & $\angbr{\kappabar,\tau^4\nu,2\nu_4}$ & & 12 & \ref{lem:toda_brackets_nu1_nu2}\\
    $\nu_6$ & $(147,1)$ & $\angbr{\kappabar,\tau^4\nu_1,\nu_4}$ & & 12 & \ref{lem:toda_brackets_nu1_nu2}\\
     &  & $\angbr{\nu_4,\tau^4 2\nu,\tau^4\nu, \kappabar^2}$ & & 10 & \ref{prop:nu2_fourfold_carrick_bracket}\\

    $\eps_1$ & $(32,2)$ & $\angbr{\nu_1,\nu,\eta}$ & & 12 & \ref{not:lifts_eps1_eps_4_eps5}\\
    $\eps_4$ & $(104,2)$ & $\angbr{\nu,2\nu_4,\eta}$ & & 12 & \ref{not:lifts_eps1_eps_4_eps5}\\
    $\eps_4$ & $(128,2)$ & $\angbr{\nu_1,\nu_4,\eta}$ & & 12 & \ref{not:lifts_eps1_eps_4_eps5}\\
    
    $\eta_1\kappa$ & $(39,3)$ & $\angbr{\nu_1, \nu, \eps}$ & & 10 & \ref{lem:toda_brackets_nu1_nu2}\\
    $\eta_4\kappa$ & $(111,3)$ & $\angbr{\nu, 2\nu_4, \eps}$ & & 12 & \ref{lem:toda_brackets_nu1_nu2}\\
    $\eta_5\kappa$ & $(135,3)$ & $\angbr{\nu_1, \nu_4, \eps}$ & & 10 & \ref{lem:toda_brackets_nu1_nu2}\\
    $\tau^2\eta_4\kappabar$ & $(117,3)$ & $\angbr{\nu, 2\nu_4, \kappa}$ & & 12 & \ref{lem:toda_brackets_nu1_nu2}\\
    
    $\pm 2\kappabar$ & $(20,4)$ & $\angbr{\nu, \eta, \eta\kappa}$ & not discussed & 10 & \ref{lem:toda_brackets_nu1_nu2}\\
    $\tau^2\kappabar$ & $(20,2)$ & $\angbr{\kappa,2,\eta,\nu}$ & not discussed & $\infty$ & \ref{cor:todabracketforkappabar}\\
    $\pm \kappabar_4$ & $(116,4)$ & $\angbr{\nu_4,\eta,\eta\kappa}$ & not discussed & 10 & \ref{lem:toda_brackets_etai}\\
    \bottomrule
\end{tabular}
\end{table}


\begin{landscape}

\subsection{Descent spectral sequence charts}\label{ssec:ss}

Here we display the descent spectral sequences for $\Tmf_{(3)}$ and $\Tmf_{(2)}$. We use the following conventions.

\begin{itemize}
    \item Black arrows are differentials. Black lines are multiplication by either $\al$ or $\be$ for $\Tmf_{(3)}$ and either $\eta$ or $\nu$ for $\Tmf_{(2)}$. Under the blue line is the \emph{connective region} (\cref{def:regions}) and right of the orange line is the \emph{$S$-region} (\cref{not:S_region}).
    \item Red lines indicate hidden extensions by $\al$ for $\Tmf_{(3)}$ and by $\eta, \nu, \eps$, or $\kappa$ for $\Tmf_{(2)}$. We only include those hidden extensions needed in this article.
    \item Hollow squares refer to $\Z_{(p)}$ and hollow circles to $\F_p$. For $\Tmf_{(2)}$, two enclosed circles represents $\Z/4$ and three enclosed circles represents $\Z/8$.
    Two symbols in the same bidegree represents their sum.
\end{itemize}

These charts correct some small oversights in those of \cite{konter_Tmf}. We would like to mention in particular the key $d_{23}$-differential of \cref{prop:diff23} is missing from Konter's chart of the DSS for $\Tmf_{(2)}$.

\begin{figure}[ht]
    \centering
    \makebox[\textwidth]{\includegraphics[trim={2.5cm 11cm 2.5cm 3cm},clip,page = 5, scale = 1]{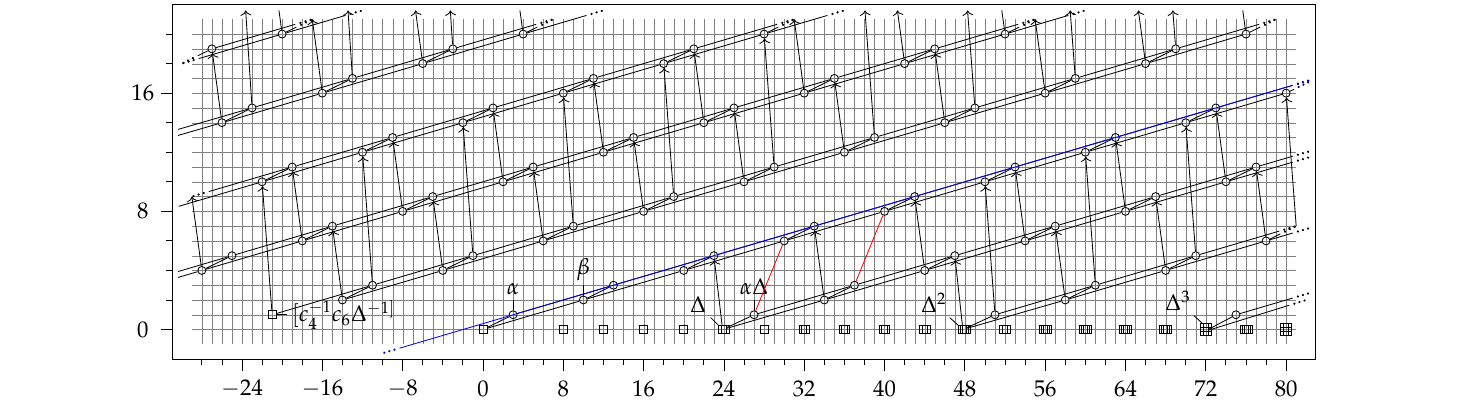}}
    \caption{The DSS for $\Tmf_{(3)}$. The lines represent multiplication by either $\al$ or $\beta$.}
    \label{ssprime3}
\end{figure}







\begin{figure}[ht]
    \centering
    \makebox[\textwidth]{\includegraphics[trim={2.5cm 10cm 2.5cm 4cm},clip,page = 1, scale = 0.75]{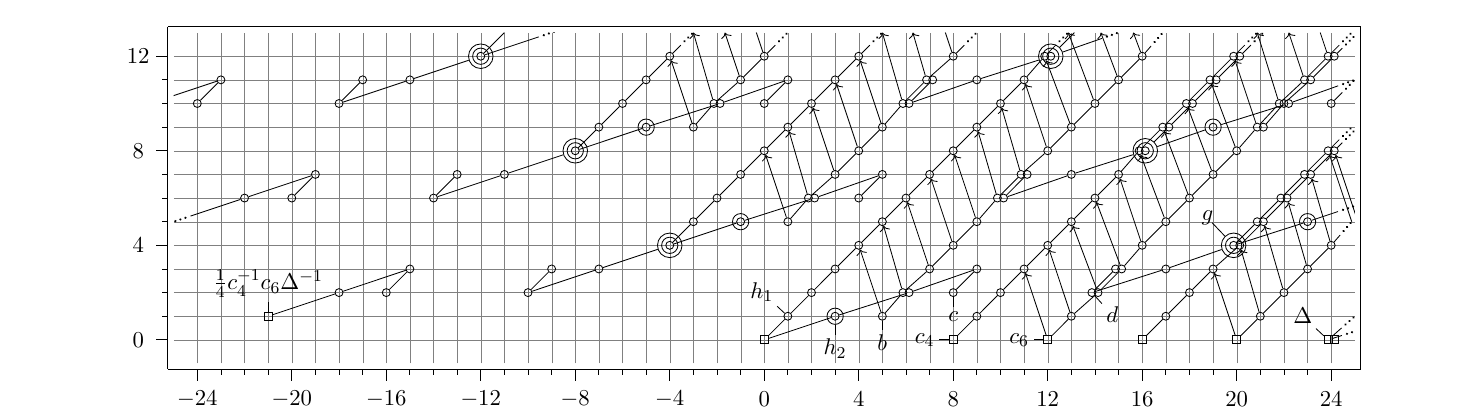}}
    \caption{$\uE_3$-page of the DSS for $\Tmf_{(2)}$.}
    \label{etwoprime2}
\end{figure}

\begin{figure}[ht]
    \centering
    \makebox[\textwidth]{\includegraphics[trim={2.5cm 5cm 2.5cm 4cm},clip,page = 2, scale = 0.9]{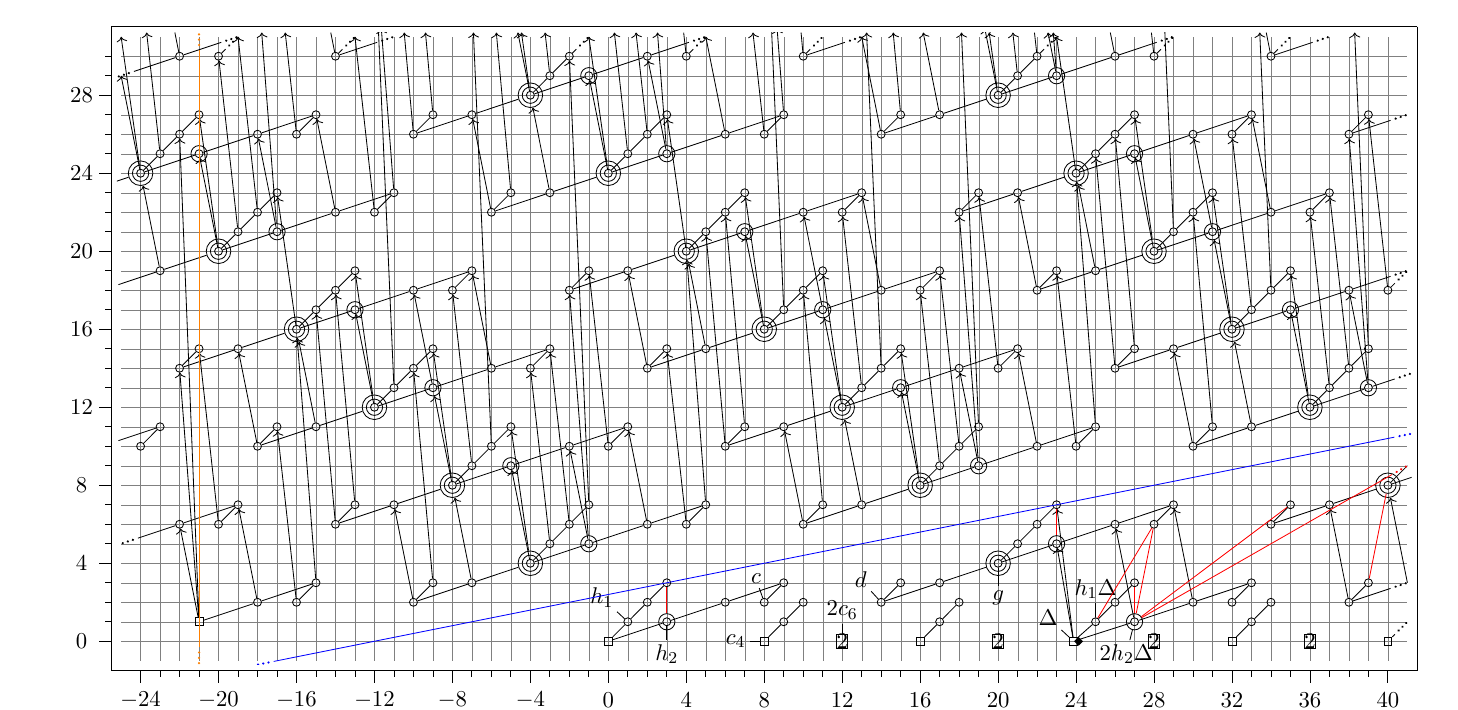}}
    \caption{The DSS for $\Tmf_{(2)}$ from the $\uE_5$-page in the range $-24\leq n \leq 40$.}
    \label{efiveprime2_part1}
\end{figure}

\begin{figure}[ht]
    \centering
    \makebox[\textwidth]{\includegraphics[trim={2.5cm 5cm 2.5cm 4cm},clip,page = 3, scale = 0.9]{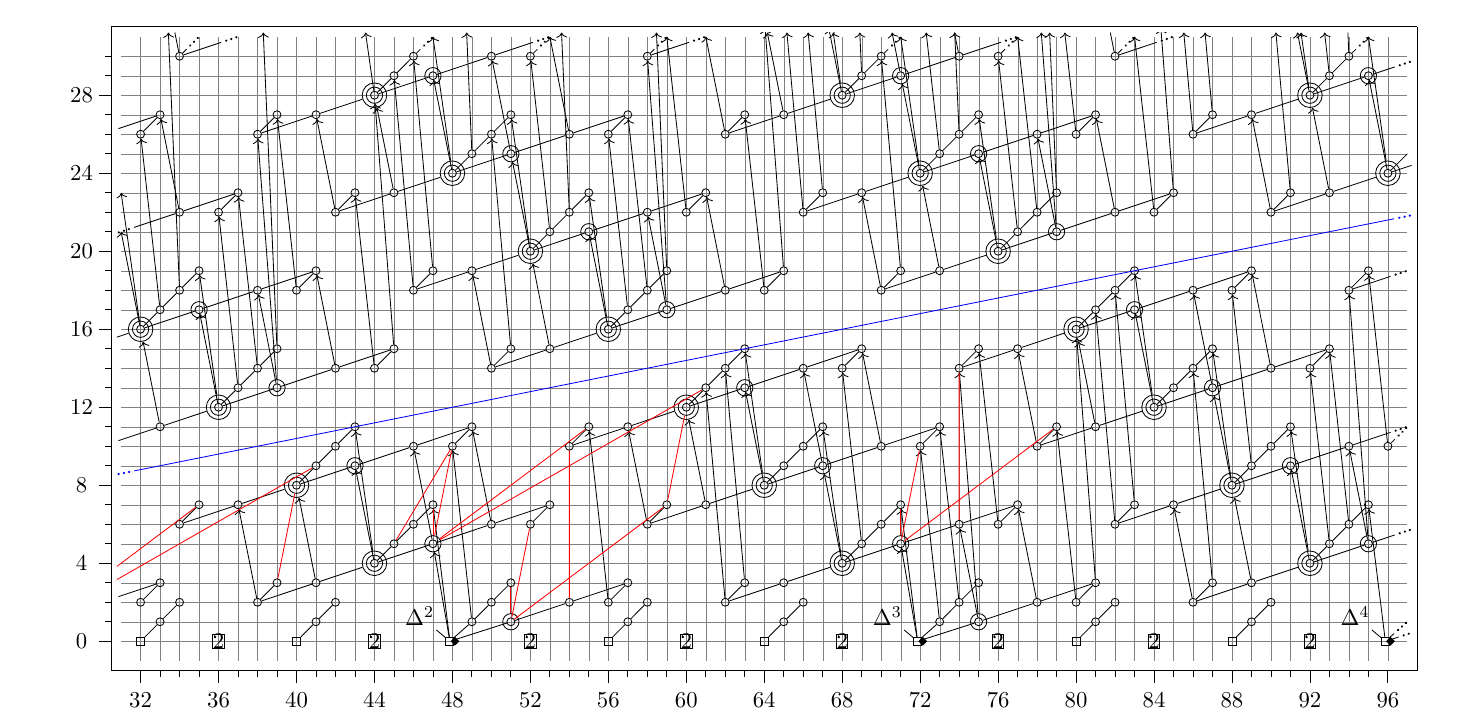}}
    \caption{The DSS for $\Tmf_{(2)}$ from the $\uE_5$-page in the range $32\leq n \leq 96$.}
    \label{efiveprime2_part2}
\end{figure}

\begin{figure}[ht]
    \centering
    \makebox[\textwidth]{\includegraphics[trim={2.5cm 5cm 2.5cm 4cm},clip,page = 4, scale = 0.9]{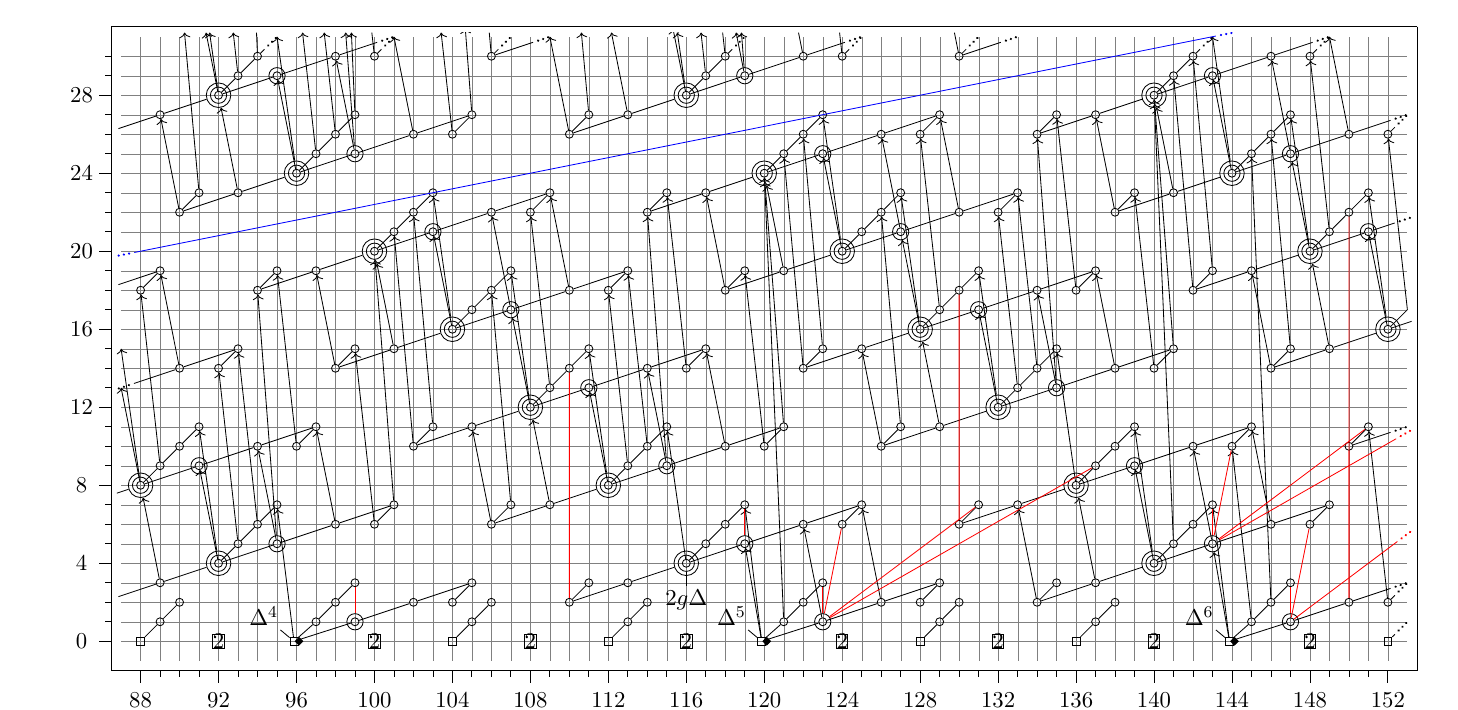}}
    \caption{The DSS for $\Tmf_{(2)}$ from the $\uE_5$-page in the range $88\leq n \leq 152$.}
    \label{efiveprime2_part3}
\end{figure}

\begin{figure}[ht]
    \centering
    \makebox[\textwidth]{\includegraphics[trim={2.5cm 5cm 2.5cm 4cm},clip,page = 5, scale = 0.9]{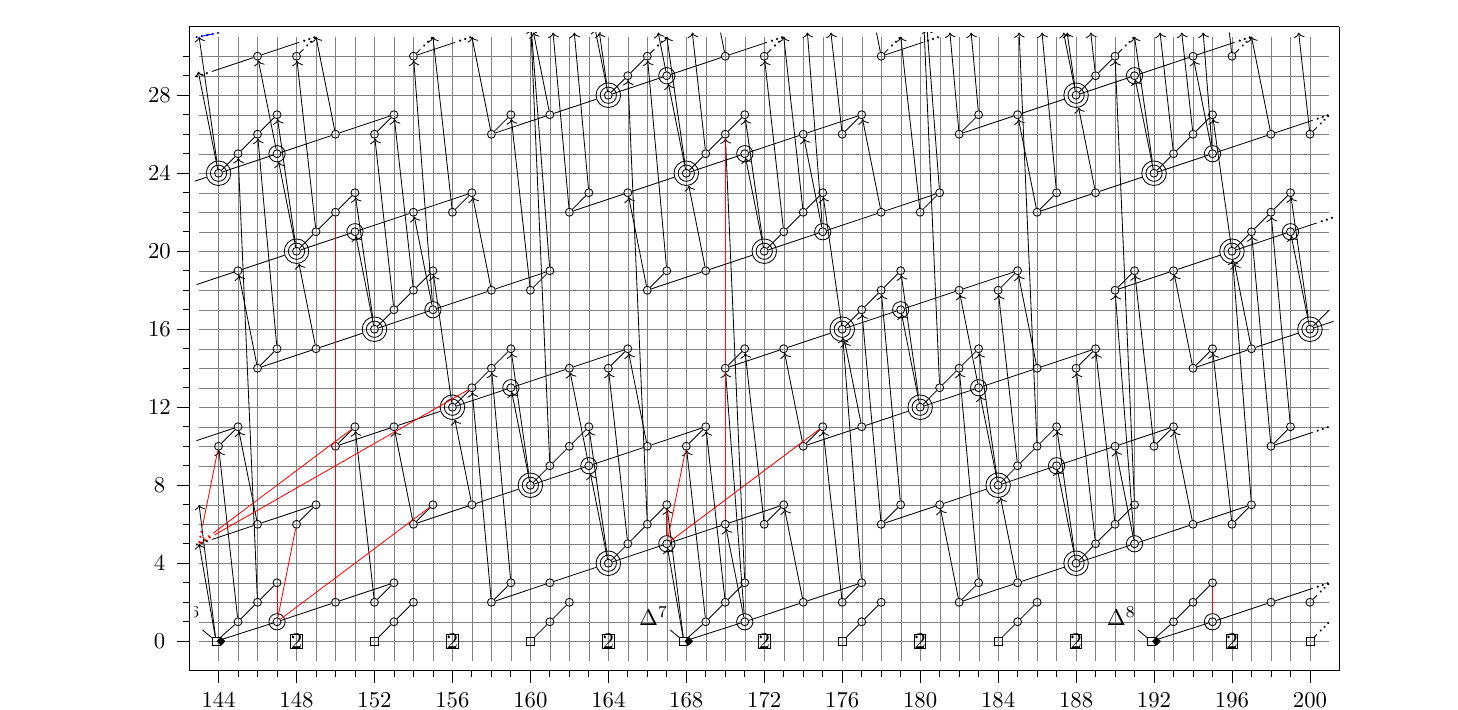}}
    \caption{The DSS for $\Tmf_{(2)}$ from the $\uE_5$-page in the range $144\leq n \leq 200$.}
    \label{efiveprime2_part4}
\end{figure}





\end{landscape}

%
%
%
%

\printbibliography[heading=bibintoc]

\end{document}